\documentclass{article}
\usepackage[utf8]{inputenc}
\usepackage{amssymb,amsmath}
\usepackage{amsmath}
\usepackage{amssymb}
\usepackage{amsthm}
\usepackage{tikz-cd}
\usepackage{float}
\usepackage{todonotes}
\usepackage{algorithm,algorithmic}
\usepackage{xcolor}
\usepackage{comment}
\usepackage{enumitem}
\usepackage{physics}
\usepackage{url}
\usepackage{rotating,lscape}

\usepackage{amsfonts}
\usepackage{amsmath}
\usepackage{amssymb}
\usepackage{graphicx}
\usepackage{epstopdf}
\usepackage{algorithmic,float}
\ifpdf
  \DeclareGraphicsExtensions{.eps,.pdf,.png,.jpg}
\else
  \DeclareGraphicsExtensions{.eps}
\fi

\usepackage{enumitem}
\setlist[enumerate]{leftmargin=.5in}
\setlist[itemize]{leftmargin=.5in}

\usepackage{tikz-cd}
\usepackage{float}
\usepackage{todonotes}
\usepackage{algorithm,algorithmic}
\usepackage{xcolor}
\usepackage{comment}
\usepackage{enumitem}
\usepackage{physics}
\usepackage{url}
\usepackage{rotating,lscape}

\usepackage{colortbl}
\definecolor{Gray}{gray}{0.85}

\usepackage{multirow}
\usepackage{longtable}
\usepackage[top=2cm,bottom=2cm]{geometry}

\newtheorem{theorem}{Theorem}
\newtheorem{proposition}[theorem]{Proposition}

\newtheorem{lemma}[theorem]{Lemma}

\theoremstyle{definition}
\newtheorem{definition}[theorem]{Definition}
\newtheorem{example}[theorem]{Example}
\newtheorem{remark}[theorem]{\bf Remark}

\newcommand{\RR}{\mathbb{R}}
\newcommand{\NN}{\mathbb{N}}
\newcommand{\ZZ}{\mathbb{Z}}
\newcommand{\QQ}{\mathbb{Q}}

\newcommand{\vf}{\boldsymbol{f}}

\newcommand{\vx}{\boldsymbol{x}}
\newcommand{\vk}{\boldsymbol{k}}
\newcommand{\vg}{\boldsymbol{g}}
\newcommand{\vs}{\boldsymbol{s}}

\newcommand{\valpha}{\boldsymbol{\alpha}}

\newcommand{\vxbar}{\boldsymbol{\bar{x}}}
\newcommand{\vkbar}{\boldsymbol{\bar{k}}}

\newcommand{\vv}{\boldsymbol{v}}

\newcommand{\vect}[1]{\boldsymbol{#1}}

\newcommand{\false}{\operatorname{false}}
\newcommand{\true}{\operatorname{true}}

\newcommand{\Ord}[1]{ {O}( #1 )}
\newcommand{\DP}[2]{ \ensuremath{ \frac{\partial #1 }{\partial #2 } } }
\newcommand{\D}[2]{ \ensuremath{ \frac{d #1 }{d #2 } } }

\newcommand{\transpose}[1]{\ensuremath{ #1^\mathsf{T}}}

\newcommand{\ConsLaws}{\mathcal{C}(F)}
\newcommand{\ParPolConsLaws}{\mathcal{S}} 
\newcommand{\Ideal}{\operatorname{Ideal}}
\newcommand{\Syz}{\operatorname{Syz}}
\newcommand{\Syzgrad}{\operatorname{Syz}_{\grad}}

\newcommand{\myrank}[1]{r_{#1}}

\usepackage{hyperref}

\usepackage{tikz-cd}
\usepackage{float}
\usepackage{todonotes}
\usepackage{algorithm,algorithmic}
\usepackage{xcolor}
\usepackage{comment}
\usepackage{enumitem}
\usepackage{physics}
\usepackage{url}
\usepackage{rotating,lscape}

\newcommand{\cor}[1]{{\color{black} #1}}

\title{A Computational Approach to Polynomial Conservation Laws}

\author{Aurélien Desoeuvres$^{1}$, Alexandru Iosif$^7$, Christoph L\"uders$^6$, \\
  Ovidiu Radulescu$^{1,*}$, Hamid Rahkooy$^8$, Matthias Sei\ss$^2$, Thomas Sturm$^{3,4,5}$ }

 \date{
    $^1$ University of Montpellier and CNRS LPHI,  Montpellier, France\\
    $^2$ University of Kassel, Kassel, Germany \\
    $^3$ CNRS, INRIA, and the University of Lorraine, Nancy, France\\
    $^4$ Max Planck Institute for Informatics, Saarbr\"ucken, Germany  \\
    $^5$ Saarland University, Saarbrücken, Germany \\
    $^6$ University of Bonn, Bonn, Germany \\
    $^7$ Rey Juan Carlos University, Madrid, Spain\\
    $^8$ University of Oxford, Oxford, United Kingdom.\\
    $^*$ corresponding author \href{mailto:ovidiu.radulescu@umontpellier.fr}{ovidiu.radulescu@umontpellier.fr} \\
 \bf \today
 }

\begin{document}

\maketitle

\begin{abstract}
For polynomial ODE models,
we introduce and discuss the concepts of
  exact and approximate conservation laws, which are the first integrals of
  the full and truncated sets of ODEs. For fast-slow \cor{systems}, truncated
  ODEs describe the fast dynamics. We define compatibility classes as
  subsets of the state space, obtained by equating the
  conservation laws to constants.  A set of conservation laws is
  complete when the corresponding compatibility classes contain 
  a finite number of steady states.  Complete sets of
  conservation laws can be used for \cor{model order reduction and 
  for studying the multistationarity of the model}.
  We provide algorithmic methods for computing linear, monomial, and
  polynomial conservation laws of 
  \cor{polynomial ODE models}
  and for testing their
  completeness. The resulting conservation laws and their completeness
  are either independent or dependent on the parameters. In the
  latter case, we provide parametric case distinctions.  In
  particular, we propose a new method to compute polynomial
  conservation laws by comprehensive Gröbner systems and syzygies.

\end{abstract}

\section{Introduction} 
 
\cor{Several important models in physics, chemistry, biology, 
 ecology, economics, and engineering can be written as systems of polynomial ODEs.} 
{\bf Conservation laws} are first integrals \cor{of such models}, i.e.~quantities that are constant on any solution of the ODEs. When they exist, they can be used for \cor{model order}
reduction \cite{boulier2011model}. 
A reduced model can be more easily simulated and analysed. For example, 
ODEs with $n$ variables and $n-1$ independent conservation laws are solvable by quadratures \cite{arnold1978mathematical}. 

\cor{The use of first integrals for eliminating variables is an exact model reduction method \cite{boulier2011model}. However, many model  reduction methods are approximate. 
For instance, the quasi-steady state approximation is an approximate model 
reduction method based on the elimination of fast variables } \cite{tikh,fenichel1979geometric,gwz3,radulescu2012reduction,kruff2020algorithmic,desoeuvres2022a}. 
\cor{This reduction method fails when the fast dynamics has first integrals. }
We introduce here the {\bf approximate conservation laws,}  a key concept 
for model reduction by variable pooling \cor{used when the fast dynamics has first integrals}
\cite{schneider2000model,auger2008aggregation,desoeuvres2022a}.

Conservation laws are 
also used for \cor{studying} the steady states of \cor{ODE models}, \cor{for instance in applications 
that deal with} multistationarity \cite{feliu2012preclusion}. In the case 
of linear conservation laws, the set of steady states can be partitioned 
into stoichiometric compatibility classes
\cite{feinberg1974dynamics}
that are characterized by fixed, different values of the linear conservation
laws.  Here we extend this definition \cor{to polynomial} conservation laws. 
Compatibility classes are invariant sets of \cor{the ODE flow} and, under some
conditions that we call {\bf completeness}, contain a finite number of steady states. 
\cor{The completeness of approximate conservation laws is also important for model reduction
applications \cite{desoeuvres2022a}.}

A lot of effort has been put into computing linear conservation laws of \cor{ODE models} \cite{schuster1991determining,soliman2012invariants}. 
Non-linear conservation laws 
have been addressed using Darboux polynomials \cite{prelle1983elementary,man1993computing,cheze2011computation,mahdi2017conservation}. 
\cor{Monomial first integrals were discussed in~\cite{goldman1987integrals,mahdi2017conservation}. For a thorough introduction to first integrals and related computational methods, refer to Zhang's book~\cite{zhang2017integrability}. 
}
Polynomial first integrals were used for formal proofs 
of hybrid systems safety in \cite{sogokon2021pegasus},
where they were computed using the method of undetermined coefficients. 

In this paper, we propose a new method for computing polynomial conservation laws
using syzygies,  which are relations among the polynomial r.h.s of 
the  ODEs. 
\cor{
Furthermore, we use Comprehensive Gr\"obner Bases techniques for partitioning the parameter space into branches such that the syzygies (and hence, conservation laws) are invariant under evaluation in each branch.
}
We also propose algorithmic solutions for testing the independence 
and the completeness of the computed conservation laws. 

\cor{The structure of this paper is as follows. In the next section, we introduce the polynomial ODE models. 
We show that any polynomial ODE model can be written as a chemical reaction network (CRN) 
model. In Section 3 we discuss exact and approximate conservation laws, that are important
in model reduction applications. In Section 4 we provide methods to test the independence 
and completeness of the conservation laws. Section 5 exposes extant methods for computing
linear conservation laws. The Section 6.1 is dedicated to monomial conservation laws. 
In Section 6.2 we present a new method to compute polynomial conservation laws using syzygies. }

\section{Models}
In this paper, we consider \cor{polynomial ODE models
\begin{equation} \label{eq:fi}
\dot x_1 = f_1(\vk,\vx), \, \dots , \, \dot x_n = f_n(\vk,\vx),
\end{equation}
whose r.h.s. are integer coefficient polynomials in the variables $\vx$ and parameters $\vk$
$f_i(\vk,\vx)  \in \ZZ[\vk,\vx]=\ZZ[k_1,\dots,k_r,x_1,\dots,x_n]$.
Up to a redefinition of parameters, the functions 
$f_i$ can be considered homogeneous of degree one in 
$\vk$.
We denote the vector of right hand sides of  \eqref{eq:fi} by 
$$\vect{F}(\vk,\vx) = \transpose{(f_1(\vk,\vx),f_2(\vk,\vx),\ldots,
f_n(\vk,\vx))}.$$ 

For certain applications we consider that
the functions $f_i$ have the particular structure of chemical reaction networks (CRNs) \cite{aris1963independence}:
\begin{equation} \label{eq:crn}
f_i(\vk,\vx) = \sum_{j=1}^r S_{ij} k_j   
\vx^{\vect{\alpha}_j}.
\end{equation}
The monomials $\vx^{\vect{\alpha}_j}= x_1^{\alpha_{j1}}\cdots x_n^{\alpha_{j_n}}$ appearing in the 
right hand sides of \eqref{eq:crn} are defined 
by $r$ multi-indices $\vect{\alpha}_j=(\alpha_{j1},\ldots,\alpha_{jn}) \in \NN^n$ and  
for each monomial the parameter $k_j$ represents a rate constant. The parameters $\vk=(k_1,\dots,k_r)$ take values in $\RR_+^*$ and 
the integer coefficients $S_{ij}$ form a matrix $\vect{S}=(S_{ij}) \in \ZZ^{n\times r}$ which is called
the stoichiometric matrix. 

Mass action chemical reaction networks with $n$ species $A_1,\ldots,A_n$ 
and concentrations $\vx = (x_1,\ldots,x_n)$ satisfy \eqref{eq:crn}.
For a mass action reaction 
$$\bar{\alpha}_{j1} A_1 +  \ldots + \bar{\alpha}_{jn} A_n \rightarrow 
\beta_{j1} A_1 +  \ldots + \beta_{jn} A_n$$ 
we have $S_{ij} = \beta_{ji} - \bar{\alpha}_{ji}$ and  
the reaction rate is  $k_j \vx^{\vect{\bar{\alpha}}_j}$, i.e.~the stoichiometric coefficients $\vect{\bar{\alpha}}_{j}$ and the 
multi-indices $\vect{\alpha}_{j}$ coincide. 
This constraint leads to algebraic properties of CRNs exploited in {\bf chemical reaction network theory} (CRNT), which has been
initiated by Horn, Jackson and Feinberg \cite{feinberg1974dynamics,Feinberg:19a}. 
However, this may not be the case in general, meaning that the multi-indices $\vect{\alpha}_{j}$ and $\vect{\bar{\alpha}}_j$ in $S_{ij} = \beta_{ji} - \bar{\alpha}_{ji}$ and $k_j \vx^{\vect{\alpha}_j}$
 are not necessarily equal. 
As a matter of fact the CRNs with monomial rate functions used here do not reduce
the generality of polynomial ODE models described by \eqref{eq:fi}.
This is shown by the following theorem.

}


\begin{theorem}\label{theorem:crn}
Any system of polynomial ODEs 
$$\dot x_1 = f_1(\vk,\vx), \, \dots, \, \dot x_n = f_n(\vk,\vx)$$ 
where each $f_i(\vk,\vx) \in \ZZ[\vk,\vx]$ is an integer coefficient polynomial 
in $\vk$ and $\vx$, homogeneous of degree one in 
$\vk=(k_1,\dots,k_r)$, 
can be written as \eqref{eq:fi},\eqref{eq:crn} for appropriate matrix 
$\vect{S}$ and \cor{multi-indices $\vect{\alpha}_j$}. 
\end{theorem}

\begin{proof}
A constructive proof is given by Algorithm~\ref{alg:S}.
\end{proof}
\begin{algorithm}[H] 
\begin{algorithmic}[1]
    \caption{\label{alg:S}$\operatorname{Smatrix}$}
 \REQUIRE 
 An ODE system whose r.h.s.~are sums of rational monomials (meaning we also allow negative powers) in $\vk$ and $\vx$, homogeneous of degree one in $\vk$, with integer coefficients:
 $$ \dot x_i = f_i(\vk,\vx) = \sum_{j =1}^{r_i} z_{i j} k_{ij}  \vx^{\vect{\beta}_{ij} } $$ where $z_{i j} \in \ZZ$, $\vect{\beta}_{ij} \in \ZZ^n$,
 $k_{ij}\in \{k_1,\dots,k_r\}$, for 
 $1\leq i \leq n$, $1\leq j \leq r_i$.
 \smallskip
\ENSURE An integer coefficient matrix $\vect{S}$ \cor{ and a set of multi-indices $ \vect{\alpha}_1,\ldots, \vect{\alpha}_r \in \NN^n $ 
such that $f_i(\vk,\vx) = \sum_{j=1}^r S_{ij} k_j \vx^{\vect{\alpha}_j}$  for all $1\leq i \leq n$}.
\STATE{Compute a list of distinct monomials
$$ \{M_j= k_j  \vx^{\vect{\alpha}_j}, 1\leq j\leq r \} = \bigcup_{i=1}^n \bigcup_{j=1}^{r_i}  \{ k_{ij}  \vx^{\vect{\vect{\beta}}_{ij}}\}.$$ Two 
monomials $k_{ij}  \vx^{\vect{\beta}_{ij}}$, 
$k_{i'j'}  \vx^{\vect{\beta}_{i'j'}}$
are distinct if $\vect{\beta}_{ij} 
\neq \vect{\beta}_{i'j'}$ or 
$k_{ij}  \neq k_{i'j'} $.}
 \FOR{$i:=1$ \TO $n$}
 \FOR{$j:=1$ \TO $r$}
 \STATE{$S_{ij}=0$}
 \FOR{$l:=1$ \TO $r_i$}
 \IF{$k_{il} \vx^{\vect{\beta}_{il}} = M_j$}
\STATE{$S_{ij} =  S_{ij} + z_{i l}$ }
\ENDIF
\ENDFOR
\ENDFOR
\ENDFOR
\end{algorithmic}
\end{algorithm}


\begin{remark}\label{rem:remark2}
 In many applications, parameters are not symbolic but given by their numerical values, sometimes linearly dependent over the integers. In this case, the independent parameters and the stoichiometric matrix are hidden behind these numbers. We propose in Algorithm~\ref{alg:findz} 
 a way to still obtain descriptions of the form
 \eqref{eq:crn}, where the $k_{j}$ are linearly independent over the integers and the stoichiometric coefficients are clearly visible. 
 For example, the ODE system 
 $\dot x_1 = 1.5 x_1 x_2, \, \dot x_2  = x_2 + 3 x_1 x_2$
 is transformed into 
 $\dot x_1 = k_1 x_1 x_2, \, \dot x_2  = k_2 x_2 + 2 k_1 x_1 x_2$ with the stoichiometric matrix 
 $$S = \begin{pmatrix} 1 & 0 \\ 2 & 1 \end{pmatrix}.$$
\end{remark}

\begin{remark}\label{rem:remark3} 
In chemical and biochemical applications CRN variables have positive values and the ODE flow preserves the positive orthant. A sufficient condition for this property is
that inputs of Algorithm~\ref{alg:S} satisfy:
$$ \beta_{iji} > 0, \text{whenever } z_{ij} < 0,$$
where $\beta_{iji}$ is the $i^{\text{th}}$ coordinate of the vector $\vect{\beta_{ij}}$. 
This condition ensures that rates of reactions
consuming a species $x_i$ tend to zero whenever the
concentration of this species tends to zero. 
The corresponding condition on the system 
\eqref{eq:fi} reads
$$ \alpha_{ji} > 0, \text{whenever } S_{ij} < 0.$$
However, these conditions can be lifted in other applications.
\end{remark}

\begin{algorithm}[H] 
\begin{algorithmic}[1]
    \caption{\label{alg:findz}$\operatorname{CompareAndSplitCoefficients}$}

 \REQUIRE 
 An ODE system of the form:  $\dot{x_i}=\cor{ f_i(\vk,\vx)} =\sum_{j=1}^m B_{ij}\vx^{\vect{\valpha}_j}$, where $\valpha_j\in \ZZ^n$, $B_{ij}\in \RR $, $m$ is the number of distinct $\valpha_j$.
 \smallskip
\ENSURE 
\cor{ A set of incommensurate real numbers $\{k_1,\ldots,k_r \}$ and a set of integers $\{ z_{ij},  
 1\leq i \leq n, 1\leq j \leq r_i \}$ such that
$$ f_i(\vk,\vx)= \sum_{j =1}^{r_i} z_{i j} k_{ij}  \vx^{\cor{\vect{\beta}_{ij}}}, $$
where $k_{ij}\in \{k_1,\dots,k_r\}$, for 
 $1\leq i \leq n$, $1\leq j \leq r_i$.
   }
\FOR{$h:=1$ \TO $m$}
\FOR{$i:=1$ \TO $n$}
\FOR{$j:=1$ \TO $i$}

\IF{$B_{jh}=0$}
\STATE{$z_{jh}:=1$, $k_{jh}:=0$, $z_{ih}:=\text{sign}(B_{ih})$, $k_{ih}:=B_{ih}/z_{ih}$\\
$\vect{\beta}_{ih}:= \vect{\alpha}_h$, $\vect{\beta}_{jh}:= \vect{\alpha}_h$}

\ELSIF{$B_{ih}/B_{jh}\in \ZZ$}
\STATE{$z_{ih}:=\text{sign}(B_{ih})|B_{ih}/B_{jh}|$, $k_{ih}:=B_{ih}/z_{ih}$, $z_{jh}:=\text{sign}(B_{jh})$, $k_{jh}:=B_{jh}/z_{jh}$
\\ $\vect{\beta}_{ih}:= \vect{\alpha}_h$, $\vect{\beta}_{jh}:= \vect{\alpha}_h$}
\ELSIF{$B_{jh}/B_{ih}\in \ZZ$}
\STATE{$z_{jh}:=\text{sign}(B_{jh})|B_{jh}/B_{ih}|$, $k_{jh}:=B_{jh}/z_{jh}$, $z_{ih}:=\text{sign}(B_{ih})$, $k_{ih}:=B_{ih}/z_{ih}$
\\ $\vect{\beta}_{ih}:= \vect{\alpha}_h$, $\vect{\beta}_{jh}:= \vect{\alpha}_h$
}
\ELSE
\STATE{$z_{jh}:=\text{sign}(B_{jh})$, $k_{jh}:=B_{jh}/z_{jh}$, $z_{ih}:=\text{sign}(B_{ih})$, $k_{ih}:=B_{ih}/z_{ih}$
\\ $\vect{\beta}_{ih}:= \vect{\alpha}_h$, $\vect{\beta}_{jh}:= \vect{\alpha}_h$
}
\ENDIF
\ENDFOR
\ENDFOR
\ENDFOR
\FOR{$i:=1$ \TO $n$} 
\STATE{$r_i:=m$, $j=1$}
\WHILE{$j<r_i$}
\IF{$k_{ij}=0$}
\FOR{$h:=j$ \TO $r_i-1$}
\STATE{/*re-indexing*/ $z_{i j} k_{ij}  \vx^{\vect{\beta}_{ij}}:=z_{i j+1} k_{ij+1}  \vx^{\vect{\beta}_{ij+1}}$}
\ENDFOR
\STATE{$z_{i r_i} k_{ir_i}  \vx^{\vect{\beta}_{ir_i}}:=0$}
\STATE{$r_i:=r_i -1$}
\ENDIF
\STATE{j:=j+1}
\ENDWHILE
\ENDFOR

\end{algorithmic}
\end{algorithm}

\cor{
\begin{remark}
The advantage of the form \eqref{eq:crn} and the reason to use it
even for non-chemical applications is, as will be shown later in Sections~\ref{sec:linear}, and \ref{sec:monomial}, 
that this form facilitates the computation of linear and monomial conservation laws.
\end{remark}
}


\section{Exact Versus Approximate Conservation Laws} \label{sec:motivatingexample}
\label{sec:exdef}
We start this section with a motivating example.
\begin{example}\label{MM}
Let us consider the irreversible 
Michaelis-Menten mechanism, \cor{used as a model of} enzymatic reactions. We choose rate constants corresponding to the so-called quasi-equilibrium, as done by Michaelis and Menten \cite{michaelis1913kinetik}. 
The reaction network for this model is
$$
S + E \underset{k_2}{\overset{k_1}{\rightleftharpoons}} ES    \overset{k_3 \delta}{\longrightarrow} E + P,
$$
where $S$ is a substrate, $E$ is an enzyme, $ES$ is an enzyme-substrate complex and $k_1$,$k_2$,$k_3$ 
are rate constants of the same order of magnitude. Here $0 < \delta < 1$ is a small positive scaling parameter, indicating that the third rate constant  is
small. 

According to mass-action kinetics, the concentrations
$x_1=[S]$, $x_2=[SE]$ and $x_3=[E]$ satisfy the system of 
ODEs
\begin{subequations}
\label{eqsMM}
\begin{eqnarray}
\dot{x_1} &=& - k_1 x_1 x_3 + k_2 x_2 , \label{eqsMM_1} \\
\dot{x_2} &=& k_1 x_1 x_3 - k_2 x_2 - \delta k_3 x_2, \label{eqsMM_2}  \\
\dot{x_3} &=& -k_1 x_1 x_3 + k_2 x_2 + \delta k_3 x_2. \label{eqsMM_3}
\end{eqnarray}
\end{subequations}
We consider the truncated system of ODEs
\begin{subequations}
\label{eqsMMtruncated}
\begin{eqnarray}
\dot{x_1} &=& - k_1 x_1 x_3 + k_2 x_2 , \label{eqsMMtruncated_1} \\
\dot{x_2} &=& k_1 x_1 x_3 - k_2 x_2, \label{eqsMMtruncated_2} \\
\dot{x_3} &=& -k_1 x_1 x_3 + k_2 x_2 \label{eqsMMtruncated_3},
\end{eqnarray}
\end{subequations}
 which is obtained by setting $\delta = 0$ in \eqref{eqsMM}. 
If $k_i, 1\leq i \leq 3$ are 
of order $\Ord{\delta^0}$, then the truncated system \eqref{eqsMMtruncated} describes the
dynamics of the model on fast timescales of order $\Ord{\delta^0}$. 

The steady state of the fast dynamics is obtained by equating to zero the r.h.s.~of \eqref{eqsMMtruncated}. The resulting
condition is called  {\bf quasi-equilibrium (QE)},
because it means that the complex
formation rate $k_1 x_1 x_3$ is equal
to the complex dissociation rate $k_2 x_2$.
In other words, the reversible reaction 
$$ S + E \underset{k_2}{\overset{k_1}{\rightleftharpoons}} ES$$ is at 
equilibrium. The QE condition is reached only at the end of the
fast dynamics and is satisfied
with a precision of order $\Ord{\delta}$ during
the slow dynamics \cite{gorban2010asymptotology}. 

We introduce the linear combinations of 
variables $x_4 = x_1 + x_2$ and $x_5 = x_2 + x_3$, corresponding 
to the total substrate and total enzyme concentrations, respectively.
Addition of equation \eqref{eqsMM_2} and \eqref{eqsMM_3} leads to 
$\dot{x_5}=0$ which means that for solutions of the full system,  $x_5$ is constant for all times.
We will call such a quantity an {\bf exact conservation law}.  
The addition of equation \eqref{eqsMMtruncated_1} 
and \eqref{eqsMMtruncated_2} 
leads to $\dot{x_4}=0$.
This means that $x_4$ is constant for 
solutions of the truncated dynamics, valid at short times $t = \Ord{\delta^{0}}$.
\cor{However, addition of \eqref{eqsMM_1} and \eqref{eqsMM_2}
leads to 
$$\dot{x_4}= - \delta k_3 x_2,$$
meaning that $x_4$ evolves slowly under the full dynamics.  
$x_4$ is not constant at large times $t = \Ord{\delta^{-1}}$.}
We call such a quantity an {\bf approximate conservation law}. 
The quantity $x_5$ is both an exact and approximate conservation law, because is conserved by both \eqref{eqsMMtruncated} and \eqref{eqsMM} and therefore it 
 is constant at all times. 
\end{example}

\cor{
\begin{remark}\label{rem:reduction}
The exact and approximate conservation laws can be used as a new parameter
and as a slow variable of the model, respectively.  Two extant variables can be then eliminated as such
$x_2 = x_4 - x_1, x_3 = k_4 + x_1 - x_4$ where $k_4=x_5$ is the new parameter. 
 In the remaining variables $x_1,x_4$ the full system reads
 \begin{subequations}
\label{eqsSF}
\begin{eqnarray}
\dot{x_1} &=& - k_1 x_1 ( k_4 + x_1 - x_4)   + k_2 (x_4-x_1) , 
\label{eqsSF1} \\
\dot{x_4} &=& - \delta k_3 (x_4-x_1). \label{eqsSF2}
\end{eqnarray}
\end{subequations}
The transformed model \eqref{eqsSF} is typically a slow-fast system with $x_1$ the fast and $x_4$ the slow variable \cite{tikh,fenichel1979geometric} and can be further reduced using the quasi-steady state approximation.  
In the transformed model the fast dynamics \eqref{eqsSF1} has a non-degenerate, hyperbolic steady state. This was not the case for the fast
dynamics \eqref{eqsMMtruncated} in the original formulation of the problem, that has a degenerate steady state manifold. However, the foliations defined by hyperplanes on which the two conservation laws are constant intersect transversally the degenerate state state manifold, resulting into non-degenerate steady states. This property, called completeness, will be detailed in the  Section~\ref{sec:complete}. 
\end{remark}
}

More generally, let us assume that the r.h.s.~of the system in \eqref{eq:fi} can be decomposed into a sum 
of two terms, the first one dominating (being much larger than the second). This decomposition can be obtained
by rescaling the model variables and parameters by powers of a small scaling parameter $\delta$, where $0 < \delta < 1$.
To this aim, we define 
$\vkbar = (\bar k_1,\ldots,\bar k_r)$
and 
$\vxbar = (\bar x_1,\ldots,\bar x_m)$
such that
$k_i = \delta^{e_i} \bar k_i$, 
$x_i = \delta^{d_i} \bar x_i$ and obtain
\begin{equation}\label{eq:splitting}
f_i(\vk,\vx) = f_i^{(1)}(\vk,\vx) + f_i^{(2)}(\vk,\vx),
\end{equation}
where 
$$f_i^{(1)}(\vk,\vx) =\delta^{b_i} \bar f_i^{(1)}(\vkbar,\vxbar) \quad \mathrm{and} \quad 
f_i^{(2)}(\vk,\vx) =\delta^{b'_i} \bar f_i^{(2)}(\vkbar,\vxbar,\delta)$$ with
$b_i < b'_i$ (for details see \cite{kruff2020algorithmic,desoeuvres2022a}).



For a small parameter $\delta$,
the terms  $\delta^{b'_i} \bar f_i^{(2)}(\vkbar,\vxbar,\delta)$ are
dominated by the terms  $\delta^{b_i}\bar f_i^{(1)}(\vkbar,$ $\vxbar)$. This justifies
to introduce the {\bf truncated system} as the system of ODEs obtained
by keeping only the lowest order, dominant terms of the system in
\eqref{eq:fi}, namely
\begin{equation} \label{TruncatedSystem}
 \dot{x}_1 =  f_1^{(1)}(\vk,\vx), \, \ldots , \,
 \dot{x}_n =   f_n^{(1)}(\vk,\vx).
\end{equation}
We denote the vector of right hand sides of the truncated system by 
$$\vect{F}^{(1)}(\vk,\vx)=\transpose{(f_1^{(1)}(\vk ,\vx), \ldots,
f_n^{(1)}(\vk ,\vx ))}.$$

\begin{remark}
  Depending on the application, the truncated system in
  \eqref{TruncatedSystem} can consist of fewer ODEs than the system in
  \eqref{eq:fi}. For instance, in the case of model reduction of
  fast-slow \cor{systems}, only the ODEs corresponding to fast variables are
  included in the truncated system.  In this case, the truncated
  system describes the fast timescale dynamics.
\end{remark}

\begin{definition} 
  \label{def:exactandapproxcons}$\,$
\begin{enumerate}[label=(\alph*)]
\item A function $\phi(\vk,\vx)$ is a {\bf parametric exact
    conservation law} if there exists a non-empty semi-algebraic set
  $V_{\vk} \subseteq \RR^{r}_{+}$ such that for all $\vk \in V_{\vk}$,
  $\phi(\vk,\vx)$ is a first integral of the full system in
  \eqref{eq:fi}, i.e.~if
  \[
    \forall \vk\in V_{\vk} \quad \forall \vx \in \RR^{n}_{+}\quad
    (\sum_{i=1}^n \DP{\phi}{x_i}(\vk,\vx) f_i(\vk,\vx) = 0).
  \]
\item A function $\phi(\vk,\vx)$ is a {\bf parametric approximate
    conservation law} if there exists a non-empty semi-algebraic set
  $V_{\vk} \subseteq \RR^{r}_{+}$ such that for all $\vk \in V_{\vk}$,
  $\phi(\vk,\vx)$ is a first integral of the truncated system in
  \eqref{TruncatedSystem}, i.e.~if
  \[
    \forall \vk \in V_{\vk} \quad \forall \vx \in \RR^{n}_{+} \quad
    (\sum_{i=1}^n \DP{\phi}{x_i}(\vk,\vx) f_i^{(1)}(\vk,\vx) = 0).
  \]
\item A parametric exact (approximate) conservation law of the form
  \[
    \phi(\vk,\vx) =c_1(\vk) x_1 + \dots + c_n (\vk) x_n
  \]
  with $c_i(\vk) \in \RR[\vk]$ is called an {\bf exact (approximate)
    linear conservation law}. If $c_i \geq 0$ for $1\leq i \leq n$,
  the linear conservation law is called {\bf semi-positive}.
\item A parametric exact (approximate) conservation law of the form
  $\phi(\vk,\vx) = a(\vk) x_1^{m_1}$ $ \cdots x_n^{m_n} $ with
  $m_i \in \ZZ$ and $a(\vk) \in \RR[\vk]$, is called an {\bf exact
    (approximate) rational monomial conservation law}.  If
  $m_i \in \ZZ_+$ for $1\leq i \leq n$, the conservation law is called
  {\bf monomial}. For simplicity, in this paper, we will call both
  types {\bf monomial}.
\item A parametric exact (approximate) conservation law of the form
  \[
    \sum_{i=1}^s a_i(\vk) x_1^{m_{1i}} \cdots x_n^{m_{ni}}
  \]
  with $m_{ji} \in \ZZ_+$ and $a_i(\vk) \in \RR[\vk]$ is called a
  parametric {\bf exact (approximate) polynomial conservation law}.
\end{enumerate}
\end{definition}


\begin{remark}\label{rmk:parametricvsuncond}
  In some applications we are interested in conservation laws
  $\phi(\vx)$ that do not depend on $\vk$. For instance, in model reduction of a fast-slow 
  \cor{systems}, approximate conservation laws that are not exact
represent variables of the reduced model \cite{schneider2000model,auger2008aggregation,gorban2010asymptotology,radulescu2012reduction,desoeuvres2022a}. 
When these conservation
  laws  depend only on $\vx$,
  one can show that
  these approximate conservation laws are always slower than the
  variables which compose them, a prerequisite for fast-slow reduction
  \cite{desoeuvres2022a}.
  Moreover \cor{this ensures, for CRN models, that} the reduced model 
  is obtained by graph rewriting operations, i.e~by \cor{species and reactions} pooling and pruning (see \cite{radulescu2012reduction}).
\end{remark}

\begin{remark}
  Following Remark \ref{rmk:parametricvsuncond}, in this paper, we use
  the following convention. An \textbf{unconditional conservation law}
  is a function $\phi(\vx)$ that does not depend on $\vk$ and satisfies
  the conditions of Definition \ref{def:exactandapproxcons}, while a
  \textbf{parametric conservation law} is a function $\phi(\vk,\vx)$
  that does depend on $\vk$. 
\end{remark}

\section{Independence and Completeness}\label{sec:complete}

Some variables $x_i$ may not appear in a conservation law 
$\phi(\vk,\vx)$. This means, in the case of a linear conservation law, that some coefficients $c_i$ may be zero or, in the case of a nonlinear conservation law, that some partial derivatives $\DP{\phi}{x_i}(\vk,\vx)$ may vanish. If $r$ is the number of all non-zero quantities $\DP{\phi}{x_i}(\vk,\vx)$, then we say that the conservation law {\bf depends on $r$ variables}.

\begin{definition}
\label{def:irreducibleConsLaw}
An exact or approximate conservation law depending on 
$r$ variables
is called {\bf simple} 
if it can't be split 
into the sum or the product of two conservation laws 
such that at least one of them depends on 
a number of variables $r'$ with $1 \leq r' < r$.
\end{definition}

\begin{definition}
\label{def:degenerateSteadyStates}
For $\vk \in \RR^r_+$  a steady state $\vx$ is a positive solution of $\vect{F}(\vk,\vx)=0$.
We denote 
the {\bf steady state variety} by $\mathcal{S}_{\vk}$. A steady state is 
{\bf degenerate} or {\bf non-degenerate}
if the Jacobian $D_{\vx} \vect{F}(\vk,\vx)$ is singular or regular, respectively. 
\end{definition}

Degeneracy of steady states implies that
$\mathcal{S}_{\vk}$ is not discrete. Reciprocally, we have the following result.
\begin{theorem}
Assume that for  $\vk \in \RR^r_+$, $\mathcal{S}_{\vk}$ is a manifold.
  If the local dimension at a point $\vx_0 \in \mathcal{S}_{\vk}$ is strictly positive, then $\vx_0$ is degenerate.  
\end{theorem}

\begin{proof}
There is a neighborhood $U \subset \RR$
of zero and a smooth function $U \rightarrow \mathcal{S}_{\vk}$, $\alpha \mapsto \vx(\alpha)$ with $\vx(0)=\vx_0$.
By differentiating $\vect{F}(\vk,\vx(\alpha))=0$ with respect to $\alpha$ at $0$ we get $D_{\vx} \vect{F}(\vk,\vx_0) \D{\vx(0)}{\alpha}=0$, and thus 
$D_{\vx} \vect{F}(\vk,\vx_0)$ is singular. 
\end{proof}

\begin{definition}
\label{def:completeConsLaw}
A set of $s$ exact
conservation laws
defined by the
vector $$\vect{\Phi}(\vk,\vx)=\transpose{(\phi_1(\vk,\vx),\ldots,\phi_s(\vk,\vx))}$$  is called {\bf complete} if
the Jacobian matrix 
$$\vect{J}_{\vect{F},\vect{\phi}}(\vk,\vx)=D_{\vx}\transpose{(\vect{F}(\vk,\vx), \vect{\Phi}(\vk,\vx) )}$$ 
has rank $n$ for any $\vk \in \RR_+^r,\vx\in \RR_+^n$. 
The set is called {\bf independent} if the Jacobian matrix of  
$ \transpose{\vect{\Phi}(\vk,\vx)}$ with respect to $\vx$
has rank $s$ for any $\vk \in \RR_+^r,\vx\in \RR_+^n$.  
In the case of a
set of approximate conservation laws defined by a vector
$\vect{\Phi}(\vk,\vx)$, completeness is defined with $\vect{F}(\vk,\vx )$
replaced by $\vect{F}^{(1)}(\vk,\vx)$.
\end{definition}

\begin{proposition}
If a system has a complete set of conservation laws, 
then the intersection of $\mathcal{S}_{\vk}$ with 
$\{\vect{x} \mid \vect{\Phi}(\vk,\vx)=\vect{c}_0 \} \cap \RR_+^n$, where $\vect{c}_0\in \RR^r$, is finite. 
\end{proposition}
\begin{proof}
Suppose that the intersection is non-empty and contains $\vx$.
Because the rank of $\vect{J}_{\vect{F},\vect{\phi}}(\vk,\vx)$ in $\vx$ is $n$, the implicit function theorem implies
that, for $\vk$ and $\vect{c}_0$, $\vx$ is isolated from
other solutions of $\vect{F}(\vk,\vx)=0$, $\vect{\Phi}(\vk,\vx)=\vect{c}_0$. As
$\vect{F}(\vk,\vx)$ is polynomial, the intersection is finite.
\end{proof}

\begin{remark}
When $\vect{\Phi}(\vk,\vx)$
is linear in $\vx$ and results from a stoichiometric matrix,
the set $$\mathcal{C}_{\vect{c}_0} = \{ \vx \mid \vect{\Phi}(\vk,\vx)=\vect{c}_0 \} \cap \RR_+^n$$ is called {\bf stoichiometric compatibility class} or {\bf reaction simplex} \cite{wei1962structure,feinberg1974dynamics}. 
Stoichiometric compatibility classes of systems with complete sets of linear conservation laws contain a finite number of steady states. 
We note that some authors call completeness of linear conservation laws  {\bf non-degeneracy} \cite{feliu2012preclusion}. 
Here, we extend these definitions
and call {\bf compatibility classes}
all sets of the type $\mathcal{C}_{\vect{c}_0}$, even if 
$\vect{\Phi}(\vk,\vx)$  is a non-linear conservation law. 
Compatibility classes of polynomial conservation laws
are semi-algebraic sets. 

\cor{Compatibility classes and the completeness of conservation laws have applications to the study of multistationarity of models \cite{feliu2012preclusion} 
and to model order reduction \cite{desoeuvres2022a} (see also Section~\ref{sec:motivatingexample} and Remark~\ref{rem:reduction}). 
}
\end{remark}

\begin{remark}
Since our concern is the number of
positive solutions in $\vx$  of $\vect{F}(\vk,\vx)=0$, $\vect{\Phi}(\vk,\vx)=\vect{c}_0$, it would be
more natural to consider in Definition~\ref{def:completeConsLaw} the rank of  
$\vect{J}_{\vect{F},\vect{\phi}}(\vk,\vx)$
on $$\mathcal{S}_{\vk} \cap \{ \vx \mid  \vect{\Phi}(\vk,\vx)=\vect{c}_0 \} \cap \RR_+^n.$$
In fact, as this rank does not depend on $\vect{c}_0$, it is simpler and equivalent to impose its value
on $\mathcal{S}_{\vk} \cap \RR_+^n$.
\end{remark}

\begin{remark}
The  independent linear conservation laws 
$$\vect{\Phi}(\vk,\vx) =(x_1+x_2,x_2+x_3) $$  of \cor{the truncated system in the}
Example~\ref{MM} are complete. More precisely, the Jacobian of $$\transpose{(\vect{F}(\vk,\vx),\vect{\Phi}(\vk,\vx))},$$ where $\vect{F}(\vk,\vx)$ is the vector of right hand sides of \eqref{eqsMMtruncated},
has a $3 \times 3$ minor 
$$M=\mathrm{det}(D_{\vx} \transpose{(-k_1x_1x_3+k_2x_2,\vect{\Phi}(\vk,\vx))}) = - k_2 - k_1 x_1 - k_1 x_3.$$ 
This minor can not be zero for positive $\vx$, $\vk$ on the steady state variety which is defined by the single equation
$-k_1 x_1 x_3 + k_2 x_2 =0$. Therefore the rank 
of $\vect{J}_{\vect{F},\vect{\phi}}(\vx,\vk)$ is $3$. 

Furthermore, all stoichiometric compatibility classes
defined by $x_1+x_2=c_{01}$, $x_2+x_3=c_{02}$, $\vx > 0$
contain a unique \cor{positive} steady state 
\begin{eqnarray*}
x_1 &=&  \big(k_1(c_{01} - c_{02}) - k_2 + \sqrt{\Delta} \big)/(2 k_1), \\
x_2 &=&  \big(k_1(c_{01} + c_{02}) + k_2 - \sqrt{\Delta}\big)/(2 k_1), \\
x_3 &=&  \big(k_1(c_{02} - c_{01}) - k_2 + \sqrt{\Delta}\big)/(2 k_1),
\end{eqnarray*}
where
$\Delta = (c_{01}-c_{02})^2k_1^2+k_2^2 + 2 k_1k_2(c_{01}+c_{02})$.
\cor{This non-degenerate steady state is precisely the one used for the quasi-steady reduction of the transformed model (see also the Remark~\ref{rem:reduction}).}
\end{remark}

The notions of completeness and independence in
Definition~\ref{def:completeConsLaw} are effective. Indeed, Algorithms~\ref{alg:iscomplete} and~\ref{alg:isindependent} below check for completeness and independence, respectively. Both algorithms use a parametric rank computation (see in both cases line 1 with the call "ParametricRank"), which we are going to present first.

Let $P$, $N$, $M \in \NN$ and let $\RR[\vv]$ be a polynomial ring in 
$P$ indeterminates $\vv=(v_1,\dots,v_P)$. Furthermore, let $A \in \RR[\vv]^{M\times N}$ be  
a matrix whose entries are polynomials in $\vv$. We are interested in the rank of $A$ 
in dependence of different values $\bar \vv \in \RR^P$ for the variables $\vv$. 
Algorithm~\ref{alg:pgauss} (ParametricRank) provides a decomposition of the affine space $\RR^P$ into disjoint semialgebraic subsets 
$$S_1,\dots,S_I$$
such that for each subset $S_i$ the rank of $A(\bar \vv)$ is constant for all points $\bar \vv$ of $S_i$.
The output of the algorithm is organized as a set of pairs 
$$R=\{ (\Gamma_1, r_1), \, \dots, \, (\Gamma_I, r_I) \}.$$ For each
$i=1,\dots,I$, $\Gamma_i$ is a conjunction of polynomial constraints in $\vv$, i.e.~polynomial equations 
and inequations in $\vv$ with coefficients in $\RR$, and $r_i$ is an integer. The semialgebraic sets $S_1,\dots,S_I$ from above are defined by $\Gamma_1,\dots,\Gamma_I$  and $r_1,\dots,r_I$ are the ranks of 
$A$ on these sets. 

In order to determine the different ranks the algorithm computes row echelon forms by introducing case distinctions. Roughly speaking the algorithm works in the following way. For a semialgebraic set $S$ defined by a conjunction of polynomial constraints $\Gamma$ and for a matrix $A$ which is in row echelon form up to the $p$-th row, the algorithm checks if there is an entry $A_{ij}$ in $A$, where $p+1\leq i\leq M$ and $p+1\leq j \leq N$, which does not vanish for any point of $S$.
If $A_{mn}$ is such an entry, then one swaps rows $p+1$ with $m$ and columns $p+1$ with $n$. One uses then the entry $A_{p+1,p+1}$ in the new matrix as a pivot element to delete all subsequent entries and introduces a new case consisting of the same set $S$, i.e.~the same $\Gamma$, but now with the matrix in row echelon form up to row $p+1$. If there is not such an entry, the algorithm checks if there is an entry $A_{ij}$, where $i$, $j$ runs as above, which does not vanish for all points of $S$. If $A_{mn}$ is such an entry, the algorithm introduces two new cases. Both cases consist of the old matrix $A$ in row echelon form up to row $p$. The difference is now that one case consists of the semialgebraic set defined by $\Gamma \land A_{mn}\neq0$ and the other one consists of the set defined by $\Gamma \land A_{mn}=0$.
The new conditions guarantee the non-vanishing and the vanishing of the entry $A_{mn}$ in the respective case. 
If there is not an entry $A_{mn}$ as described above, on $S$ the matrix $A$ is in row echelon form and has rank $p$. The algorithm returns $\Gamma$ and $p$.
The algorithm processes a stack of cases in the way described above starting with the case which consists of the semialgebraic set $\RR^P$ defined by $\Gamma="\true"$ and the matrix $A$ in row echelon form up to row $p=0$. 

Technically, the algorithm uses a logical deduction procedure $\vdash$ to heuristically derive from
$\Gamma$ whether or not relevant entries vanish in $\RR$. The
correctness of the algorithm requires only two very natural assumptions on
$\vdash$:
\begin{enumerate}
\item[(i)] If $\Gamma \vdash \gamma$, then $\Gamma$ entails $\gamma$ in $\RR$, i.e., $\vdash$ is sound;
\item[(ii)] ${\gamma \mathrel{\land} \Gamma} \vdash \gamma$, i.e., $\vdash$ can derive
  constraints $\gamma$ that literally occur on the left hand side.
\end{enumerate}
Of course, our notation in (ii) should be read modulo associativity and
commutativity of the logical conjunction operator.  Notice that (ii) is easy to
implement, and implementing only (ii) is certainly sound.

\begin{example}
We consider the $3\times 3$-matrix 
$$A^{(0)}=\begin{pmatrix}
0 & v_1 & 1 \\ v_1 & v_3^2-1+v_1 v_2 & v_3 \\ v_3-1 & v_3+1 + v_1v_3^3 & v_3^3
\end{pmatrix} $$
with polynomial entries in $\RR[\vv]$ with $\vv=(v_1,v_2,v_3)$. Starting with $(\true, A^{(0)}, 0)$ we find in $A^{(0)}$ the entry $A^{(0)}_{1,3}=1$ which 
clearly does not vanish on $\RR^3$. We swap the first and third column and use $1$ as a pivot to delete the entries $v_3$ and 
$v_3^3$. We obtain 
$$A^{(1)}=\begin{pmatrix}
1 & v_1 & 0 \\ 0 & v_3^2-1 & v_1 \\ 0 & v_3+1  & v_3-1
\end{pmatrix}$$
and introduce $(\true,A^{(1)},1)$.
For each entry $A^{(1)}_{i,j}$ with $2\leq i,j\leq3$ there is at least one point in $\RR^3$ such that $A^{(1)}_{i,j}$ vanishes.
On the other hand for each entry $A^{(1)}_{i,j}$ with $2\leq i,j\leq3$ there is also at least one point in $\RR^3$ such that  $A^{(1)}_{i,j}$ does not vanish. We choose $A^{(1)}_{2,2}=v_3^2-1$  and introduce the cases $(v_3^2-1\neq 0,A^{(1)},1)$ and  $(v_3^2-1=0,A^{(1)},1)$.
Continuing with the first case, we trivially find that the entry $A^{(1)}_{2,2}=v_3^2-1$ does not vanish on the semialgebraic set $\{\boldsymbol{\overline{v}} \in \RR^3 \mid \overline{v}_3^2-1\neq 0\}$ and we use it as a pivot to delete 
the entry $v_3+1$. We obtain 
$$A^{(2)}=\begin{pmatrix}
1 & v_1 & 0 \\ 0 & v_3^2-1 & v_1 \\ 0 & 0  & (v_3-1)^2-v_1
\end{pmatrix} $$
and introduce $(v_3^2-1\neq0,A^{(2)},2)$. Since the set $\{\boldsymbol{\overline{v}} \in \RR^3 \mid \overline{v}_3^2-1\neq 0\}$ contains points where 
$A^{(2)}_{3,3}$ vanishes and points where $A^{(2)}_{3,3}$ does not vanish, we introduce the two cases
\begin{eqnarray*}
&&(v_3^2-1\neq0 \, \land \, (v_3-1)^2\neq v_1 ,A^{(2)},2),  \\
&& (v_3^2-1\neq0 \, \land \, (v_3-1)^2= v_1 ,A^{(2)},2).
\end{eqnarray*}
In the first case, the entry $A^{(2)}_{3,3}$ clearly does not vanish on the semialgebraic set
$$ \{ \boldsymbol{\overline{v}} \in \RR^3 \mid \overline{v}_3^2-1\neq0 \, \land \, (\overline{v}_3-1)^2\neq \overline{v}_1 \}.$$
Since there is no entry left to delete with the pivot  $A^{(2)}_{3,3}$, we simply introduce the new case 
$$(v_3^2-1\neq0 \, \land \, (v_3-1)^2\neq v_1 ,A^{(2)},3) ,$$
which yields the output $(v_3^2-1\neq0 \, \land \, (v_3-1)^2\neq v_1,3)$. In the second case we have that the entry $A^{(2)}_{3,3}$ is zero on the constructed semialgebraic set and so 
the algorithm returns $(v_3^2-1\neq0 \, \land \, (v_3-1)^2= v_1,2)$. Finally, we are coming back to the case $(v_3^2-1=0,A^{(1)},1)$. On the semialgebraic set $\{\boldsymbol{\overline{v}} \in \RR^3 \mid \overline{v}_3^2-1 = 0\}$ the entries $v_3^2-1$, $v_3+1$ and $v_3-1$ are 
zero while there are points in this set where the entry $A^{(1)}_{2,3}=v_1$ vanishes and points where it does not vanish. We introduce the cases
$(v_3^2-1=0\, \land \, v_1\neq 0,A^{(1)},1)$ and $(v_3^2-1=0 \, \land \, v_1=0,A^{(1)},1)$. In the first case the entry $A^{(1)}_{2,3}=v_1$ is nonzero on the semialgebraic set
$$\{ \boldsymbol{\overline{v}} \in \RR^3 \mid \overline{v}_3^2-1=0\, \land \, \overline{v}_1\neq 0 \} $$
and so we obtain $(v_3^2-1=0\, \land \, v_1\neq 0,A^{(1)},2)$, which leads to the output 
$(v_3^2-1=0\, \land \, v_1\neq 0,2)$. In the second case all entries $A^{(1)}_{i,j}$ with $2\leq i,j\leq 3$ are zero on the corresponding semialgebraic set and therefore the algorithm returns $(v_3^2-1=0 \, \land \, v_1=0,1)$.
\end{example}

\cor{\begin{remark}
Similar approaches \cite{Sit-1992} do not necessarily yield disjoint parameter ranges $S_i$ (see \cite{BallarinKauers:04a}, Section 5.3).
\end{remark}}

Let us get back to checking for completeness and independence.
We start with Algorithm~\ref{alg:iscomplete} below which handles completeness.

\begin{algorithm}[H]
  \caption{ParametricRank}\label{alg:pgauss}
  \begin{algorithmic}[1]
    \REQUIRE An $M \times N$-matrix $A(\boldsymbol{v})$ with polynomial entries in
    parameters $\boldsymbol{v}$.
    
    \ENSURE A set of pairs
    $\{(\Gamma_1(\boldsymbol{v}), r_1), \dots, (\Gamma_I(\boldsymbol{v}), r_I) \}$, where each
    $\Gamma_{i}(\boldsymbol{v})$ is a conjunction of polynomial equations and
    inequations and $r_i \in \{1, \dots, N\}$. For any real choice
    $\bar{\boldsymbol{v}}$ of parameters $\boldsymbol{v}$ there is one and only
    one $i \in \{1, \dots, I\}$ such that
    $\Gamma_{i}(\bar{\boldsymbol{v}})$ holds in $\RR$.
    For this $i$ we have
    $\rank A(\bar{\boldsymbol{v}}) = r_{i}$.
    \STATE $I:=0$
    \STATE create an empty stack
    \STATE
    $\operatorname{push}~(\true, A, 0)$
    \WHILE{stack is not empty}
    \STATE $(\Gamma, A, p):=\operatorname{pop}$
    \IF{$\Gamma\nvdash \false$}
    \IF{there is $m\in\{p+1,\dots,M\}$, $n\in\{p+1,\dots,N\}$ s.t.~$\Gamma\vdash A_{mn}\neq0$}
    \STATE in $A$, swap rows $p+1$ with $m$ and columns $p+1$ with $n$
    \STATE in $A$, use row $p+1$ to obtain $A_{p+2,p+1}= \dots = A_{M,p+1}=0$
    \STATE $\operatorname{push}~(\Gamma, A, p+1)$
    \ELSIF{there is $m\in\{p+1,\dots,M\}$, $n\in\{p+1,\dots,N\}$ s.t.~$\Gamma\nvdash A_{mn}=0$}
    \STATE $a := A_{mn}$
    \STATE $\operatorname{push}~(\Gamma\land a\neq0, A, p)$
    \STATE in $A$, set $A_{mn}:=0$\hfill \textit{optional optimisation}
    \STATE $\operatorname{push}~(\Gamma\land a=0, A, p)$
    \ELSE[$A$ is in row echelon form modulo $\Gamma$]
    \STATE
    $I:=I+1$
    \STATE $(\Gamma_I, r_I) := (\Gamma, p)$
    \ENDIF
    \ENDIF
    \ENDWHILE
    \RETURN $\{(\Gamma_1,r_1), \dots, (\Gamma_I,r_I)\}$
  \end{algorithmic}
\end{algorithm}
\begin{algorithm}[H]
  \begin{algorithmic}[1]
    \caption{\label{alg:iscomplete}$\operatorname{IsComplete}$.  }

    \REQUIRE
    1.~$\vect{F}(\vk, \vx) = (f_1(\vk, \vx)), \dots, f_n(\vk, \vx))$;
    2.~$\vect{\Phi}(\vk,\vx) = (\phi_1 (\vk,\vx), \dots \phi_s (\vk,\vx))$;
    3.~$\vk = (k_1, \dots, k_r)$;
    4.~$\vx = (x_1, \dots, x_n)$
    
    \ENSURE ``yes'' if $\Phi$ is complete according to
    Definition~\ref{def:completeConsLaw}, ``no'' otherwise
    
    \STATE $R :=
    \operatorname{ParametricRank}\bigl(
    \vect{J}_{\vect{F},\vect{\phi}}(\vk,\vx)\bigr)$ 
    \STATE $R_n := \{\,(\Gamma, r) \in R \mid r=n\,\}$
    \STATE $\varrho_n := \bigvee_{(\Gamma,n) \in R_n} \Gamma$
    \STATE $\gamma := \forall \, \vk \, \forall \, \vx \, (
    \vk > 0 \, \land \, \vx > 0 \, \land \, \vect{F}(\vk,\vx) = 0 \longrightarrow \varrho_n)$
    \IF{$\RR \models \gamma$}
    \RETURN ``yes''
    \ELSE
    \RETURN ``no''
    \ENDIF
  \end{algorithmic}
\end{algorithm}
\begin{algorithm}[H]
  \begin{algorithmic}[1]
    \caption{\label{alg:isindependent}$\operatorname{IsIndependent}$}

    \REQUIRE
    1.~$\vect{F}(\vk, \vx) = (f_1(\vk, \vx)), \dots, f_n(\vk, \vx))$;
    2.~$\vect{\Phi}(\vk,\vx) = (\phi_1 (\vk,\vx), \dots \phi_s (\vk,\vx))$;
    3.~$\vk = (k_1, \dots, k_r)$;
    4.~$\vx = (x_1, \dots, x_n)$
    
    \ENSURE ``yes'' if $\Phi$ is independent according to
    Definition~\ref{def:completeConsLaw}, ``no'' otherwise
    
    \STATE $R := \operatorname{ParametricRank}(D_{\vect{x}}\vect{\Phi}(\vk,\vx)^T)$
    \STATE $R_s := \{\, (\Gamma, r) \in R \mid r=s\,\}$
    \STATE $\varrho_s := \bigvee_{(\Gamma,s) \in R_s} \Gamma$
    \STATE $\iota := \forall \, \vect{k} \, \forall \, \vect{x} \, (
    \vk > 0 \, \land \, \vx > 0 \, \land \, \vect{F}(\vk,\vx) = 0 \longrightarrow \varrho_s)$
    \IF{$\RR \models \iota$}
    \RETURN ``yes''
    \ELSE
    \RETURN ``no''
    \ENDIF
  \end{algorithmic}
\end{algorithm}

%
On the grounds of Algorithm~\ref{alg:pgauss}, we construct in line 2--3 in Algorithm~\ref{alg:iscomplete} an equivalent logic condition
$\varrho_n$ for $\rank \vect{J}_{\vect{F},\vect{\phi}}(\vk,\vx) = n$. In line 4 we construct
$\gamma$ as a direct formalization of the definition of completeness and in line 5 we
finally test validity of $\gamma$ over the reals. Technically we use a combination of
various effective quantifier elimination procedures for the theory of real
closed fields \cite[and the references
there]{DolzmannSturm:98a,Sturm:17a,Sturm:18a} combined with heuristic
simplification techniques \cite{DolzmannSturm:97c} in the Redlog system
\cite{dolzmann1997redlog,Kosta:16a,Seidl:06a}. Alternatively, one could use
Satisfiability Modulo Theories solving over the logic $\texttt{QF\_NRA}$
\cite{NieuwenhuisOliveras:06b,BarrettFontaine:17a,AbrahamAbbott:2016b}.
Algorithm~\ref{alg:isindependent} proceeds analogously to
Algorithm~\ref{alg:iscomplete} and tests for independence.

\begin{example}\label{ex:cointr}
  We automatically process Example~\ref{MM} with our
  Algorithm~\ref{alg:iscomplete} 
  
  $\operatorname{IsComplete}$. Our input is
  \begin{eqnarray*}
    \vect{F}(\vk,\vx)&=&( -k_{1} x_{1} x_{3} + k_{2} x_{2},
    k_{1} x_{1} x_{3} - k_{2} x_{2},
    -k_{1} x_{1} x_{3} + k_{2} x_{2}), \\ 
    \vect{\Phi}(\vk,\vx)&=&(x_1+x_2, x_2+x_3)
  \end{eqnarray*}
  and $\vk=(k_1, k_2)$, $\vx=(x_1, x_2, x_3)$. We obtain the parametric Jacobian
  \begin{displaymath}
    \vect{J}_{\vect{F},\vect{\phi}}(\vk,\vx) =
    \begin{pmatrix}
      -k_1x_3 & k_2 & -k_1x_1\\
      k_1x_3  & -k_2 & k_1x_1\\
      -k_1x_3 & k_2 & -k_2x_1\\
      1 & 1& 0\\
      0 & 1 & 1
    \end{pmatrix}
  \end{displaymath}
  and compute in line 1 the parametric rank
  \begin{displaymath}
    R=
    \{( k_{1} x_{1} + k_{1} x_{3} + k_{2} \neq 0,3),
    ( k_{1} x_{1} + k_{1} x_{3} + k_{2} = 0,2)\},
  \end{displaymath}
  from which we select $R_3=\{( k_{1} x_{1} + k_{1} x_{3} + k_{2} \neq 0,3)\}$ in
  line 2 and $$\varrho_3=(k_{1} x_{1} + k_{1} x_{3} + k_{2} \neq 0)$$ in line 3. Completeness is
  straightforwardly formalized in line 4 by
  \begin{displaymath}
    \gamma = \forall \, \vk \, \forall \, \vx \, (\vk>0 \, \land \, \vx >0 \, \land \,
    k_{1} x_{1} x_{3} - k_{2} x_{2} = 0  \longrightarrow
    k_{1} x_{1} + k_{1} x_{3} + k_{2} \neq  0),
  \end{displaymath}
  where some redundant equations are automatically removed from
  $\vect{F}(\vk,\vx)=0$ via the standard simplifier described in
  \cite{DolzmannSturm:97c}. Real quantifier elimination on $\gamma$ in line 5
  equivalently yields ``$\true$'', which confirms completeness, and we return
  ``yes'' in line 6.

  Similarly, Algorithm~\ref{alg:isindependent}
  $\operatorname{IsIndependent}$ computes
  \begin{displaymath}
   \transpose{ D_{\vect{x}}\vect{\Phi}(\vk,\vx)} = \begin{pmatrix} 1&1&0\\ 0&1&1 \end{pmatrix},
  \end{displaymath}
  $R=\{(\true, 2)\}$, $R_2=R$ and $\varrho_2=\true$.
  This yields
  $$\iota = \forall \, \vk \, \forall \, \vx \, (\vk>0 \, \land \, \vx >0 \, \land \, k_{1} x_{1} x_{3} - k_{2} x_{2} = 0 \longrightarrow
  \true),$$ which by quantifier elimination is equivalent to ``$\true$'', and we
  return ``yes''. Recall that any implication with ``$\true$'' on the right hand
  side holds already due to Boolean logic \cite{SeidlSturm:03c}. In our
  framework here, this corresponds to the observation that completeness and
  independence hold whenever we encounter full rank of the corresponding
  Jacobian for all choices of parameters.
  The automatic computations described here take less than 0.01~sec altogether.
\end{example}


The next example shows that \cor{there are models with degenerate steady states and no complete set of conservation laws.}

\begin{example}\label{example2}
One checks easily that the system 
$$\dot x_1 = 1 - x_1 - x_2, \quad \dot x_2 = x_1 + x_2 -1$$
has the linear conservation law $\Phi(\vk,\vx)= x_1 + x_2$ and that 
the Jacobian of $\transpose{(\vect{F}(\vk,\vx),$ $\Phi(\vk,\vx) )}$ is constant and  
has rank $1$. Trivially the Jacobian has everywhere rank $1$ and so $\Phi(\vk,\vx)$ is not complete. 
The explicit solutions of the ODE system 
are 
\begin{eqnarray*}
x_1(t) &=& (1-c_0)t + c_1, \\ x_2(t)&=&(c_0-1)t + c_0 - c_1.
\end{eqnarray*}
Thus, all invariant curves are of the 
form $x_1+x_2=c_0$. We conclude that there are no further first integrals and so the system has no 
complete set of conservation laws. 
We can also note that the intersection of the 
steady state variety defined by $x_1 + x_2 = 1$ with a stoichiometric
compatibility class $\{ x_1+x_2 = c_0, x_1 > 0, x_2 > 0  \}$
is either empty or a line segment.

In the absence of the parameters $\vk$ Algorithm \ref{alg:iscomplete}
$\operatorname{IsComplete}$ computes the parametric rank $R = \{(\true, 1)\}$.
However, we have $n=2$, which yields $R_2 = \emptyset$ and $\varrho_2 = \false$. This gives us
$$\gamma = \forall \, \vk \, \forall \, \vx \, (\true \, \mathrel{\land} \, \vx > 0 \, \land \, x_1 + x_2 - 1 = 0 \longrightarrow \false),$$
which quantifier elimination identifies to be ``$\false$'', and the algorithm
returns ``no''. Notice that $\gamma$ is equivalent to
$$\forall \, \vx \, (\vx > 0 \longrightarrow x_1 + x_2 - 1 \neq 0),$$ which illustrates that with deficient
rank, completeness can only hold formally when there is no steady state in the
positive orthant.
\end{example}

\section{Computing Linear Conservation Laws}\label{sec:linear}

\cor{Throughout this section we consider that the polynomial ODE system \eqref{eq:fi} has the structure given by \eqref{eq:crn}. If the model is a CRN this is automatically true, otherwise we can always bring the system to this form by using the Algorithms~\ref{alg:S},\ref{alg:findz}. }






By definition a vector $\vect{c} = (c_1,\ldots,c_n) \in \RR^n$ defines an exact linear conservation law $$\Phi(\vk,\vx) =\sum_{i=1}^n c_i x_i$$ 
of the system in  \eqref{eq:fi} \cor{with velocity field satisfying \eqref{eq:crn}}, if
$$\sum_{i=1}^n c_i f_i(\vk,\vx)
=\sum_{j=1}^r  (\vect{c}\vect{S})_j k_j  \vx^{\vect {\alpha}_j} = 0$$ for all $\vk$, $\vx$.
Exact linear conservation laws of the system in \eqref{eq:fi} can be obtained from the matrix $\vect{S}$ as the following theorem shows. 

\begin{theorem} \label{th:linear_conservation}
Assume that all monomial reaction rates
with the same multi-index have different 
rate constants, i.e. 
$\vect{\alpha}_j = \vect{\alpha}_{j'}$ with $j \neq j'$
implies $k_j  \neq k_{j'}$. This condition is always fulfilled in our setting (see the Remark~\ref{rem:remark2}). 
Let $\myrank{\vect{S}}$ be the rank of the matrix $\vect{S}$. 
Then there is 
a full rank matrix $\vect{C}$ with $n-\myrank{\vect{S}}$ rows such that $\vect{C} \vect{S} = 0$. Furthermore, 
$\vect{C}$ can be chosen such that its rows 
form a  set of independent, simple, exact linear conservation laws unconditional on $\vk$. 
Finally, all $\vect{c}=(c_1,\ldots,c_n)$ defining unconditional, exact 
linear conservation laws are linear combinations
of the rows of $\vect{C}$. 
\end{theorem}

\begin{proof}
As a direct consequence of the rank-nullity theorem, there is a matrix $\vect{C}$ with  
$n-\myrank{\vect{S}}$ independent rows such that $\vect{C} \vect{S} = 0$. If  $\vect{c}=(c_1,\ldots ,c_n)$
defines an exact, unconditional linear conservation law, we then have 
that  $$\sum_{j=1}^r  (\vect{c}\vect{S})_j k_j  \vx^{\vect {\alpha}_j} =0$$ for all $\vk$,
$\vx$.  If $i \neq j$ one has $\vect{\alpha}_i \neq \vect {\alpha}_j$ or $k_i  \neq k_j $.
It follows that $k_j  \vx^{\vect {\alpha}_j}$ are linearly independent in $\RR[\vk,\vx]$
 and $\vect{c}\vect{S} =0$. Therefore, $\vect{c}$ is a linear combination of rows of $\vect{C}$.
 Those rows of $\vect{C}$ that are not
 simple
 can be decomposed as a sum of simple
  linear conservation laws. The resulting conservation laws form a generating set for the vector space of linear conservation laws 
 and therefore this set contains a basis. 
 The $n-\myrank{\vect{S}}$ elements of this basis can be chosen as the simple rows of $\vect{C}$.
Here, the unconditionality of $\vect{c}$
results from the fact that $\vect{S}$ and therefore
$\vect{c}$ do not depend on $\vk$.
\end{proof}

\begin{remark}
Since $S_{ij} \in \ZZ$, the coefficients $c_i$ can be chosen to be integers. 
For conservative CRNs linear conservation laws can be chosen semi-positive, that is all $c_i\geq 0$ (see \cite{schuster1991determining}). Semi-positive linear
conservation laws are important, because their existence implies that concentrations of
some species are bounded at all times. Also, semi-positive conservation laws allow easy interpretation
in terms of pools of chemical species. An algorithm to compute 
 semi-positive linear conservation laws
can be found in \cite{schuster1991determining}. We must emphasize that it is not always
possible to transform complete systems of linear conservation laws into semi-positive 
linear conservation laws.  
As an example, one may consider
the formation of a heterodimer described by
$\dot x_1 = -k_1x_1x_2 + k_2 x_3,\,
\dot x_2 = -k_1x_1x_2 + k_2 x_3,\,
\dot x_3 = k_1x_1x_2 - k_2 x_3$. This system
has the complete system of
linear conservation laws
$\vect{\Phi}=(x_1-x_2,x_1+x_2-2x_3)$
that can not be transformed into a semi-positive
system. \cor{Of course, semi-positiveness is not relevant
for non-chemical application, except for population 
dynamics models used in ecology.}
\end{remark}

\begin{remark}
In Example~\ref{example2} we have seen that $n-\myrank{\vect{S}}$ independent linear conservation laws may not form a complete system. 
The system in this example satisfies \eqref{eq:crn}, although it is not a mass action network. 
\end{remark}


\begin{theorem}\label{conditions_completeness}
Consider the Jacobian matrix 
$\vect{J}(\vk,\vx) = D_{\vx} \vect{F}(\vk,\vx)$ and the matrix $\vect{C}$ introduced in Theorem~\ref{th:linear_conservation}. The rows of $\vect{C}$ provide
a complete set of linear conservation laws, if for all $\vk \in \RR_+^r$, $\vx \in \RR_+^n$ with $\vect{F}(\vk,\vx)=\mathbf{0}$  the following conditions hold:
\begin{enumerate}[label=(\roman*)]
    \item \label{conditions_completeness_item1} $rk(\vect{J}(\vk,\vx))= rk(\vect{S}) =\myrank{\vect{S}}$, 
    \item \label{conditions_completeness_item2} 
    no row of the product $J(\vk,\vx) \vect{S}$ is zero.
\end{enumerate}



\end{theorem}

\begin{proof}
We check that the assumptions imply the conditions in Definition \ref{def:completeConsLaw}.
From \ref{conditions_completeness_item1}  it follows that $\vect{J}(\vk,\vx)$ has $\myrank{\vect{S}}$ independent rows. Moreover, from \ref{conditions_completeness_item2} we obtain 
that these rows are not in the left
kernel of $\vect{S}$, which is spanned by the $n-\myrank{\vect{S}}$
rows of $\vect{C}$. Thus the matrix 
$\transpose{( \vect{J}(\vk,\vx) , \vect{C} )}$ has rank $n$.
\end{proof}


The CRN in 
Example~\ref{example2} has no complete set of linear conservation laws. This model fails to satisfy condition \ref{conditions_completeness_item2} of the theorem, but it fulfills \ref{conditions_completeness_item1}. The network is not 
a mass action CRN, so
one may ask whether mass action CRNs automatically
satisfy both \ref{conditions_completeness_item1} and \ref{conditions_completeness_item2}. The answer to this question is no, as is shown by the following example.
\begin{example}
The CRN 
$$\dot x_1 = - k_1x_1 - k_2x_1 - 2k_3x_1^2,\quad \dot x_2 = 2k_3x_1^2 + 2k_2x_1,\quad \dot x_3 = k_3x_1^2 + k_1x_1$$
is a mass action network described by the reactions
$$A_1 \xrightarrow{k_1} A_3,\quad A_1 \xrightarrow{k_2} 2 A_2,\quad 2 A_1 \xrightarrow{k_3} 2A_2 + A_3.$$
Its stoichiometric matrix  is $$\vect{S}=\begin{pmatrix} -1 &   -1 &   -2 \\ 0 &    2 &    2\\ 1  &   0 &    1 \end{pmatrix}$$ and one easily checks that it has rank two. There is a single linear conservation law, namely $\phi(\vk,\vx) = 2x_1 + x_2 + 2x_3 $. 
 The Jacobian matrix of the system is  
 $$J(\vk,\vx) = \begin{pmatrix} - k_1 - k_2 - 4k_3x_1 & 0 &  0 \\  2k_2 + 4k_3x_1 & 0 & 0  \\ 
k_1 + 2k_3x_1 & 0&   0\end{pmatrix},$$ which clearly has rank one. Therefore, this model does not satisfy  condition  \ref{conditions_completeness_item1} of the Theorem~\ref{conditions_completeness}, whereas  \ref{conditions_completeness_item2} is fulfilled. Furthermore, 
the equations $\vect{F}(\vk,\vx)=0$, $\phi(\vk,\vx)=c_0$, which 
define the intersection of $\mathcal{S}_{\vk}$ with the
stoichiometric compatibility class, 
have degenerate solutions $x_1=0,\, x_2 = c_0 - 2 x_3$ and so the conservation law is not complete. 
\end{example}

In  \cite{feliu2012preclusion}, sufficient conditions were found for mass action CRNs to be ``injective", meaning to have a unique steady state in 
any stoichiometric compatibility class (see Theorems 5.6 and 9.1 of \cite{feliu2012preclusion}). The same conditions also imply completeness 
of any $n-\myrank{\vect{S}}$ independent, linear conservation laws.

\section{Computing Nonlinear Conservation Laws}

Nonlinear conservation laws, i.e.~first integrals, of polynomial systems of ODEs have been approached by means of a variety
of methods. For example, simple conservation laws, 
i.e.~expressions composed of 
polynomials, logarithms and exponentials, can be obtained by {\bf Darboux polynomials} \cite{prelle1983elementary,cheze2011computation}. 
The bottleneck of this method is the computation of all Darboux polynomials 
of the ODE system. This problem has been approached in computational algebra by 
using Gr\"obner bases \cite{man1993computing}. 
For particular  models, \cor{for instance} Lotka-Volterra models, Darboux polynomials
and conservation laws of specific kind can be
obtained with the methods described in \cite{mahdi2017conservation}. 

In this section we restrict ourselves to polynomial
conservation laws.  
We show that polynomial conservation laws are obtained 
from syzygies, that are easier to compute than
Darboux polynomials. Our methods 
\cor{are more general and apply to models not covered}  in \cite{mahdi2017conservation}. However,  
we do not cope with parametric exponential 
conservation laws existing
in Lotka-Volterra models.

We also discuss monomial and 
rational monomial conservation laws that can be 
effectively computed without syzygies or Darboux polynomials.   
We show that the computation of  monomial conservation laws can be reduced to the 
computation of linear conservation laws. 
Finally we use the
databases {\it Biomodels} \cite{BioModels2015a} and {\it ODEbase} \cite{luders2022odebase}, to study  
the scalability of our methods.

\subsection{Computing Monomial Conservation Laws}
\label{sec:monomial}

The following proposition establishes a link between 
linear and monomial conservation laws \cor{(see Theorem 1 of \cite{goldman1987integrals} for a similar result)}.
\begin{proposition}\label{prop:MCL}
Let $E$ be a system of ODEs given by 
$$\dot{x_1}=f_1(\vk,\vx), \, \dots, \, \dot{x_n}=f_n(\vk,\vx)$$
and let $E'$ be the system of ODEs given by
$$\dot{x_1}=\frac{f_1(\vk,\vx)}{x_1}, \, \dots, \, \dot{x_n}=\frac{f_n(\vk,\vx)}{x_n}.$$
Then $E$ admits a monomial conservation law if and only if $E'$ admits a linear conservation law with 
integer coefficients. Moreover, if the linear conservation law for $E'$ is 
$m_1x_1+\dots+m_n x_n$ then the monomial conservation law for $E$ is 
$x_1^{m_1}\cdots x_n^{m_n}$.
The linear conservation laws of $E'$ are \cor{simple} if and only if the
corresponding monomial conservation laws of $E$ are \cor{simple}. 
\end{proposition}
\begin{proof}
Let $M=x_1^{m_1}\cdots x_n^{m_n}$ be a monomial. Then the derivative of $M$ is 
$$\dot M=\big( \sum_{i=1}^n m_i\frac{\dot{x_i}}{x_i} \big) M . $$
The equivalences
$$\dot M=0 
\Longleftrightarrow \sum_{i=1}^n m_i\frac{\dot{x_i}}{x_i}=0 \Longleftrightarrow \sum_{i=1}^n m_i\frac{f_i(\vk,\vx)}{x_i}=0$$
prove the first part of the proposition.

Let $\phi = m_1 x_1 + \dots + m_n x_n$ be the linear conservation law of $E'$ corresponding
to the monomial conservation law $M$ of $E$.
Assume that $M$ is \cor{not simple}, say $M=M_1M_2$ with 
$M_k= x_1^{m_{k1}}\cdots x_n^{m_{kn}}$
and $\dot M_k=0$ for $k\in \{1,2\}$.
 Then by the first part $$\dot M_k=0\Longleftrightarrow \sum_{i=1}^n m_{ki}\frac{\dot{x_i}}{x_i}=0$$ for $k\in \{1,2\}$.
Thus, $\phi = \phi_1 + \phi_2$, where $\phi_k =  m_{k1} x_1+\dots + m_{kn} x_n$ are linear conservation laws of $E'$ for 
$k \in \{1,2\}$.
This means that the linear conservation for $E'$ is  \cor{not simple}. The reverse follows from the same reasoning.
\end{proof}
\cor{We consider that $E'$ has the structure given by \eqref{eq:crn}
with a stoichiometric matrix $\vect{S}'$.
This is always possible, eventually after using the Algorithm~\ref{alg:S}. }

\begin{remark}
Proposition~\ref{prop:MCL} holds even when the r.h.s.~of $E'$ are sums of 
rational monomials, which can be the case when $E$ is a system as in \eqref{eq:fi}
and some of the polynomials $f_i(\vk,\vx)$ are not a multiple of $x_i$. In this
case, a stoichiometric matrix $\vect{S}'$ can still be associated to $E'$ using 
Algorithm~\ref{alg:S} and the linear conservation laws of $E'$ can be computed
from the matrix $\vect{S}'$. This procedure is summarized in Algorithm~\ref{alg:Smonomial}.
\end{remark}

\begin{algorithm}[ht!] 
\begin{algorithmic}[1]
    \caption{\label{alg:Smonomial}$\operatorname{StoichiometricMonomialConservation}$}

 \REQUIRE 
 Polynomial ODE system $E$ written as in \eqref{eq:fi}.
 \smallskip
\ENSURE Integer coefficient matrix $\vect{C}$ whose
rows define monomial conservation laws $\vx^{\vect{C}_i}$
for $E$.
\STATE{Compute $E'$ by dividing each ODE in $E$ by $x_i$.}
\STATE{$\vect{S}' = \text{Smatrix}(E')$.}
\STATE{Find $\vect{C}$ such that 
$\vect{C}\vect{S}'=0$ and $\vect{C}$ has $n-\myrank{\vect{S}}$ \cor{simple}, independent rows, where $\myrank{\vect{S}}=rk(\vect{S}')$. }
\end{algorithmic}
\end{algorithm}

\begin{example}\label{example3}
The model 
\begin{eqnarray*}
\dot x_1 &=& x_1 (x_2 - x_1 ) - \delta x_1, \\  \dot x_2 &=& x_2 (x_1 - x_2 )
\end{eqnarray*}
is a mass action network and is described by the reactions  
$$A_1 + A_2 \xrightarrow{1} 2A_2,\quad A_1+A_2 \xrightarrow{1} \emptyset,\quad A_2 + A_2 \xrightarrow{1} \emptyset, \quad
A_1 \xrightarrow{\delta} \emptyset.$$
The truncated system 
\begin{eqnarray*}
\dot x_1 &=& x_1 (x_2 - x_1 ) , \\ \dot x_2 &=& x_2 (x_1 - x_2 )
\end{eqnarray*}
has the steady state variety  $x_1=x_2$ and Jacobian matrix is singular on it.
After truncation and elimination of the factors  $x_1$, $x_2$ we obtain the system 
\begin{eqnarray*}
\dot x_1 &=& x_2 - x_1 , \\ \dot x_2 &=& x_1 - x_2 ,
\end{eqnarray*}
which has the linear \cor{simple} conservation law $\phi_l(x_1,x_2)=x_1 +x_2$.
We conclude that the model has the monomial \cor{simple} approximate conservation law  $\phi_m(x_1,x_2) =x_1 x_2$. 
The monomial conservation law is complete, since  
the $2\times 2$ minor
$$\mathrm{det}(\transpose{D_{\vx} (x_1(x_2-x_1),x_1x_2)} )  = - 2x_1^2$$
 of the Jacobian matrix does not vanish for $x_1 > 0$. 
For $c_0 \in \RR_+$ the intersection of the steady state variety with the compatibility class 
defined by $\phi_m(x_1,x_2) = c_0$ and $x_1,x_2 >0$ is the point 
$x_1=x_2=\sqrt{c_0}$.
\end{example}

\begin{example}\label{example5}
A model by Volpert (see \cite{mahdi2017conservation}) is described by the mass action network
$$A_1 + A_2 \xrightarrow{1} 2A_2,\quad A_2 +A_3 \xrightarrow{1} 2A_3,
\quad A_3 +A_1 \xrightarrow{1} 2A_1,$$
and the ODEs
\begin{eqnarray*}
\dot x_1 &=& x_1 (x_3 - x_2), \\ 
\dot x_2 &=& x_2 (x_1 - x_3), \\
\dot x_3 &=& x_3 (x_2 - x_1)
\end{eqnarray*}
have the steady state 
variety  $x_1 = x_2 = x_3$. 
Moreover, 
the stoichiometric matrix $$\vect{S} =  
\begin{pmatrix} 
-1 & 0 & 1 \\ 
1 & -1 & 0 \\
0 & 1 & -1 \end{pmatrix}$$ 
 has rank $\myrank{\vect{S}}=2$ and 
its left kernel is generated by
$\vect{c}=(1,1,1)$.
The stoichiometric matrix of the system obtained after
elimination of the factors $x_1$, $x_2$, $x_3$  is 
$$\vect{S}' =  
\begin{pmatrix} 
0 & -1 & 1 \\ 
1 & 0 & -1 \\
-1 & 1 & 0 \end{pmatrix}$$ and has rank $\myrank{\vect{S}}=2$. Moreover, its left 
kernel is also generated by 
$\vect{c}=(1,1,1)$.
Thus the model has two approximate conservation laws, one linear $\phi_l(\vx) = x_1 + x_2 + x_3$
and one monomial $\phi_m(\vx) = x_1 x_2 x_3$. 
Both conservation laws are complete. Indeed, the Jacobian  $\transpose{D_{\vx}(F_1,F_2,F_3,\phi_m)}$ 
has a non-vanishing $3\times 3$ minor 
$$
\mathrm{det}(\transpose{D_{\vx} (x_1(x_3 - x_2),x_2(x_1 - x_3),x_1x_2x_3)}) = x_1^2x_2x_3 + x_1x_2^2x_3 + x_1x_2x_3^2.
$$
Moreover,
the Jacobian $\transpose{D_{\vx}(F_1,F_2,F_3,\phi_l)}$  
has a $3\times 3$ minor 
\begin{eqnarray*}
M &=&\mathrm{det}(\transpose{D_{\vx} (x_1(x_3 - x_2),x_2(x_1 - x_3),x_1+x_2+x_3)}) \\
&=&- (x_1-x_2)^2  + 2x_1x_3  + 2x_2x_3 - x_3^2
\end{eqnarray*}
On $\mathcal{S}_{\vk}\cap \RR_+^3$
we have $x_1=x_2=x_3$ and so $M=3x_1^2$ does not vanish there. 

However, the two conservation laws are not independent on $\mathcal{S}_{\vk}\cap \RR_+^3$. Indeed, for $x_1=x_2=x_3$ 
we have 
$$D_{\vx}\transpose{(\phi_l,\phi_m)} = 
\begin{pmatrix}
1 & 1 & 1 \\
x_1^2 & x_1^2 & x_1^2
\end{pmatrix}
.$$

We want to point out that in this example we have used a special
choice of the rate constants.  
Parameter dependent monomial
conservation laws can be computed with 
methods presented in
\cite{goldman1987integrals}.
In particular, monomial
conservation laws with 
parameter dependent exponents are reported in \cite{mahdi2017conservation}. 
\end{example}


\subsection{Computing Polynomial Conservation Laws from Syzygies}\label{sec:syzygies}
In this section, we are going to compute parametric polynomial conservation laws. To this end, we will use results and techniques from computational algebra, that is, from Gröbner bases theory.
Our algorithm will be based on the computation of parametric syzygies for the polynomial right-hand sides in $\QQ[\vx,\vk]$ of \eqref{eq:fi}. Since these right-hand sides depend on the parameters $\vk$, we will compute a comprehensive Gröbner system and compute for each branch the parametric syzygies. The comprehensive Gröbner system will give us a decomposition of the parameter space into semi-algebraic sets such that evaluating the parameters $\vk$ at every point of such a semi-algebraic set, the syzygies of the respective Gröbner basis will still be syzygies but with entries in $\QQ[\vx]$.
Vector calculus will then tell us which syzygies are the gradients of polynomial functions in $\QQ[\vk,\vx]$ defining parametric conservation laws. 

In the first subsection, we introduce the concept of syzygies for a polynomial system, and in the second subsection, we explain comprehensive Gröbner systems. The next subsection is devoted to establishing a connection between syzygies, conservative vector fields, and conservation laws. We also explain there why we need to consider the set of syzygies as a $\QQ$-vector space instead of a $\QQ[\vx,\vk]$-module. Having this connection, in the next subsection, we develop our algorithm. Finally, in the last subsection, we present the results and discuss our benchmark.

\subsubsection{The $\QQ$ -Vector Space of Syzygies}
In this section $\QQ[\vx]$ denotes a polynomial ring in the indeterminates
$\vx=(x_1,\ldots,x_n)$ over the field of rational numbers $\QQ$. For $m$ polynomials 
\[
F=(f_1,\dots,f_m) 
\]
of $\QQ[\vx]$ we consider the ideal $I$ of $\QQ[\vx]$ generated by the polynomials in $F$.
\begin{definition}
  A syzygy of $I$ with respect to the basis $F=( f_1,\ldots,f_m )$ is an
  element $\vg=(g_1,\ldots,g_m) \in \QQ[\vx]^m$ such that
  $$  g_1f_1 + \ldots + g_mf_m = 0.$$
  A syzygy of the form 
$$\boldsymbol{g}=(0,\dots,0,f_j,0,\dots,0,-f_i,0,\dots,0) ,$$ 
where $f_j$ is the $i-$th and $-f_i$ is the $j-$th entry of $\boldsymbol{g}$, is called a trivial syzygy. The set of all syzygies for $F$ is denoted by $\Syz(F)$.
\end{definition}
 In the definition of a syzygy, some authors require
  $F=(f_1,\ldots,f_m)$ to be a minimal generating set of the ideal
  $I$ (for example \cite{eisenbud2005geometry}). However, for the purpose of this paper, we do not require such
  a condition on the generators. 
 Moreover, it is well known that $\Syz(F)$ form a $\QQ[\vx]$-submodule of $\QQ[\vx]^m$.

\begin{remark}
In this section, for computational purposes, we consider only 
polynomials over the rational numbers, that is polynomials in $\QQ[\vx]$. Note that since $\QQ \subset \RR$ all computations are also valid over $\RR$ and so all results can be considered as polynomials in $\RR[\vx]$ making them applicable to CRNs.
\end{remark}

\begin{example}
 Consider the ideal $I=\langle x_1,x_2 \rangle \subseteq \QQ[x_1,x_2]$. Since 
$$(-x_2)x_1+(x_1)x_2=0,$$ the vector $\vg=(-x_2,x_1) \in \QQ[x_1,x_2]^2$ is a trivial syzygy of the ideal $I$ 
\cor{
with respect to $F=\{x_1,x_2\}$. One can easily see that the $\QQ[x_1,x_2]-$ module $\Syz(F)$ is generated by $\vg$, that is, every syzygy is a multiplication of $\vg$ by a polynomial in $\QQ[x_1,x_2]$.
}
\end{example}

Syzygies have been extensively studied in commutative algebra and
algebraic geometry. Buchberger's algorithm for computing Gr\"obner
bases \cite{Buchberger:65a,Bose1995} leads to an algorithm which
computes a basis \cor{of} a syzygy module (see Schreyer's results in
\cite{berkesch2015syzygies,Schreyer1991}). One can also obtain a basis
of the syzygy module of $I$ with respect to a generating set $F$ as a bi-product of computing a Gr\"obner basis of $I$ via
signature-based algorithms \cite{GaoVW16,volny2011,hauke2020}.  Many
computer algebra systems such as Singular, CoCoA, etc.~include
implementations for syzygy computations. For more on syzygies and in
particular on algorithms, we refer to the books
of Weispfenning \cite{BeckerWeispfenning:93a} and Cox, Little \&
O'Shea \cite[Chapter 5.3]{cox2006using}.

Above we mentioned that $\Syz (F)$ is a $\QQ[\vx]$-module. It is clear that $\Syz (F)$ is a $\QQ$-vector space. We will need to consider here the structure of $\Syz (F)$ as a vector space instead of a module since for the computation of conservation laws we need to restrict to syzygies with a specific property which is only respected by the $\QQ$-vector space structure and not by the $\QQ[\vx]$-module structure as we will see in the subsections \ref{sub:connectionSyzygiesadnConservationlaws} and \ref{sub:ComputingConservationLawsfromSyzygies}. 
\begin{lemma}
 Let $G$ be a Gr\"obner basis for the $\QQ[\vx]$-module $\Syz (F)$ and denote the set of monomials in $\QQ[\vx]$ by $[\vx]$. Consider $\Syz  (F)$ as a vector space over $\QQ$. Then 
  \[
    \mathrm{Gen}(\Syz (F)):= \{ u \vg \mid \vg \in G, \ u \in [\vx] \}
  \]
  is a generating set for the $\QQ$-vector space $\Syz (F)$.
\end{lemma}

\begin{proof} 
 Let $\vs \in \Syz (F)$ be a syzygy. Since $G=\{ \vg_1,\dots,\vg_r\}$ is a Gröbner basis for the $\QQ[\vx]$-module $\Syz (F)$, there 
  are $h_1,\dots,h_r \in \QQ[\vx]$  such that 
  \[
  \vs = \sum_{i=1}^r h_i \vg_i .
  \]
  For each $h_i$ there are monomials $u_{ij}$ with $j=1,\dots,r_i$ with some $r_i \in \NN$ and constants $c_{ij} \in \QQ$ such that 
  \[
  h_i = \sum_{j=1}^{r_i} c_{ij} u_{ij} .
  \]
  We conclude that 
    \[
  \vs = \sum_{i=1}^r h_i \vg_i = \sum_{i=1}^r  \sum_{j=1}^{r_i} c_{ij} u_{ij} \vg_i .
  \]
  is a $\QQ$-linear combination of elements of $\{ u \vg \mid \vg \in G, \ u \in [\vx] \}$. 
\end{proof}
\begin{remark}\label{rmk:DegdSzygyVecSpace}
    If we choose a non-negative integer $d \in \NN$, then the  syzygies 
      \[
    \Syz (F)_d: = \{ \vg=(g_1,\dots,g_m) \in \Syz (F) \mid deg(g_i) \leq d,
    \quad 1 \leq i \leq m\},
  \]
    of degree at most $d$ form a finite dimensional $\QQ$-vector subspace of $\Syz (F)$. By the proof of Proposition 4, Chap. 9, Section 3 in \cite{cox96:_ideal_variet_algor}, \cor{if $G$ is a Gr\"obner basis with respect to a degree ordering}, then the set 
    \[\mathrm{Gen}(\Syz (F)_{d}):= \{ u \vg \in \mathrm{Gen}(\Syz (F)) \mid deg(u g_i) \leq d , \ 1 \leq i \leq m \} \]
    is a finite generating set for $\Syz (F)_d$. Furthermore, those elements of $\mathrm{Gen}(\Syz (F)_d)$ whose leading terms are different form a basis for the vector space $\Syz(F)_d$. 
    Alternatively, one can use Gaussian elimination to determine a basis of $\Syz(F)_d$ from $\mathrm{Gen}(\Syz (F)_{d})$.  
\end{remark}
\subsubsection{Comprehensive Gröbner Systems}\label{subsec:cgs}

To compute parametric conservation laws from syzygies for CRNs we need to handle polynomial equations whose coefficients are again polynomials in the parameters, that is polynomials 
\[
F=(f_1(\vk,\vx),\dots,f_m(\vk,\vx))
\]
in the ring $\QQ[\vk][\vx]$ in the indeterminates $\vx=(x_1,\dots,x_n)$ over the coefficient ring $\QQ[\vk]$ in the indeterminates $\vk =(k_1,\dots,k_r)$. The respective tool from computational algebra is the notion of a comprehensive Gröbner system.

 A specialisation of
$\QQ[\vk]$ is a homomorphism $\sigma: \QQ[\vk] \rightarrow \QQ$. Every
element $\boldsymbol{\alpha} \in \QQ^r$ defines a specialisation
$\sigma_{\boldsymbol{\alpha}}$ by evaluating $\vk$ at
$\boldsymbol{\alpha}$
and it can be canonically extended to a
specialisation $$\sigma_{\boldsymbol{\alpha}}: \QQ[\vk][\vx] \rightarrow \QQ[\vx].$$



\begin{definition}
  Let $F$ and $G_1,\dots, G_l$ be finite subsets of
  $\QQ[\vk][\vx]$. Further let $A_1, \dots, A_l$ be semi-algebraic
  sets of $\QQ^r$ and let $S \subseteq \QQ^r$ such that
  $S \subset A_1 \cup \dots \cup A_l$. A finite set
  $$\mathcal{G} = \{ (A_1 , G_1 ),\dots, (A_l , G_l ) \}$$ is called a
  {\bf comprehensive Gr\"obner system (CGS)} on $S$ for $F$ if for all
  $i = 1,\dots,l$ the set $\sigma_{\boldsymbol{\alpha}}(G_i)$ is a
  Gr\"obner basis for the ideal
  $\langle\sigma_{\boldsymbol{\alpha}}(F) \rangle \subset \QQ[\vx]$
  for every $\boldsymbol{\alpha} \in A_i$. Each $(A_i,G_i)$ is called
  a branch of $\cal{G}$. The generic branch corresponds to the
  semi-algebraic set that assumes all leading coefficients are
  non-zero and ignores the other cases.
   If $S = \QQ^r$, then
  $\cal{G}$ is called a comprehensive Gr\"obner system for $F$.

\end{definition}

Comprehensive Gr\"obner systems were introduced by Weispfenning in
1992 in \cite{Weispfenning92}. Independent of Weispfenning, Kapur
introduced the same concept in \cite{kapur95}, calling it parametric
Gr\"obner bases.  Today in the literature there are many publications
dealing with the theory and algorithmics of CGSs, for example
\cite{SuzukiS03,Montes02,DarmianHM11, grobcov-book,
  KapurSW13a,Kapur17, HashemiDB17} only mentioning a few.
For a survey paper on this topic see \cite{cgs-survey-lu19}.
Moreover, several computer algebra systems provide implementations of
algorithms to compute comprehensive Gr\"obner systems, such as {\it
  Reduce} \cite{10.1145/2402536.2402544,Hearn:05a,DolzmannSturm:97a,
  Sturm:07a} and {\it Singular} \cite{singular}.

\cor{
\begin{lemma}[cf. Lemma 3.7, Chapter 5.3, \cite{cox2006using}]
Let $\{ (A_1 , G_1 ),\dots, (A_l , G_l ) \}$ be a CGS of the ideal in $\QQ[\vk,\vx]$ generated by 
\[
F=(f_1(\vk,\vx), \dots, f_n(\vk,\vx)).
\]
For every $i$, $1 \leq I \leq l$, there exists a linear map that converts a basis of the syzygy modules $\Syz(F)$ to a basis of the syzygy module $\Syz(G_i)$. Also, there exists a map that converts  a basis of the syzygy module $\Syz(G_i)$ into a basis of the parametric syzygies for $F$ only valid on the semi-algebraic sets $A_i$.
\end{lemma}
}
\begin{proof}
    Consider a branch $(A_i,G_i)$ of the given CGS. In the following, we mean by $G_i$ the vector of elements of the Gröbner basis $G_i$ instead of the set $G_i$. One can use the algorithms described in \cite[Chapter 5.3]{cox2006using} to compute a basis of the $\QQ[\vk,\vx]$-module $\Syz(G_i)$ of the ideal generated by the elements of $G_i$.
    If $\boldsymbol{g} \in \Syz(G_i)$, then for every $\boldsymbol{\alpha} \in A_i$, the specialisation $\sigma_{\boldsymbol{\alpha}}(\boldsymbol{g})$ will be a syzygy of the ideal generated by $\sigma_{\boldsymbol{\alpha}}(G_i)$. Following \cite[Chapter 5.3, Lemma 3.7]{cox2006using}, there exist matrices $B^{1,i} $ and $B^{2,i}$ such that
    \[
    G_i B^{1,i}=F^T \quad \mathrm{and} \quad FB^{2,i}= G_i^T. 
    \]
    We obtain that for every $\vg \in \Syz(G_i)$ the product $(B^{2,i} \vg^T)^T$ is a syzygy of $F$. On the other hand for every $\vf \in \Syz(F)$ we have
    $(B^{1,i} \boldsymbol{f}^T)^T  \in \Syz(G_i)$. The linear map $B^{1,i}$ between the syzygy modules $\Syz(F)$  and $\Syz(G_i)$ is not necessarily one-to-one, and one \cor{may} not obtain a basis for $\Syz(F)$ by converting a basis of $\Syz(G_i)$ via $B^{2,i}$. However, the kernel of $B^{1,i}$ will give us the missing elements of a basis of $\Syz(F)$. Moreover, we have
    \[
    G_i \vg^T = (B^{2,i})^T F^T \vg^T =  (B^{2,i})^T F^T B^{1,i} \vf^T = F^T \boldsymbol{f} ,
    \]
    since $F^T=G_i B^{1,i} = (B^{2,i})^T F^T B^{1,i} $.  Applying $\boldsymbol{\sigma}_{\boldsymbol{\alpha}}$ to the equation we obtain
    \[
    0=\boldsymbol{\sigma}_{\boldsymbol{\alpha}}(G_i)
    \boldsymbol{\sigma}_{\boldsymbol{\alpha}}(\boldsymbol{g}) =
    \boldsymbol{\sigma}_{\boldsymbol{\alpha}}(F)
    \boldsymbol{\sigma}_{\boldsymbol{\alpha}}(\boldsymbol{f}),
    \]
    and therefore,
    $ \boldsymbol{\sigma}_{\boldsymbol{\alpha}}(\boldsymbol{f})$ is a syzygy of $\boldsymbol{\sigma}_{\boldsymbol{\alpha}}(F)$ for every $\boldsymbol{\alpha} \in A_i$. Hence, having computed a CGS of the ideal generated by $F$ with branches $(A_i,G_i)$ and bases of the syzygy modules $\Syz(G_i)$, we obtain bases of the parametric syzygies for $F$ only valid on the semi-algebraic sets $A_i$ by applying the conversion matrices described above to the bases elements of $\Syz(G_i)$.
\end{proof}

\cor{
Similar to the above lemma, one can use \cite[Proposition 3.8, Chapter 5.3]{cox2006using} in order to construct a matrix that converts Gr\"obner bases for $\Syz(G_i)$ into a Gr\"obner basis for $\Syz(F)$, respecting the relevant specialisations.
}

\begin{example}
We consider here the ideal generated by the polynomials 
\[
F(\vk,\vx)=(-k_1x_1,-k_2x_2,(k_1+k_2)x_1 x_2).
\]
Computing a CGS of the ideal generated by $F(\vk,\vx)$ we obtain the four branches
\begin{eqnarray*}
A_1 &=& V(\langle 0 \rangle ) \setminus V(\langle k_1 k_2\rangle ) , \ 
G_1 =( k_2 x_2, \ k_1 x_1 ) ,\\
A_2 &=& V(\langle k_2 \rangle ) \setminus V( \langle k_1 \rangle ), \
G_2 = ( k_1 x_1 ), \\
A_3 &=& \{V( \langle k_2,k_1 \rangle) \setminus V(\langle 1 \rangle) \}, \ 
G_3 = (0 ), \\
A_4 &=& V(\langle k_1 \rangle ) \setminus V(\langle k_2 \rangle), \
G_4 = (k_2 x_2 ).
\end{eqnarray*}     
We only look at the first branch which is the generic branch. We compute the conversion matrices 
\[
B^{1,1} = \begin{pmatrix}
    0 & -1 & x_1 (k_1+k_2)/ k_2 \\
    -1 & 0 & 0
\end{pmatrix}
\ \mathrm{and} \
B^{2,1} = \begin{pmatrix}
    0 & -1 \\ -1 & 0  \\ 0 & 0 
\end{pmatrix},
\]
which satisfy $G_1 B^{1,1} = F^T$ and $F B^{2,1} = G_1^T$. 
The Syzygy module of $G_1$ is generated by $\vg=(-k_1x_1,k_2x_2)$. Computing $(B^{2,1} \vg^T)^T$, one obtains the syzygy $\vf_1=(-k_2x_2,k_1x_1,0)$ of $F$. 
The kernel of the matrix $B^{1,1}$ gives us the additional syzygy $\vf_2=(0,k_1x_1+k_2x_1, k_2)$ of $F$. The syzygy module of $F$ is generated by $\vf_1$ and $\vf_2$.
 Alternatively, one can find the whole generating set of  $\Syz(F)$ using \cite[Proposition 3.8]{cox2006using} 
 \[
 \Syz(F) = \langle As_{ij},\vg_1,\dots, \vg_m \rangle.
 \]
\end{example}

\begin{example}
We consider the system 
\[
F(\vk,\vx)=(f_1,f_2,f_3)(x_2+x_1, x_2+x_1+1,k_1 x_2)
\]
This example shows how the system and its syzygies change in the different branches. We have here two branches. 
In the second branch, we have $k_1=0$ and so we only find the trivial syzygy $\vf=(-f_2,f_1,0)$. In the generic branch, the parameter has to satisfy $k_1\neq 0$. Here we find more and non-trivial syzygies, namely   
\[
\vf_1=(-k_1 x_1-k_1, k_1 x_1, 1) \ \mathrm{and} \ \vf_2= (-k_1 x_2,k_1 x_2,-1).
\]
\end{example}

\begin{example}
We consider the system 
\[ F(\vk,\vx) = (x_1+k_1 x_2+k_2,k_1x_1+x_2+k_2 ).\]
As in the example above we want to see how the system and its syzygies change in the different branches. 
  In the first branch, that is the generic one, the parameters have to satisfy $k_1\neq0$, $k_2 \neq 0$ and we find that the syzygy module is generated by the trivial syzygy $\vg=(k_1x_1+x_2+k_2,-x_1-k_1x_2-k_2)$. In the second branch the parameters are $k_1=-1$, $k_2\ne 0$ and so the system reduces to  
 \[ F(\vk,\vx) = (x_1-x_2+k_2, -x_1+x_2+k_2 ).\] 
 Again we only find the nonzero trivial syzygy here. In the third branch, the parameters have values 
$k_1=1$, $k_2=0$ which reduces our system to 
\[ F(\vk,\vx) = (x_1 +x_2,x_1+x_2  ).\] 
Here the syzygy module will be generated by $\vf=(1,-1)$.
\end{example}

\subsubsection{The Connection between Syzygies and Conservation Laws}\label{sub:connectionSyzygiesadnConservationlaws}
In this section, we consider system \eqref{eq:fi}, that is the system of ODEs of the form
\begin{equation*}
  \dot x_1 = f_1(\vk,\vx), \, \ldots, \, \dot x_n = f_n(\vk,\vx) \in
  \ZZ[\vk,\vx]. 
\end{equation*}
 We will explain how a parametric conservation law of this ODE system leads to a syzygy of the ideal generated by $F=(f_1(\vk,\vx), \, \ldots, \, f_n(\vk,\vx))$ and how one can obtain parametric conservation laws from  specific elements of $\Syz(F)$. 

\cor{
We remind that the curl of the vector field $\vg=(g_1(\vk,\vx),\dots,g_n(\vk,\vx))$ is defined as the anti-symmetric tensor field $\curl \boldsymbol{g}$ with entries
\[
\DP{g_i(\vk,\vx)}{x_j} - \DP{g_j(\vk,\vx)}{x_i}  \quad \mathrm{with} \ 1 \leq i,j\leq n .
\]

\begin{definition}
\begin{enumerate}
    \item  For the dynamical system \ref{eq:fi} and its right hand side $F$ we define the set of conservative syzygies as 
    \[
    \Syzgrad(F):=\{ \vg \in \Syz(F) \mid \curl \vg =0 \}.
    \]
    \item For a positive integer $d \in \NN$ we define the set of conservative syzygies of degree at most $d$ as 
    \[
    \Syzgrad(F)_d:=\{ \vg \in \Syz(F)_d \mid \curl \vg =0 \}.
    \]
    \item We define $\ConsLaws \subset \QQ[\vk,\vx]$ to be the set of polynomial conservation laws of $F$. 
\end{enumerate}
\end{definition}

\begin{proposition}\label{prop:cons-syz1}
There is a one-to-one correspondence between conservation laws and conservative syzygies:
$$
\varphi_1 :\ConsLaws \rightarrow  \Syzgrad(F), \quad  \phi(\vk,\vx) \mapsto \grad \phi(\vk,\vx)   \quad and \quad
\varphi_2: \Syzgrad(F) \rightarrow \ConsLaws , \quad \vg  \mapsto \int \vg 
$$
with $\varphi_1 \circ \varphi_2 =id_{\Syzgrad(F)}$ and $ \varphi_2 \circ \varphi_1=id_{\ConsLaws}$, where $\int\vg$ denotes the potential of a conservative vector field $\vg \in \Syzgrad(F)$.
\end{proposition}

\begin{proof}
First, we show that $\varphi_1$ is well-defined. Let $\phi(\vk,\vx) \in \ConsLaws$, that is, let $\phi(\vk,\vx) \in \QQ[\vk,\vx]$ be a parametric exact polynomial conservation law  for \eqref{eq:fi} with \cor{semi-algebraic} set $V_{\vk}$ for the parameters $\vk$.    We fix this semialgebraic set throughout the proof and ignore $k$. By Definition~\ref{def:exactandapproxcons}, we have
$$
\sum_{i=1}^n \DP{\phi}{x_i}(\vk,\vx) f_i(\vk,\vx) = 0 , \quad \forall \vk \in V_{\vk}, \ \forall \vx \in \RR^n , 
$$ 
which implies that
$$
\left(\DP{\phi}{x_1}(\vk,\vx), \, \dots, \, \DP{\phi}{x_n}(\vk,\vx) \right) \in \Syz(F).
$$
In other words, this means that the gradient
$$
\grad \phi(\vk,\vx) =\left(\DP{\phi}{x_1}(\vk,\vx), \, \dots, \, \DP{\phi}{x_n}(\vk,\vx) \right)
$$ 
of a parametric polynomial conservation law $\phi(\vk,\vx)$ is a syzygy of the ideal generated by $F$ when $\vk \in V_{\vk}$. Moreover, by well-known identities of gradient and curl, we have that 
\[
\curl (\grad \phi) =0
\]
and so $ \grad \phi \in \Syzgrad(F)$. This shows that $\varphi_1$ is well-defined.

Second, we show that $\varphi_2$ is well-defined.  Assume that  $\vg=(g_1,\dots,g_n) \in \Syzgrad (F)$ is a parametric syzygy of the ideal generated by $F$ with semi-algebraic set $V_{\vk}$ obtained from the comprehensive Gröbner system described at the end of subsection~\ref{subsec:cgs}. Therefore, 
$$f_1(\vk,\vx) g_1(\vk,\vx) + \dots + f_n(\vk,\vx) g_n(\vk,\vx)=0$$ for points $\vk \in V_{\vk}$.  

It is well-known in vector calculus that a potential function $\phi$ for  $\vg$ as a vector field, is well-defined in a simply connected domain $\mathcal{D}\subset \RR^n$ if only if $\curl \vg = 0$ everywhere in $\mathcal{D}$, i.e.,~only when $\vg$ defines a conservative (or irrotational) vector field in $\mathcal{D}$ (see \cite{spivak2018calculus}). In this case, $\phi$ satisfies
\[
\grad \phi(\vk,\vx) = \vg.
\]



By well-known vector calculus method, one can integrate $\vg$ along a suitable path to obtain the corresponding potential function. For the sake of self-containedness, we explain one possible integration below. Consider the differential form $d\vg=g_1dx_1+\dots+g_ndx_n$ and take a piecewise smooth curve connecting the origin with $\vx$, for instance
\begin{gather*}
  \mathcal{C}=\{ (s x_1,0,\dots,0) \mid s\in [0,1]\} \cup \\ \{
  (x_1,sx_2,0,\dots,0) \mid s\in [0,1]\} \cup \dots \cup \{
  (x_1,\dots,x_{n-1},sx_n) \mid s\in [0,1]\}.
\end{gather*}
Integrating the differential form $d\boldsymbol{g}$ along the curve
$\mathcal{C}$ yields
\begin{equation}\label{eq:integ-syz1}
  \int_0^{x_1}g_1(\cdot,0,...,0)dx_1
  +\int_0^{x_2}g_2(x_1,\cdot,0,...,0)dx_2 +
  \dots+\int_0^{x_n}g_n(x_1,...,x_{n-1},\cdot)dx_n,
\end{equation}
where the constant of the integration is assumed to be zero. We denote this integration (with constant being zero) by $\int \vg$.
%
Denote that if $\curl \boldsymbol{g} \neq 0$,
the result of the integration depends on the path $\mathcal{C}$ and
there is no function $\phi(\vk,\vx)$ satisfying
$\grad \phi(\vk,\vx) = \boldsymbol{g}$ for all $\vx \in \RR^n$. 

Therefore, if $\vg \in \Syzgrad (F)$, then integration implies that $\int \vg $ is an element of $\ConsLaws$ as we showed above. This shows that the definition of the map $\varphi_2$ is correct. 

Having shown that both $\varphi_1$ and $\varphi_2$ are well-defined, we just need to show that $\varphi_1$ and $\varphi_2$ are inverse to each other. 
The identity $ \varphi_2 \circ \varphi_1=id_{\ConsLaws}$ follows from $\int \grad \phi = \phi$. Moreover, 
since $\grad \int \vg = \vg$, we have $\varphi_1 \circ \varphi_2 =id_{\Syzgrad(F)}$. 
\end{proof}

\begin{remark}
In vector calculus, there are other standard methods to compute from a conservative syzygy
$\vg=(g_1(\vk,\vx),\dots,g_n(\vk,\vx))$, the corresponding conservation law $\phi(\vk,\vx)$. One can first start with integrating $g_n$ with respect to $x_n$, obtaining 
\[
\phi(\vk,\vx) = \int_{x_n}g_n(\vk,\vx) dx_n + c_{n-1}(\vk,x_1,\dots,x_{n-1}),
\]
where the integration constant $c_{n-1}(\vk,x_1,\dots,x_{n-1})$ is in $\QQ[\vk,x_1,\dots, x_{n-1}]$. Let $p_n:=\int_{x_n}g_n(\vk,\vx) dx_n$ where we choose the integration constant to be zero.
Then in order to compute $c_{n-1}(\vk,x_1,\dots,x_{n-1})$, one can use the fact that $\DP{\phi(\vk,\vx)}{x_{n-1}} = g_{n-1}(\vk,\vx)$ and obtain an equation for $\DP{c_{n-1}(\vk,x_1,\dots,x_{n-1})}{x_{n-1}}$, namely 
\[
\DP{\phi(\vk,\vx)}{x_{n-1}} = \DP{p_n(\vk,\vx)}{x_{n-1}} +\DP{c_{n-1}(\vk,x_1,\dots,x_{n-1})}{x_{n-1}}.
\]
Replacing $\DP{\phi(\vk,\vx)}{x_{n-1}}$ with $g_{n-1}(\vk,\vx)$ one obtains
\[
g_{n-1}(\vk,\vx)= \DP{p_n(\vk,\vx)}{x_{n-1}} +\DP{c_{n-1}(\vk,x_1,\dots,x_{n-1})}{x_{n-1}}.
\]
As we know the polynomials $g_{n-1}(\vk,\vx)$ and $\DP{p_n}{x_{n-1}}$, one can obtain 
equations for the coefficients of $\DP{c_{n-1}(\vk,x_1,\dots,x_{n-1})}{x_{n-1}}$. Continuing the above procedure by taking partial derivatives of $\phi(\vk,\vx)$ with respect to the variables $x_{n-2}$, replacing it with $g_{n-2}(\vk,\vx)$, one can obtain more equations for the coefficients of the integration constant $c_{n-1}(\vk,x_1,\dots,x_{n-1})$. As $\vg(\vk,\vx)$ is a conservative syzygy, the existence of a conservation law $\phi(\vk,\vx)$ such that $\grad\phi(\vk,\vx) = \vg(\vk,\vx)$ is guaranteed, which implies that the equations obtained for the coefficients of $c_{n-1}(\vk,x_1,\dots,x_{n-1})$ yields a unique solution.   

An alternative method for integration is considering the following integration
  \begin{equation}
    \phi = \int g_1 dx_1 + \dots \int g_n dx_n + C,
\end{equation}
and then removing the duplicated terms that show up in the integrals.
\end{remark}
}

\begin{example}\label{ex:cons-syzygy}
Consider the trivial syzygy $\vg=(-x_2,x_1)$ for the ideal in $\QQ[x_1,x_2]$ generated by $F=( x_1,x_2 )$.  As the vector field
arising from this particular syzygy is not irrotational, i.e.,
$\curl \vg \ne 0$, one cannot obtain a parametric polynomial
conservation law from this syzygy. But there are generators of the same ideal
with a trivial syzygy which is the gradient of a parametric
conservation law. For example, consider the generators
$F=( x_1,-x_2 )$ with the trivial syzygy
$\vg=(x_2,x_1)$. Then $\curl \vg =0$, which means that the trivial syzygy is
irrotational and so it is the gradient of a parametric conservation
law, which is $x_1 x_2 - c$, where $c$ is a constant.  
\end{example}

\begin{example}[Linear Conservation Laws]
  Assume that $F$ admits a degree zero syzygy $\vg$. As the entries of a degree
  zero syzygy are constants, we have $\curl \vg =0$ and so a degree zero syzygy is conservative. This means that $\Syzgrad (F)_0 = \Syz(F)_0$.
  Conservation laws corresponding to degree zero syzygies are the well-known linear conservation laws.
\end{example}
\begin{example}
 A syzygy of the form $\vg=(x_1^{\alpha_1},\dots, x_n^{\alpha_n})$ is
  obviously conservative and leads to the conservation law
  \[ \phi(\vx)= \alpha_1^{-1}x_1^{\alpha_1 +1}+\dots+  \alpha_n^{-1}x^{\alpha_n+1}.\]  
\end{example}

\begin{remark}\label{rem:algstructure}
  \begin{enumerate}
  \item First we want to point out that $\ConsLaws$ is not necessarily an ideal in the ring $\QQ[\vk,\vx]$. Indeed, let $\phi(\vk,\vx) \in \ConsLaws$, $\phi(\vk,\vx) \neq 0$ and 
  $h \in \QQ[\vk,\vx]$. Then by the chain rule, we have
  $$ 
  \grad(h \phi) = h \grad \phi(\vk,\vx) + \phi(\vk,\vx) \grad h.
  $$
  Since $\phi(\vk,\vx)$ is a conservation law, that is $\grad \phi(\vk,\vx) \cdot F= 0$, we obtain that 
  $$  \grad(h \phi(\vk,\vx)) \cdot F= ( h \grad \phi + \phi(\vk,\vx) \grad h)\cdot F = \phi(\vk,\vx) (\grad h \cdot F)  $$
  and $\grad h \cdot F$ is in general not zero as the following example shows. In Example~\ref{ex:cons-syzygy}, for the generating set $F=(x_1,-x_2)$, we found the conservation law $\phi(\vx)=x_1x_2$. One easily checks that the product $x_1 \phi(\vx)$ is not a conservation law.
  \item \label{rem:algstructurepoint2} The set $\Syzgrad(F)$ is not necessarily a $\QQ[\vk,\vx]$-submodule of the $\QQ[\vk,\vx]$-module
   $\Syz(F)$. Indeed, for $h\in \QQ[\vk,\vx]$ and $\vg \in \Syzgrad(F)$ we obtain that
    \[
      \curl (h \vg) = h (\curl \vg) + \grad h \times \vg,
    \]
    where we mean by $\boldsymbol{u} \times \vv$ for two vectors $\boldsymbol{u}=(u_1,\dots, u_n)$ and $ \vv=(v_1,\dots,v_n)$ the vector with entries 
    \[
     u_iv_j-u_jv_i , \quad \mathrm{with} \ 1 \leq i, j \leq n, \ i \neq j.
    \]
     As $\vg \in \Syzgrad(F)$, we have that $\curl \vg =0$, however, not
    necessarily $\grad h \times \vg=0$ as the example at the end shows.
     Therefore, $\Syzgrad(F)$ is not
    necessarily closed under multiplication by the elements of
    $\QQ[\vk,\vx]$, which implies that $\Syzgrad(F)$ is not necessarily a $\QQ[\vk,\vx]$-submodule of
    $\Syz(F)$.

     As a counter-example, consider again the generating set 
     $F=(x_1,-x_2)$ in Example~\ref{ex:cons-syzygy}. It yields the conservative syzygy $\vg=(x_2,x_1)$. However, one easily checks that the product $x_1\vg$ is not a conservative syzygy, that is $\curl (x_1\vg) \neq 0$.
  \end{enumerate} 
\end{remark}

\begin{remark}\label{rmk:ConsSyzNotModule}
    We want to point out that by Remark \ref{rem:algstructure}, \ref{rem:algstructurepoint2} it is not sufficient to check only the elements of a basis of the syzygies module $\Syz(F)$ to be conservative to obtain from the ones which are conservative all of $\Syzgrad (F)$, since $\Syzgrad (F)$ has not the structure of a $\QQ[\vk,\vx]$-module.
\end{remark}

\begin{proposition}\label{prop:DegdConsSyzVecSpace}
    The set $\Syzgrad(F)$ is a $\QQ$-vector space.
\end{proposition}
\begin{proof}
    It is easy to check that the operator $\curl (\cdot)$ is $\QQ$-linear, that is for $c\in \QQ$ and $\vg_1$, $\vg_2 \in \QQ[\vk,\vx]^n$ we have
    \[
    \curl (c \vg_1 +  \vg_2) = c \curl \vg_1 +  \curl \vg_2 .
    \]
    Thus, if $\vg_1$, $\vg_2 \in \Syzgrad (F)$, then $\curl (c \vg_1 +  \vg_2)=0$ and so $\Syzgrad (F)$ is a $\QQ$-vector space.  
\end{proof}

\cor{
We will use the $\QQ-$vector space structure of $\Syzgrad(F)$ for computing conservation laws. Apart from this, the conservation laws form an algebra. Although this is well-known in the literature on first integrals (e.g., see Arnold's book~\cite[\S 10.3]{Arnold-ode-book}), we present it here and detail the proof, not only for the sake of self-containedness, but also because we will refer to the algebra structure of conservation laws and discuss its generators when we present our benchmark results.
}

\begin{proposition}\label{prop:alg-cons1}
The set $\ConsLaws$ is a $\QQ$-subalgebra of $\QQ[\vk,\vx]$.
\end{proposition}
\begin{proof}
 Let $\phi_1(\vk,\vx)$, $\phi_2(\vk,\vx) \in \ConsLaws$ be two conservation laws. The chain rule implies 
    that
    $$\grad(\phi_1(\vk,\vx) \phi_2(\vk,\vx) ) = \phi_2(\vk,\vx) \grad \phi_1(\vk,\vx) +
    \phi_1(\vk,\vx) \grad \phi_2(\vk,\vx).$$ Moreover, since  $\phi_1(\vk,\vx)$, $\phi_2(\vk,\vx)$ are conservation laws, we have that 
    $\grad \phi_1(\vk,\vx) \cdot F = \grad \phi_2(\vk,\vx) \cdot F = 0$ and so we obtain
    \begin{align*}
      \grad(\phi_1(\vk,\vx)\phi_2(\vk,\vx))\cdot F  & =
                                    \left( \phi_2\grad\phi_1(\vk,\vx) +
                                    \phi_1(\vk,\vx)\grad\phi_2(\vk,\vx) \right) \cdot
                                    F \\
                                  & = \left( \phi_2(\vk,\vx)\grad\phi_1(\vk,\vx) \cdot
                                    F \right) + \left( 
                                    \phi_1(\vk,\vx)\grad\phi_2(\vk,\vx) \cdot
                                    F \right) \\
                                  & = 0 + 0.       
    \end{align*}
 So the product of two conservation laws
    is a conservation law is again a conservation law. Since $\grad$ is additive, $\phi_1(\vk,\vx) + \phi_2(\vk,\vx)$ is again a conservation law. Finally, for $c \in \QQ$ we have 
    \[
    \grad (c \phi_1(\vk,\vx)) \cdot F = c (\grad\phi_1(\vk,\vx) \cdot F) =0,
    \]
    that is $c \phi_1(\vk,\vx)$ is again a conservation law.
    This implies that $\ConsLaws$ is a
    $\QQ$-subalgebra of $\QQ[\vk,\vx]$.
\end{proof}

\subsubsection{Computing Conservation Laws from Syzygies}\label{sub:ComputingConservationLawsfromSyzygies}

As pointed out in Remark \ref{rmk:ConsSyzNotModule},  conservative syzygies do not form a $\QQ[\vk,\vx]$-module, and hence, conservative syzygies cannot be generated by the conservative elements of a basis for the syzygy module. However, we have seen  in Proposition \ref{prop:DegdConsSyzVecSpace} that conservative syzygies  form a $\QQ$-vector space. Restricting to conservative syzygies up to degree $d$, one can compute a finite basis for the vector space $\Syzgrad (F)_d$ via a basis of the syzygy module $\Syz(F)_d$ by the $\QQ$-linearity of 
$\curl (\cdot)$.

\begin{proposition}
  \begin{enumerate}
  \item  
  The set 
  $$V_{\grad}(F) = \{\curl \boldsymbol{f} \mid \boldsymbol{f} \in \Syz(F) \} $$ 
  is a $\QQ$-vector space. 
    \item The curl induces a surjective $\QQ$-linear map 
  \[
  \curl (\cdot): \Syz(F) \rightarrow V_{\grad}(F), \ 
  \boldsymbol{f} \mapsto \curl \boldsymbol{f}
  \]
    with 
    \[
    \mathrm{ker}( \curl (\cdot)) = \Syzgrad (F).
    \]
  \end{enumerate}
\end{proposition}
\begin{proof}
  \begin{enumerate}
  \item 
 Since $\Syz(F)$ is a $\QQ$-vector space and $\curl (\cdot)$ is a $\QQ$-linear operator, its image is a $\QQ$-vector space.  
 \item The map $\curl (\cdot)$ is clearly $\QQ$-linear and surjective. Obviously every element of  $\mathrm{ker}( \curl (\cdot))$ is an element of $\Syzgrad (F)$ and every syzygy $\vg$ with $\curl \vg=0$ lies in the kernel.
\end{enumerate}
\end{proof}

We remind that a basis for the $\QQ$-vector space of syzygies up to a given degree bound $d \in \NN$ can be computed as in Remark~\eqref{rmk:DegdSzygyVecSpace}. Then one can apply the curl operator and compute its kernel in order to obtain a basis for the conservative syzygies up to degree $d$.

 We combine now our results into Algorithm~\ref{alg:pol-cons} which computes for the right-hand side $F(\vk,\vx)$ of an ODE system a basis of the $\QQ$-vector space of parametric conservation laws up to a given degree $d$.

\begin{algorithm}[ht!]
\begin{algorithmic}[1]
  \caption{\label{alg:pol-cons}$\operatorname{ParametricPolynomialConservationLawsViaSyzygies}$
}  
  \REQUIRE A polynomial ODE system
  $F(\vk,\vx)=(f_1(\vk,\vx),\dots,f_n(\vk,\vx))$ whose
  r.h.s.~$f_i(\vk,\vx)$ are polynomials in $\QQ[\vk,\vx]$ and a degree bound $d$. 
    \smallskip
    \ENSURE
    A set of pairs
    \[
      \ParPolConsLaws=\{ (V_{\vk,1},\ConsLaws_{1}),\dots
      ,(V_{\vk,l},\ConsLaws_{l})   \},
    \]
    where $V_{\vk,i} \subseteq \QQ_+^r$ is a semi-algebraic set and
    $\ConsLaws_{i}$ is a basis of the vector space of parametric polynomial conservation laws 
    of degree at most $d$ valid for all $\vk \in V_{\vk,i}$. 
    \STATE{$\ParPolConsLaws:=\emptyset$}
    \STATE{$I:=\Ideal(f_1(\vk,\vx),\dots,f_n(\vk,\vx))$ in $\QQ[\vk,\vx]$}
    \STATE{Compute a
      CGS $\mathcal{G}:=\{(A_1,G_1),\dots, (A_l,G_l)\}$ of $I$}
  \FOR{$i=1, \dots, l$}
       \STATE{$\ConsLaws_{i}:=\emptyset$}
    \STATE{Compute a basis $B(\Syz(G_i))$ of the syzygy
      module $\Syz(G_i)$ of $G_i$}
    \STATE{Convert $B(\Syz(G_i))$ into a basis $B(\Syz(F))$ of the parametric syzygies $\Syz(F)$ of $F$ valid for $A_i$.}
      \STATE{Compute a basis $B(\Syz(F)_d)$ of the generating set 
      \[\mathrm{Gen}(\Syz(F)_d)=\{ u \vg \mid u \in [\vk,\vx], \ \vg \in B(\Syz(F)), \ \deg(u g_i) \leq d \} \] of the $\QQ$-vector space $\Syz(F)_d$.}
       \STATE{Use $B(\Syz(F)_d)$ to compute a basis $B$ of the kernel of the restriction of $\curl (\cdot)$ to $\Syz(F)_d$.} 
    \FOR{ $\boldsymbol{b}=(b_1,\dots,b_n) \in B$}
    \STATE{ 
      $\phi(\vk,\vx):= \int \boldsymbol{b}$}
    \STATE{$\ConsLaws_{i}:=\ConsLaws_{i} \cup \{ \phi(\vk,\vx) \}$} 
    \ENDFOR
    \STATE{Compute $V_{\vk,i}:= A_i \cap \QQ^r_+$}
    \STATE{$\ParPolConsLaws:=\ParPolConsLaws \cup \{ (V_{\vk,i},\ConsLaws_{i})\}$}
    \ENDFOR
    \RETURN    $\ParPolConsLaws$.
\end{algorithmic}
\end{algorithm}

\begin{remark}
In line 14 of Algorithm~\ref{alg:pol-cons} one needs to determine the intersection $ A_i \cap \QQ^r_+$, where 
$A_i$ is the semi-algebraic set computed via comprehensive Gröbner basis in line 3. 
The set $A_i$ is described by equations and inequalities in $\QQ[\vk]$. Joining them the inequalities 
 $k_1 >0,\dots, k_r >0$, which represent the positive orthant, one obtains a set of equations and inequalities describing $V_{\vk,i}$.
By applying the cylindrical algebraic
decomposition(CAD), implemented in e.g., QEPCAD~\cite{Brown:03a}, to
the latter set of equations and inequations one may obtain another (though
not necessarily better) representation of $V_{\vk,i}$. The program
SLFQ
(\textit{Simplifying Large Formulas with
  QEPCAD})~\cite{BrownGross:06a} can be used once QEPCAD is applied to
obtain such a simplified version. 
\end{remark}

\begin{remark}
  Let $\{ \phi_1(\vk,\vx),\dots, \phi_s(\vk,\vx)\} \subseteq \ConsLaws_{i})$ 
  be a set of $s$ parametric polynomial conservation laws with respect to the semi-algebraic set $V_{\vk,i}$ returned by Algorithm  \ref{alg:pol-cons}. 
 Then their Jacobian matrix
  \[ \left( \begin{matrix} \frac{\partial \phi_1(\vk,\vx)}{\partial x_1} &
        \dots & \frac{\partial
          \phi_1 (\vk,\vx)}{\partial x_n} \\
          \vdots & & \vdots \\
        \frac{\partial \phi_s(\vk,\vx)}{\partial x_1} & \dots &
        \frac{\partial \phi_s (\vk,\vx)}{\partial x_n}
      \end{matrix}
    \right)
  \]
  has rank $s$. Indeed, the conservation laws stem from $\QQ$-linearly independent syzygies, that is, there are 
  $\QQ$-linearly independent syzygies $\boldsymbol{g}_1,\dots,\boldsymbol{g}_s$ such that 
  $\grad \phi_i =\vg_i$ for all $1 \leq i \leq s$. Thus the Jacobian has rank $s$. Note that the conservation laws 
  $\phi_1(\vk,\vx),\dots, \phi_s(\vk,\vx) $ are not necessarily algebraically independent over $\QQ$. One can check $\phi_1(\vk,\vx),\dots, \phi_s(\vk,\vx) $ for algebraic independence using Gröbner bases  \cite[Chapter 9, section 5]{cox96:_ideal_variet_algor} or linear algebra.
\end{remark}

\begin{remark}\label{rmk:uncond-law-elimination1}
  By  Definition~\ref{def:exactandapproxcons} parametric conservation laws
  depend on the rate constants $\vk$, whereas unconditional conservation
  laws are independent of $\vk$.  If $\phi$ is an unconditional
  conservation law, then it can be seen as a parametric conservation law as well,
  hence, $\phi$ must be obtained from a syzygy that does not contain
  variables $\vk$.   Therefore, in order to compute unconditional conservation
  laws, we need to find those syzygies of the ideal generated by $F(\vk,\vx)$ that only depend on the variables $\vx$ and not on $\vk$.  To this end, we compute directly a basis of the syzygy module of $F(\vk,\vx)$ without the detour of the CGS using the elimination property of Gr\"obner bases. We fix a block term order on the variables $\vk$, $\vx$ with $\vk>\vx$.   If $G$  is a Gr\"obner basis for $\Syz(F)$ with respect to the block order, then
  $G \cap \QQ[\vx]^m$ is a Gr\"obner basis for $\Syz(F) \cap \QQ[\vx]^m$, see \cite[Thm 3.4.5]{robbiano-kreuzer-book1}.
  This means that in order to compute unconditional conservation laws,
  one can modify Algorithm \ref{alg:pol-cons} by adding a computation
  of $G \cap \QQ[\vx]^m$. Also in order to compute a generating set for $\Syz(F)_d\cap \QQ[\vx]^m$ one needs to multiply the elements of the Gröbner basis $G \cap \QQ[\vx]^m$ with the monomials only in $\vx$.
\end{remark}

\begin{remark}\label{rmk:non-CGB1} In the context of CRNs, some
  authors work with the ring $\QQ(\vk)[\vx]$ without specialising
  $\vk$, while in Algorithm \ref{alg:pol-cons}, we work with the ring
  $\QQ[\vk,\vx]$, considering case distinctions done for CGS
  computation, as $\vk$ can be specialised in $\QQ_+^r$. One can see
  that a Gr\"obner basis $G$ of the ideal generated by $F(\vk,\vx)$  in
  $\QQ(\vk)[\vx]$ corresponds to the Gr\"obner basis in the generic
  branch of a comprehensive Gr\"obner system in the ring
  $\QQ[\vk,\vx]$. 
  This fact about the generic branch has been considered in
  \cite{nabeshima-aca:22}, and a new algorithm for computing CGS has
  been presented based on this. Details on how to obtain the semi-algebraic set of the generic branch from $G$ can be found in \cite{nabeshima-aca:22}.
\end{remark}

We implemented Algorithm \ref{alg:pol-cons} in a \texttt{Singular}
library 
called~\texttt{polconslaw.lib}. Our implementation uses \texttt{grobcov} library for computing comprehensive Gr\"obner systems.  We also implemented the approach in Remark \ref{rmk:non-CGB1} for computing polynomial conservation laws
corresponding to the generic branch of a CGS in a Singular library, called \texttt{genpolconslaw.lib}, 
by just computing the conservation laws considering $\QQ(\vk)[\vx]$. 
Therefore, this approach is much faster than computing the
whole comprehensive Gr\"obner system and picking up the generic
branch. 
\cor{
Our Singular libraries \texttt{polconslaw.lib} and \texttt{genpolconslaw.lib} are available in Zenodo\footnote{\url{https://zenodo.org/record/7474523}} and in Github\footnote{\url{https://github.com/oradules/polynomial_conservation_laws}}.
}

We present a series of examples clarifying our computational methods. 



\begin{example}\label{ex:tiny-ex1}
  For the system 
  \[  \dot x_1 =  \ x_2,\
    \dot x_2 = \ x_1,\
    \dot x_3 = \ x_1x_2,
  \]
  the syzygy module has the basis
  \[\{(0,x_2,-1),\ (x_1,-x_2,0) \}.\] 
  Multiplying these two syzygies by all monomials in the variables $x_1$, $x_2$, $x_3$ up to degree three including also the monomial $1$, that is by the elements of the set 
  \[
  \{1, x_1, x_2, x_3, x_1^2, \dots, x_3^2 x_1, x_3^2 x_2, x_1 x_3 x_2 \},
  \]
  we obtain forty $\QQ$-linearly independent syzygies. We compute for these syzygies the images under the map $ \curl (\cdot)$ and then the $\QQ$-linear dependencies between the results. We obtain that the kernel of the linear map $ \curl (\cdot)$ has a basis formed by the elements
  \begin{gather*}
      \vf_1=(0,x_2,-1), \  \vf_2=(0,x_2^3-2x_2x_3,-x_2^2+2x_3), \   \vf_3=(x_1,-x_2,0), \\ 
       \vf_4=(x_1^3-2x_1x_3,0,-x_1^2+2x_3), \  \vf_5=(x_2^2 x_1-2x_1x_3,x_1^2x_2-x_2^3,-x_1^2+2x_3), \ 
  \end{gather*}
  that is $\vf_1,\dots,\vf_5$ span the $\QQ$-vector space of all syzygies up to degree $3$ with curl zero. Integration leads to the five unconditional conservation laws 
  \begin{gather*}
      \phi_1= \frac{x_2^2}{2} - x_3 , \ \phi_2= \frac{1}{4} x_2^4 - x_2^2 x_3 + x_3^2 , \ \phi_3 = \frac{x_1^2}{2} - \frac{x_2^2}{2} , \ 
      \phi_4 = \frac{1}{4} x_1^4 - x_1^2 x_3 + x_3^2, \\ \phi_5 = \frac{1}{2} x_1^2 x_2^2 - \frac{1}{4} x_2^4 - x_1^2 x_3 + x_3^2.
  \end{gather*}  
  
\cor{
Note that $\vf_1$ and $\vf_3$ are the generators of the syzygy module, and they are both conservative, hence, they are integrated to the conservation laws $\phi_1$ and $\phi_2$. One can see that the three other conservation laws obtained above are actually in the algebra generated by $\phi_1$ and $\phi_2$. The syzygy $\vf_2$ is the product $(x_2^2-2 x_3) \ \vf_1 = 2\phi_1\vf_2$, that $\vf_4$ is the product $(x_1^2-2x_3) (\vf_1+\vf_3) = 2(\phi_1+\phi_3)(\vf_1+\vf_3)$ and that the syzygy $\vf_5$ is the product $(x_1^2-2x_3) \vf_1+(x_2^2-2x_3) \vf_3 = 2(\phi_1+\phi_3)\vf_1 = 2\phi_1\vf_3$. 

The computations for this example can be performed automatically using our Singular library \texttt{genpolconslaw.lib}, in the corresponding files \texttt{example6\_29.sg}, in our Github and Zenodo Repositories.
}
\end{example}

\begin{example}\label{ex:pol-con-law-11}
  We apply Algorithm \ref{alg:pol-cons} to the system 
  \begin{align*}
    \dot{x_1}= &-k_1x_1,\\
    \dot{x_2}= &-k_2x_2,\\
    \dot{x_3}= & (k_1+k_2)x_1x_2
  \end{align*}
  only considering the generic branch of a comprehensive Gr\"obner basis, which is given by the
  semi-algebraic set defined by $k_1 \ne 0$ and $k_2 \ne 0$. For the generic branch a
  basis of the syzygy module is 
  \[
    \{(x_2,x_1,1), (-k_2x_2,k_1x_1,0)\}.
  \]
  From the above generating set, one obtains 20 syzygies that form  a basis for the $\QQ-$vector space of syzygies up to degree three, out of which only two elements are conservative. One conservative syzygy is $(x_2,x_1,1)$, which is integrated to the unconditional conservation law 
  \[ 
  \phi_1= x_1x_2+x_3.
  \]
  The other conservative syzygy  
  \[
  (x_1 x_2^2+x_2 x_3, x_1^2 x_2 + x_1 x_3, x_1 x_2 + x_3)
  \]
 leads to the degree three conservation law 
\[
\phi_2 = \frac{1}{2}x_1^2 x_2^2 + x_1 x_2 x_3 +\frac{1}{2} x_3^2.
\]
One can easily see that $\phi_2 = \frac{1}{2}\phi_1^2$, hence the two conservation laws computed are algebraically dependent.
One can check that the second element of the module basis is not conservative, because
  \[
  \curl (-k_2x_2,k_1x_1,0) =(\DP{}{x_2}(-k_2x_2)- \DP{}{x_1}(k_1x_1) = -(k_1+k_2),0,0) \neq 0.
  \]

Note that
\[
\{ ((k_1+k_2)x_2, 0, k_1), (-k_2x_2,k_1x_1,0) \}
\]
is another basis of the syzygy module of the generic branch, in which both elements have non-zero curl, but the linear combination of the basis elements
\[
 ((k_1+k_2)x_2, 0, k_1) + (-k_2x_2,k_1x_1,0)  = k_1 (x_2,x_1,1)
\]
 has zero curl and it is the only one 
of degree two and up to multiplication.
\cor{
Computations for all branches of the dynamical system in this example were carried out using our Singular library \texttt{polconslaw.lib}. The reader can use  \texttt{example6\_30.sg}, available in our Github and Zenodo Repositories, in order to carry on the computations automatically.
}
\end{example}

\begin{example}[BIOMD0000000629, \cite{LuedersSturmRadulescu:22}]\label{ex:bm629}
  The ODEs corresponding to Biomodel 629 in the ODEbase
  repository~\cite{LuedersSturmRadulescu:22} are the following:
  \begin{align*}
    \dot{x_1}=& -k_2x_1x_3 + k_3x_2, \\
    \dot{x_2}=&k_2x_1x_3 - k_3x_2 - k_4x_2x_4 + k_5x_5, \\
    \dot{x_3}=& -k_2x_1x_3 + k_3x_2, \\
    \dot{x_4}=& -k_4x_2x_4 + k_5x_5,\\
    \dot{x_5}=&k_4x_2x_4 - k_5x_5.
  \end{align*}

\cor{
Computations for generic branch using our Singular library \texttt{polconslaw.lib} can be found in the file  gbm629.txt, and computations for all branches using the Singular library \texttt{genpolconslaw.lib} can be carried on automatically via the file \texttt{bm629.sg}, which is available in our Github and Zenodo Repositories.
}
  A comprehensive Gr\"obner system for the steady state ideal consists
  of 10 branches. We consider first the generic branch and after that
  one nongeneric branch.
  One easily checks that the semi-algebraic set of the generic branch of the comprehensive
  Gr\"obner basis consists of $\{k_2 \ne 0\ , k_4 \ne 0 \}$, and the corresponding polynomial set consists of the five polynomials in the  right-hand sides of the ODE system. The syzygy module has a basis 
  \begin{gather*}
    \{\boldsymbol{f}_1= (-1,0,1,0,0), \ \boldsymbol{f}_2=
    (0,-1,-1,1,0), \
    \boldsymbol{f}_3= (0,0,0,1,1), \\
    \boldsymbol{f}_4= (k_2x_1x_3-k_4x_2x_4
    -k_3x_2+k_5x_5,k_2x_1x_3-k_3x_2,0,0,0)\}.
  \end{gather*}
 Obviously, $\vf_1$, $\vf_2$, $\vf_3$ are conservative, and integrating them leads to the linearly independent unconditional conservation laws
  \begin{equation*}
    \phi_1(\vk,\vx)= -x_1+x_3, \
    \phi_2(\vk,\vx)= -x_2-x_3+x_4, \
    \phi_3(\vk,\vx)=  x_4+x_5 .
  \end{equation*}
  The syzygy $\boldsymbol{f}_4(\vk,\vx)$ is not conservative
  as $\curl \boldsymbol{f}_4 \neq 0$. Multiplying  $\vf_1$, $\dots$, $\vf_4$ with all monomials such that the degree of the products is at most two, we obtain an $18$ dimensional $\QQ$-vectors space of syzygies and find nine conservation laws, where the six new found degree two conservation laws can be expressed as linear combinations of $\phi_1(\vk,\vx)$, $\phi_2(\vk,\vx)$ and $\phi_3(\vk,\vx)$ with coefficients in $\QQ[\vx,\vk]$.
  Enlarging the degree by one gives us a $64$ dimensional $\QQ$-vector space of syzygies. Our algorithm finds here $10$ new conservation conservation laws of degree three. Again these conservation laws are in the algebra generated by $\phi_1(\vk,\vx)$, $\phi_2(\vk,\vx)$ and $\phi_3(\vk,\vx)$. One may assume that it is enough to check the generators of the syzygy module to be conservative and that all further  conservation laws can be obtained as a linear combination of the generators that are conservative over $\QQ[\vk,\vx]$. But the last paragraph of Example~\ref{ex:pol-con-law-11} shows that this is not true.

\end{example}

As discussed earlier in Remark~\ref{rmk:uncond-law-elimination1} one
can compute unconditional conservation laws from the syzygy module by
an additional Gr\"obner basis computation.
On the other hand, unconditional conservation laws up to a given degree bound can be computed via ansatz, i.e., guessing the conservation laws with undetermined coefficients, and then obtaining coefficients by imposing the definition of first integrability. This method has been discussed in~\cite{goriely2001integrability} for computing polynomial first integrals of dynamical systems and it has been implemented in the automated prover tool Pegasus~\cite{sogokon2021pegasus}.  

Computations based on an ansatz avoid Gr\"obner basis (and syzygy) computations in the ring $\QQ[\vk,\vx]$.
However, computing polynomial conservation laws via ansatz would not necessarily lead to a basis in $\QQ[\vk,\vx]$ for the syzygy module, but rather to a basis for the $\RR$-vector space of conservative syzygies up to a given degree.
%
Moreover, the ansatz method can not be used to compute branches of parametric conservation laws. 
\cor{
Finally, using syzygies brings  an algebraic structure to the problem of first integral computation. 
Since this paper is the first step in this particular research direction, we are confident that there exists ample room for exploration regarding the algebraic characteristics of the conservation laws. In particular, we could anticipate the
possibility of finding a
 basis for the algebra of  conservation laws con be found directly, rather than 
 computing a $\QQ-$vector space basis for syzygies and subsequently filtering them via curl operator.
}

\cor{

\subsubsection{Testing syzygy computation on a biochemical models database}
We have systematically tested Algorithm~\ref{alg:pol-cons} using the models of the ODEbase repository~\cite{LuedersSturmRadulescu:22}.
Selected results and some interpretations based on the benchmark table are presented below, whereas the comprehensive report of the results is provided as Supplementary Materials.

In Table~\ref{tab:cgs}, we present the result of our computations up to degree $4$ for four biomodels from the ODEbase, including model 629 considered in Example~\ref{ex:bm629}. Running times are in seconds, with a limit of 4 hours. Columns $n$ and $r$ represent the number of variables and parameters in the model, respectively. While computations for model 629 have finished in quite a short time, for models 710 and 1053, we were only able to finish the computations for a few branches, and for model 1054, computations for none of the branches finished in four hours. This is due to several reasons. First of all comprehensive Gr\"obner bases are known to be hard to compute, due to the huge number of branches that they may produce. Also syzygy computations can be hard as their computation is equivalent to computing Gr\"obner bases. Moreover, computing a vector space basis for the syzygies has exponential complexity. 

An interesting point in this table is that in several rows, the number of linear and polynomial conservation laws are related. This is because most of the polynomial conservation laws computed are in the algebra generated by the linear conservation laws. As noted in Proposition~\ref{prop:alg-cons1}, conservation laws form an algebra, hence, for algebraically independent conservation laws $\phi_1,\dots,\phi_m$, every polynomial of the form $\phi_1^{n_1}\phi_2^{n_2}\ldots \phi_m^{n_m}$, where $n_1,\dots,n_m \in \NN$ is a conservation law. There are $\binom{m+d}{d}-1$ many such polynomials of degree at most $d$. For $d=4$ and $m=1,2,3,4,5,6$, this number is  $4,14,34,69,125,209$, respectively. We obtain exactly those numbers (in the column "$\#$ cons. laws" that presents the total number of conservation laws) in all rows, except for the gray rows that include more polynomial conservation laws. Those are polynomial conservation laws that are algebraically independent of the linear conservation laws. 

Similarly, in Examples~\ref{ex:tiny-ex1} and~\ref{ex:pol-con-law-11}
one can see that the  polynomial conservation laws are generated by the lowest degree (not necessarily linear) conservation laws.  This suggests that the $\QQ$-algebra of polynomial conservation laws is finitely generated and that the generators include curl-zero elements of a minimal basis of the syzygy module. Furthermore, the generators other than those coming from a basis of the syzygy module are rare. This subject will be further developed elsewhere.


{\footnotesize
\begin{longtable}{|r|r|r|r|r|r|r|r|r|}
\caption{Conservation Laws in Branches of CGS
\label{tab:cgs}}\\
    \hline
   Bio- & $n$ & $r$ & running  & \# branches & branch & \# cons.& \# lin. cons. & \# pol. cons.\\
    Model &  &  & time &  & number  &  laws &   laws & laws \\
    \hline

629  & 5 & 8& 872  & 10 & 1 & 34 & 3 & 31 \\ \cline{6-9}    
&&&&& 2 & 34 & 3 & 31 \\ \cline{6-9}
&&&&& 3 & 69 & 4 & 65 \\ \cline{6-9}
&&&&& 4 & 69 & 4 & 65 \\ \cline{6-9}
&&&&& 5 & 125 & 5 & 120 \\ \cline{6-9}
&&&&& 6 & 34 & 3 & 31 \\ \cline{6-9}
&&&&& 7 & 69 & 4 & 65 \\ \cline{6-9}
&&&&& 8 & 34 & 3 & 31 \\ \cline{6-9}
&&&&& 9 & 69 & 4 & 65 \\ \cline{6-9}
&&&&& 10 & 34 & 3 & 31 \\ \hline

710  & 7 & 18& 14401  & 5891 & 3 & 4 & 1 & 3 \\ \cline{6-9}
&&&&& 4 & 4 & 1 & 3 \\ \cline{6-9}
&&&&& 5 & 14 & 2 & 12 \\ \cline{6-9}
&&&&& 6 & 14 & 2 & 12 \\ \cline{6-9}
&&&&& 7 & 34 & 3 & 31 \\ \cline{6-9}
&&&&& 8 & 34 & 3 & 31 \\ \cline{6-9}
&&&&& 9 & 34 & 3 & 31 \\ \cline{6-9}
&&&&& 10 & 69 & 4 & 65 \\ \cline{6-9}
&&&&& 11 & 125 & 5 & 120 \\ \cline{6-9}
&&&&& 12 & 125 & 5 & 120 \\ \hline

1053  & 6 & 7& 14401  & 28 & 3 & 4 & 1 & 3 \\ \cline{6-9}
&&&&& 4 & 69 & 4 & 65 \\ \cline{6-9}
&&&&& 5 & 125 & 5 & 120 \\ \cline{6-9}
&&&&& 6 & 209 & 6 & 203 \\ \cline{6-9}
&&&&& 7 & 209 & 6 & 203 \\ \cline{6-9}
&&&&& 8 & 4 & 1 & 3 \\ \cline{6-9}
&&&&& 9 & 4 & 1 & 3 \\ \cline{6-9}
 &&&&&   10 & 21 & 2 & 19 \\ \cline{6-9}
&&&&& 11 & 34 & 3 & 31 \\ \cline{6-9}
&&&&& 12 & 69 & 4 & 65 \\ \cline{6-9}
&&&&& 13 & 69 & 4 & 65 \\ \cline{6-9}
&&&&& 15 & 4 & 1 & 3 \\ \cline{6-9}
&&&&& 19 & 4 & 1 & 3 \\ \cline{6-9}
 &&&&&   21 & 17 & 2 & 15 \\ \cline{6-9}
 &&&&&   22 & 38 & 3 & 35 \\ \hline

1054  & 7 & 12& 14401  & 1379 &  &  & &  \\ \hline 
\end{longtable}
}

The discussion above leads us to consider computing only conservation laws that arise from conservative generators of a Gr\"obner basis of the syzygy module, as a heuristic. Such computations would naturally be much faster than computing a basis for syzygies as a vector space. Moreover, considering only the generic branch and using the trick mentioned in Remark~\ref{rmk:non-CGB1}, one avoids computing a comprehensive Gr\"obner basis, which further speeds up the computations. The latter heuristic has been considered in Table~\ref{tab:syz-mod-generic}, carrying on computations for the same biomodels as in Table~\ref{tab:cgs}. The computations took only one second per model, and we found two linear conservation laws for model 1054, for which we had run out of time during previous computations.

{\footnotesize
\begin{longtable}{|r|r|r|r|r|r|r|r|r|} 
  \caption{Computations for Syzygy Module Basis (Generic Branch)
    \label{tab:syz-mod-generic}} \\
  \hline
  BioModel & $n$ & $r$ & running  & \# syzygy mod. & \# cons. laws & \# lin. cons. & \# pol.  cons.\\
 &  &   &  time  & basis &     & laws &  laws \\
      \hline  
      629 & 5&8 & 1 & 4 & 3 & 3 & 0 \\ \hline
      710 & 7&18 & 1 & 28 & 0 & 0 & 0 \\ \hline
      1053 & 6&7 & 1 & 15 & 0 & 0 & 0 \\ \hline
      1054 & 7&12 & 1 & 21 & 2 & 2 & 0 \\ \hline
      \end{longtable}
      }

}
\cor{Figure~\ref{fig:figure1} illustrates how the number of computed syzygies and conservation laws increases with respect to the number of variables of the model. As expected, the  $Q$-vector space basis calculation fails more often, but when successful it leads to much more polynomial conservation laws, while syzygy module basis calculation produces algebraically independent conservation laws much faster.}

\begin{figure}[h!]
   \centering
   \includegraphics[width=0.7\linewidth]{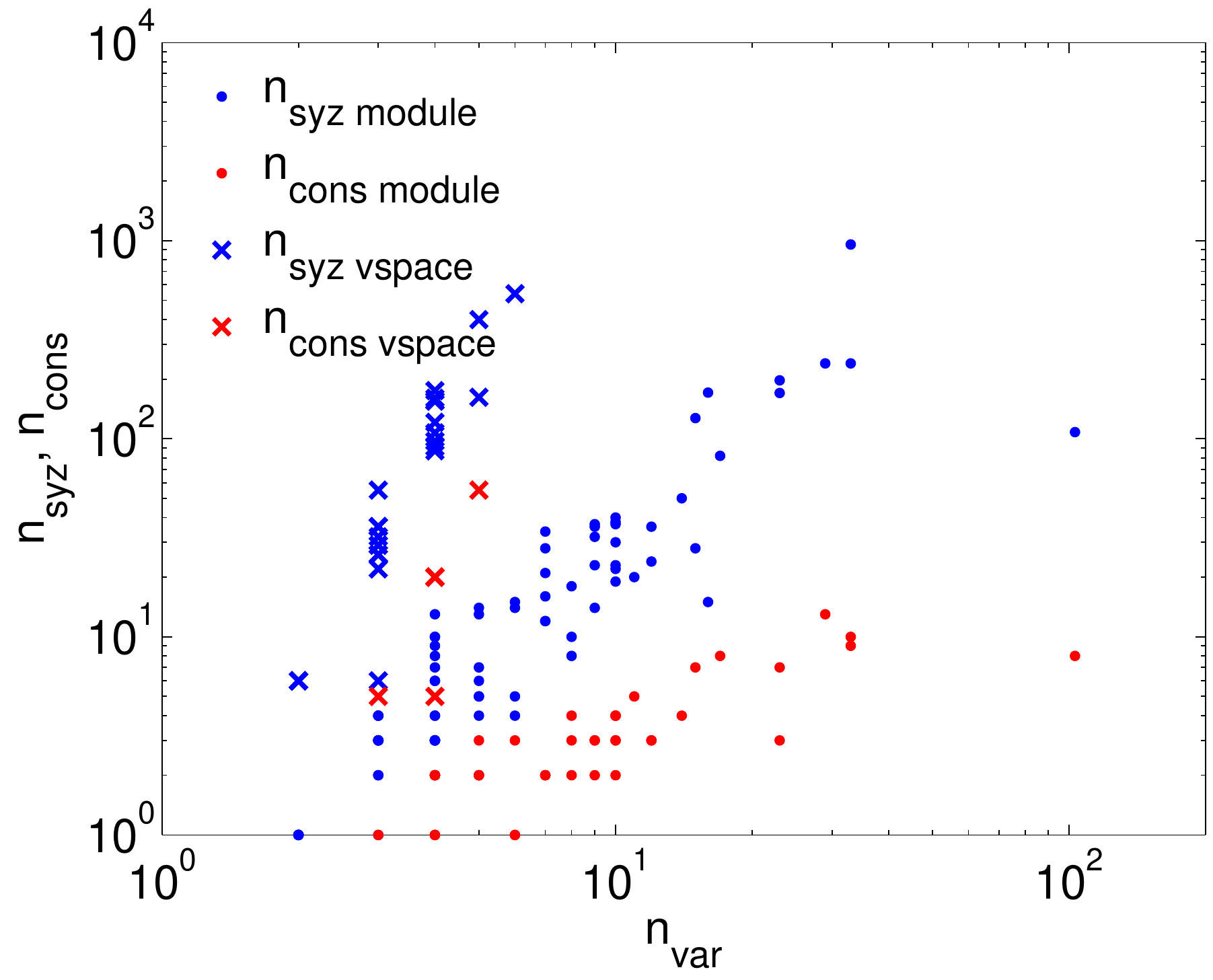}
   \caption{Computation of polynomial conservation laws arising from the generic branch of comprehensive Gr\"obner system for CRN models from the database Biomodels by two methods: from the generators of the syzygy module and
   from the basis of $\QQ-$vector space of syzygies. 
   $n_{var}$ is the number of variables and ODEs in the CRN model. The number of syzygies and polynomial 
   conservation laws increases with the number of variables. When the computation of the  $\QQ-$vector space basis succeeds, the number of computed polynomial conservation laws is sensitively larger. 
   \label{fig:figure1}
   }
\end{figure}






\section*{Availability} 
ODEbase is available at \url{https://www.odebase.org/}. The Singular libraries \texttt{polconslaw.lib} and \texttt{genpolconslaw.lib} are available at Zenodo under the record \url{https://zenodo.org/record/7474523} and through Github at the following url \url{https://github.com/oradules/polynomial_conservation_laws}.

\section*{Acknowledgements}
\cor{We thank the referees for their thoughtful comments that helped us improve our manuscript. }
OR thanks Werner Seiler for very insightful discussions and for his
hospitality \cor{at} the University of Kassel where the ideas about
approximate conservation laws were initiated.  HR would like to thank
Deepak Kapur and Amir Hashemi for the discussions about comprehensive
Gr\"obner bases and syzygy computations.    The project
SYMBIONT owes a lot to Andreas Weber who sadly left us in 2020, but
who is still present in our memories.


\newpage
\section*{APPENDIX: Polynomial conservation laws of polynomial
  CRNs in the Biomodels database}

{\it Biomodels} \cite{BioModels2015a} is a comprehensive database of
CRN models. The models of this database were mapped onto ODEbase, a
canonical source of symbolic computation input
\cite{luders2022odebase}.  We have used ODEbase for making experiments
using our implementations and for testing the scalability of our
computational methods for polynomial conservation laws.  All of the computations were carried on a single core of an Intel i5-9400, 2.9 GHz processor, using Ubuntu 20.04 in 64-bit mode.

We have run the implementation of Algorithm~\ref{alg:pol-cons} 
in two different ways:
\begin{enumerate}
    \item First we considered only the generating set of the module of syzygies of the steady state ideals. This means not considering lines 8 and 9 of the algorithm. As explained earlier, this leads to computing some $Q[\vx]$-linearly independent conservation laws, but not all possible $Q-$linearly independent conservation laws. Table~\ref{tab:table1} shows the result of the computations only for the generic branch of a comprehensive Gröbner system (as in Remark~\ref{rmk:non-CGB1}) and Table~\ref{tab:table2} shows the result of the computations for all branches of a comprehensive Gröbner basis. The running time for each biomodel was at most one hour. 
    \item Secondly, we did a full implementation of our algorithm, in order to compute a basis for the $Q$-vector space of conservation laws up to degree 4 for the generic branch of a comprehensive Gröbner basis, and up to degree 5 for all branches of a comprehensive Gröbner basis. The results can be found in Tables~\ref{tab:new-table1} and \ref{tab:new-table2}, consecutively. The running time for each biomodel was at most four hours. 
\end{enumerate}


\cor{Below is an interpretation of the four tables produced out of our experiments.
}
\begin{itemize}
    \item Table~\ref{tab:table1} shows the result of computations of the generic branch, performed using our Singular library \texttt{genpolconslaw.lib}, only computing the basis for the syzygy module and not considering the vector space of syzygies up to a given degree. This results in computing some conservation laws that correspond to the conservative elements of the syzygy module basis. 
    The computations either finished within a few seconds or did not finish after one hour. This is due to the EXPSPACE complexity of the syzygy computations (which is the same as Gröbner bases basis computation complexity). Another interesting observation concerning Table~\ref{tab:table1} is the large proportion of irrotational syzygies. For example, for Biomodel 270, there are 956 syzygies (computed just within a second), out of which only 10 syzygies are irrotational. 
    Furthermore, \cor{with this method we only found linear unconditional conservation laws in the generic branch. We will see below that this is because the computed generators of the syzygy module correspond to 
    linear conservation laws, which does not exclude the existence of other polynomial conservation laws. }
    \item Table~\ref{tab:table2} shows the result of parametric polynomial conservation law computations using our Singular library \texttt{polconslaw.lib}.  As explained in Section~\ref{sec:syzygies}, parametric polynomial conservation laws result from non-generic branches. Interestingly, most of the parametric conservation laws that are computed within one hour from the non-generic branches are polynomial. Again, we must clarify that we only considered a basis of the syzygy module, but not a basis for the $Q$-vector space of syzygies up to a given degree. This results in computing only those conservation laws that arise \cor{from} the syzygy module basis elements.
    
    We must emphasise that in Table~\ref{tab:table2}, the parametric conservation laws were computed using reduced Gr\"obner bases, which may lead to removing some linear conservation laws in the generic branch. For instance, Biomodel 6 has two syzygies in the generic branch, and one of them is conservative, which leads to a linear conservation law ($x_1=const.$, resulting from $\D{x_1}{t} = 0$ and the syzygy $(1,0,0,0)$).  In fact, this conservation law is not parametric, and therefore the reduction has no effect on the final result. Our calculation \cor{in Table~\ref{tab:table1}} of unconditional conservation laws does not employ reduction and detects the linear conservation law $x_1 = const.$ of the Biomodel 6.
    
    
    As expected, computing conservation laws in the generic branch is much faster than computing parametric conservation laws. Most of the
    computations in the generic branch are finished within one hour,
    but that is not the case for parametric polynomial conservation law
    computations. This is because when the number of branches of CGS is
    high, computing syzygies for all of those branches can be expensive.
    
    
    
    \item Table~\ref{tab:new-table1} shows the result of computing a basis for the $Q$-vector space of conservation laws up to degree 5 in the generic branch.

    \item Table~\ref{tab:new-table2} presents the result of computing a $Q$-vector space basis for syzygies up to degree 4, in each branch of a comprehensive Gröbner system. Unlike Table~\ref{tab:table2}, we do not compute the reduced Gröbner basis for each branch, but rather we specialise the generators of the input steady state ideal via the conditions describing the corresponding algebraic set. This allows us to recover those parametric conservation laws that were lost during the computations in Table~\ref{tab:table2}. 
\end{itemize}

In comparison with computing just a basis for the syzygy module in Tables~\ref{tab:table1} and \ref{tab:table2}, computing a basis for the $Q$-vector space of syzygies up to a given degree in Tables~\ref{tab:new-table1} and \ref{tab:new-table2} requires extra computations in lines 8 and 9 of Algorithm~\ref{alg:pol-cons}. 
The extra computations have exponential complexity in the number of variables and the degree bound, and therefore, may need considerably longer time for computations. This can be seen from those biomodels in Tables~\ref{tab:new-table1} and \ref{tab:new-table2} for which the computations did not finish in four hours, while in Tables~\ref{tab:table1} and \ref{tab:table2} the computations finished in one hour. However, theoretically, the dominant complexity factor will be the syzygy computation, which is EXPSPACE-complete, rather than the exponential factor.

As expected, when the computations finished in four hours, Tables~\ref{tab:new-table1} and \ref{tab:new-table2} contain more polynomial conservation laws than Tables~\ref{tab:table1} and \ref{tab:table2}. 
In order to obtain a polynomial conservation law quickly, we recommend using syzygy modules, while for computing a basis for the $Q$-vector space of conservation laws, we recommend using the $Q$-vector space of conservative syzygies. 

We note that our implemented libraries are prototype implementations of our algorithm and there is a lot of room for improvement in terms of speed and memory. For instance, one can ask the algorithm to first compute the conservative syzygies obtained from a basis of the syzygy module and then continue with computing the $Q$-vector space of conservative syzygies. Also, our implementation repeats the construction of a $Q$-vector space of syzygies up to the given degree. An improvement would be to compute such a basis once and reuse it.

Our experiments also show that all the polynomial conservation laws computed in Table~\ref{tab:new-table1}  are algebraically dependent on the linear conservation laws. 
Indeed, $r$ independent linear conservation 
laws generate the linearly independent polynomial
conservation laws $\phi_1^{n_1}\phi_2^{n_2}\ldots \phi_r^{n_r}$ of degree at most $d$, 
where $n_1,\ldots,n_r$ satisfy
\begin{equation}
1 \leq n_1+n_2+\ldots+n_r \leq d, n_i \in \ZZ_{\geq 0}, 1\leq i\leq r. \label{inequalities}
\end{equation}
\cor{
The inequalities \eqref{inequalities} have 
$\binom{r+d}{d}-1$ distinct solutions.
It follows that
$r$ linearly independent linear conservation laws generate 
$\binom{r+d}{d}-1$ linearly independent polynomial conservation laws. 

There are $5,20,55$ such generated polynomial conservation laws for $d=5$ and $r=1,2,3$ linear conservation laws, respectively. Comparing these numbers with the dimension  of the linear space of polynomial conservation laws, given in Tables~\ref{tab:new-table1}, we realize that all the conservation laws in this table are generated by linear conservation laws. 

Interestingly, some non-generic branches computed in Table~\ref{tab:new-table2} have more polynomial conservation laws than those generated by the linear conservation laws, meaning that some polynomial conservation laws are algebraically independent of the linear ones. Indeed, 
$4,14,34,69,125,209$ conservation laws are generated by linear conservation laws for $d=4$ and $r=1,2,3,4,5,6$, respectively. Whenever we find larger numbers of conservation laws, some polynomial conservation laws are independent of the linear ones. The rows colored in gray in Table~\ref{tab:new-table2} contain polynomial conservation laws, algebraically independent of the linear ones. 
}


Similarly, Examples~\ref{ex:tiny-ex1},~\ref{ex:pol-con-law-11}
of the main text show that  polynomial conservation laws are generated by lowest degree (not necessarily linear) conservation laws.  This suggests that the $\QQ$-algebra of polynomial conservation laws is finitely generated and that the generators 
include curl-zero elements of a minimal basis of the syzygy module. Furthermore, the generators other than those coming from a basis of the syzygy module are rare. This subject will be further developed elsewhere.






{\footnotesize
\begin{longtable}{|r|r|r|r|r|r|r|r|r|} 
  \caption{Conservation laws of CRNs of the database Biomodels
  resulting from the syzygy \textbf{module} basis in the generic branch of comprehensive 
  Gr\"obner system. 
  The columns represent the biomodel number, running time in seconds, the number of
  the generators of the syzygy module of the corresponding steady state ideal, the number of
  conservation laws obtained from those syzygies (equivalently, the
  number of irrotational syzygies in the module basis), the number of linear conservation
  laws, and finally the number of polynomial conservation laws for the
  biomodel, respectively. Blank rows show that the computations did
  not finish in 3600 seconds. The major part of the computation time in the
  algorithm is spent on computing syzygies, and for those CRNs
  that the computations did not finish in 3600 seconds, in fact the
  syzygy computation did not finish in that time period (on a single core of an Intel i5-9400, 2.9 GHz processor).
  \cor{
$n$ and $r$ are the numbers of variables and parameters of the models, respectively.  
  }
  \label{tab:table1}} \\
  \hline
  Bio & n&r & running  & \# syzygy  & \# cons. & \# lin.  & \# pol.  \\
 Model&  &   &  time  & mod. basis & laws      & cons. laws & cons. laws \\
      \hline  
2 & 13&37 & 3601 &&&& \\ \hline
6 & 4&7 & 1 & 2 & 1 & 1 & 0 \\ \hline
28 & 16&31  & 3601 &&&& \\ \hline
30 & 18&36 & 3601 &&&& \\ \hline
38 & 17&29 & 6 & 82 & 8 & 8 & 0 \\ \hline
40 & 5&9 & 1 & 6 & 2 & 2 & 0 \\ \hline
46 & 16&22 & 3601 &&&& \\ \hline
69 & 10&25 & 1 & 38 & 4 & 4 & 0 \\ \hline
72 & 7&11 & 1 & 12 & 2 & 2 & 0 \\ \hline
80 & 10&17 & 1 & 22 & 4 & 4 & 0 \\ \hline
82 & 10&17 & 1 & 23 & 4 & 4 & 0 \\ \hline
85 & 17&42 & 3601 &&&& \\ \hline
86 & 17&56 & 3601 &&&& \\ \hline
92 & 4&7 & 1 & 3 & 2 & 2 & 0 \\ \hline
99 & 7&16 & 1 & 34 & 0 & 0 & 0 \\ \hline
102 & 13&42 & 3601 &&&& \\ \hline
103 & 17&50 & 3601 &&&& \\ \hline
108 & 9&19 & 3601 &&&& \\ \hline
150 & 4&5 & 1 & 3 & 2 & 2 & 0 \\ \hline
156 & 3&8 & 1 & 3 & 0 & 0 & 0 \\ \hline
159 & 3&8 & 1 & 3 & 0 & 0 & 0 \\ \hline
198 & 12&20 & 1 & 24 & 3 & 3 & 0 \\ \hline
199 & 15&18 & 1 & 28 & 7 & 7 & 0 \\ \hline
200 & 22&54 & 3601 &&&& \\ \hline
205 & 194&351 & 3603 &&&& \\ \hline
243 & 23&19 & 1 & 197 & 3 & 3 & 0 \\ \hline
252 & 4&10 & 1 & 7 & 0 & 0 & 0 \\ \hline
257 & 11&23 & 3601 &&&& \\ \hline
270 & 33&49 &  1 & 956 & 10 & 10 & 0 \\ \hline
282 & 6&8 & 1 & 5 & 3 & 3 & 0 \\ \hline
283 & 4&5 & 1 & 3 & 2 & 2 & 0 \\ \hline
292 & 6&8 & 1 & 4 & 1 & 1 & 0 \\ \hline
314 & 12&15 & 1 & 36 & 3 & 3 & 0 \\ \hline
315 & 20&52 & 3601 &&&& \\ \hline
335 & 34&54 & 3601 &&&& \\ \hline
357 & 9&15 & 5 & 23 & 2 & 2 & 0 \\ \hline
359 & 9&19 & 1 & 37 & 3 & 3 & 0 \\ \hline
360 & 9&19 & 1 & 32 & 3 & 3 & 0 \\ \hline
361 & 8&13 & 1 & 18 & 3 & 3 & 0 \\ \hline
362 & 34&56 & 3601 &&&& \\ \hline
363 & 4&6 & 1 & 4 & 1 & 1 & 0 \\ \hline
364 & 14&24 & 3601 &&&& \\ \hline
365 & 30&15 & 3601 &&&& \\ \hline
405 & 8&7 & 1 & 8 & 2 & 2 & 0 \\ \hline
430 & 27&52 & 3601 &&&& \\ \hline
431 & 27&52 & 3601 &&&& \\ \hline
447 & 13&24 & 3601 &&&& \\ \hline
475 & 23&34 & 1 & 170 & 7 & 7 & 0 \\ \hline
483 & 8&13 & 1 & 10 & 4 & 4 & 0 \\ \hline
530 & 10&18 &  3 & 30 & 3 & 3 & 0 \\ \hline
539 & 6&13 & 1 & 14 & 1 & 1 & 0 \\ \hline
552 & 3&5 & 1 & 1 & 0 & 0 & 0 \\ \hline
553 & 3&5 & 1 & 1 & 0 & 0 & 0 \\ \hline
609 & 5&12 & 1 & 7 & 0 & 0 & 0 \\ \hline
614 & 1&4 & 1 & 0 & 0 & 0 & 0 \\ \hline
629 & 5&8 & 1 & 4 & 3 & 3 & 0 \\ \hline
647 & 11&17 & 1 & 20 & 5 & 5 & 0 \\ \hline
651 & 29&48 & 1 & 240 & 13 & 13 & 0 \\ \hline
667 & 103&91 & 4 & 108 & 8 & 8 & 0 \\ \hline
676 & 14&25 & 1 & 50 & 4 & 4 & 0 \\ \hline
679 & 4&12 & 1 & 3 & 0 & 0 & 0 \\ \hline
680 & 4&12 & 1 & 3 & 0 & 0 & 0 \\ \hline
687 & 15&17 & 1 & 127 & 0 & 0 & 0 \\ \hline
688 & 16&21 & 194 & 171 & 0 & 0 & 0 \\ \hline
707 & 5&11 & 1 & 13 & 0 & 0 & 0 \\ \hline
710 & 7&18 & 1 & 28 & 0 & 0 & 0 \\ \hline
742 & 2&7 & 1 & 1 & 0 & 0 & 0 \\ \hline
745 & 5&14 & 3601 &&&& \\ \hline
747 & 33&54 & 1 & 240 & 9 & 9 & 0 \\ \hline
748 & 4&9 & 1 & 10 & 0 & 0 & 0 \\ \hline
755 & 9&11 & 1 & 14 & 3 & 3 & 0 \\ \hline
758 & 2&5 & 1 & 1 & 0 & 0 & 0 \\ \hline
780 & 4&16 & 1 & 9 & 0 & 0 & 0 \\ \hline
781 & 3&11 & 1 & 4 & 0 & 0 & 0 \\ \hline
782 & 2&7 &  1 & 1 & 0 & 0 & 0 \\ \hline
783 & 3&9 & 1 & 4 & 0 & 0 & 0 \\ \hline
793 & 2&5 & 1 & 1 & 0 & 0 & 0 \\ \hline
795 & 2&7 & 1 & 1 & 0 & 0 & 0 \\ \hline
815 & 2&7 & 1 & 1 & 0 & 0 & 0 \\ \hline
827 & 10&18 & 1 & 37 & 3 & 3 & 0 \\ \hline
854 & 4&11 & 1 & 4 & 1 & 1 & 0 \\ \hline
868 & 5&10 & 1 & 5 & 2 & 2 & 0 \\ \hline
870 & 7&17 & 1 & 16 & 2 & 2 & 0 \\ \hline
875 & 4&9 & 1 & 6 & 0 & 0 & 0 \\ \hline
880 & 4&15 & 1 & 10 & 0 & 0 & 0 \\ \hline
882 & 3&6 &  1 & 2 & 0 & 0 & 0 \\ \hline
886 & 5&19 & 3601 &&&& \\ \hline
887 & 4&11 & 1 & 8 & 0 & 0 & 0 \\ \hline
894 & 3&9 & 1 & 3 & 0 & 0 & 0 \\ \hline
905 & 5&16 & 1 & 14 & 0 & 0 & 0 \\ \hline
906 & 3&10 & 1 & 3 & 0 & 0 & 0 \\ \hline
916 & 5&7 & 1 & 5 & 2 & 2 & 0 \\ \hline
922 & 3&10 & 1 & 3 & 0 & 0 & 0 \\ \hline
932 & 4&8 & 1 & 6 & 0 & 0 & 0 \\ \hline
934 & 16&24 & 1 & 15 & 0 & 0 & 0 \\ \hline
940 & 20&43 & 3601 &&&& \\ \hline
951 & 34&72 & 3601 &&&& \\ \hline
957 & 4&5 & 1 & 3 & 2 & 2 & 0 \\ \hline
968 & 10&15 & 1 & 19 & 4 & 4 & 0 \\ \hline
987 & 9&12 & 1 & 36 & 0 & 0 & 0 \\ \hline
1004 & 10&21 & 1 & 40 & 2 & 2 & 0 \\ \hline
1021 & 7&11 & 1 & 12 & 0 & 0 & 0 \\ \hline
1024 & 2&5 & 1 & 1 & 0 & 0 & 0 \\ \hline
1031 & 3&5 & 1 & 3 & 0 & 0 & 0 \\ \hline
1035 & 4&12 & 3 & 13 & 0 & 0 & 0 \\ \hline
1037 & 2&7 & 1 & 1 & 0 & 0 & 0 \\ \hline
1038 & 3&11 & 1 & 4 & 0 & 0 & 0 \\ \hline
1045 & 3&3 & 1 & 2 & 1 & 1 & 0 \\ \hline
1053 & 6&7 & 1 & 15 & 0 & 0 & 0 \\ \hline
1054 & 7&12 & 1 & 21 & 2 & 2 & 0 \\ \hline
\end{longtable}
}

\clearpage
{\footnotesize
\begin{longtable}{|r|r|r|r|r|r|r|r|r|}
  \caption{Parametric conservation laws of CRNs of the database Biomodels resulting from computations only on a basis of the syzygy \textbf{module} in each branch, using our Singular library \texttt{polconslaw.lib}.
  The columns represent Biomodel number, the running time in seconds,
  number of branches of the corresponding comprehensive Gr\"obner
  system computed within the running time (3600 seconds
  on a single core of an Intel i5-9400, 2.9 GHz processor), number of the
  branch, the total number of conservation laws computed  from the syzygy module basis elements in the
  branch, number of linear conservation laws computed, and finally,
  the number of polynomial conservation laws computed in the
  branch.
  Rows whose running time was above 3600 seconds and did not finish
  computing conservation laws within that branch, have blank entries
  under the columns of number of syzygies, and (linear/polynomial) conservation
  laws.
Rows whose running time was below 3600 seconds but no conservation laws were found, 
  have blank entries for the number of (linear/polynomial)
  conservation laws.
For all the other branches, 
the branches leading to conservation laws are listed in the order they are found, with 
the corresponding number of conservation laws. 
\cor{To simplify the table we regrouped the branches with the same numbers of conservation
laws and displayed only the numbers per subgroup.} 
Branches are missing because they were not
computed in 3600 seconds or because they do not lead to conservation laws (for instance 
branches 1-18 of Biomodel 40 lead to no conservation law). 
\cor{$n$ and $r$ are the numbers of variables and parameters of the models, respectively. }
  \label{tab:table2}
  }\\
  \hline
  Bio- & n & r & running   & \# branches & \# branches & \# cons. & \# lin. & \# pol. \\
  Model &  &   & time     &             & per subgroup &  laws &  cons. laws & cons. laws \\
  \hline        
2 & 13&37 & 3601 & 0 & &&& \\ \hline
6 & 4&7 & 1 & 7 & & & & \\ \hline
28 & 16&31  & 3601 & 0 & &&& \\ \hline
30 & 18&36 & 3601 & 0 & &&& \\ \hline
38 & 17&29 &3601 & 0 & &&& \\ \hline
40 & 5&9 & 1 & 35 & 
16 & 1 & 0 & 1 \\ \cline{6-9} &&
&&& 1 & 2 & 0 & 2 \\ \hline
46 & 16&22 & 3601 & 0 & &&& \\ \hline
69 & 10&25 & 3601 & 12354 & 1 & 12 & 2 & 10 \\ \cline{6-9} &&
&&& 1 & 19 & 4 & 15 \\ \cline{6-9} &&
&&& 1 & 25 & 5 & 20 \\ \cline{6-9} &&
&&& 1 & 25 & 5 & 20 \\ \hline
72 & 7&11 & 1 & 74 & &&& \\ \hline
80 & 10&17 & 3619 & 264 & 1 & 2 & 0 & 2 \\ \cline{6-9} &&
&&& 27 & 4 & 0 & 4 \\ \hline
82 & 10&17 & 3620 & 344 & 28 & 1 & 0 & 1 \\ \hline
85 & 17&42 & 3606 & 0 & &&& \\ \hline
86 & 17&56 & 3606 & 0 & &&& \\ \hline
92 & 4&7 & 1 & 10 & &&& \\ \hline
99 & 7&16 & 3620 & 3783 & 2 & 2 & 0 & 2 \\ \cline{6-9} &&
&&& 1 & 8 & 0 & 8 \\ \cline{6-9} &&
&&& 1 & 12 & 0 & 12 \\ \cline{6-9} &&
&&& 1 & 21 & 0 & 21 \\ \cline{6-9} &&
&&& 1 & 26 & 0 & 26 \\ \cline{6-9} &&
&&& 16 & 32 & 0 & 32 \\ \cline{6-9} &&
&&& 6 & 34 & 0 & 34 \\ \hline
102 & 13&42 & 3606 & 0 & &&& \\ \hline
103 & 17&50 & 3601 & 0 & &&& \\ \hline
108 & 9&19 &3601 & 0 & &&& \\ \hline
150 & 4&5 &1 & 10 & &&& \\ \hline
156 & 3&8 &1 & 25 & 16-19 & 2 & 0 & 2 \\ \cline{6-9} &&
&&& 20-25 & 3 & 0 & 3 \\ \hline
159 &3&8 & 1 & 36 & 19-36 & 2 & 0 & 2 \\ \hline
198 &12&20 & 3481 & 1317 & 1 & 4 & 0 & 4 \\ \cline{6-9} &&
&&& 1 & 8 & 0 & 8 \\ \cline{6-9} &&
&&& 1 & 10 & 0 & 10 \\ \cline{6-9} &&
&&& 1 & 12 & 0 & 12 \\ \cline{6-9} &&
&&& 24 & 13 & 0 & 13 \\  \hline
199 &15&18 & 3610 & 942 & 1 & 2 & 0 & 2 \\ \cline{6-9} &&
&&& 1 & 2 & 0 & 2 \\ \cline{6-9} &&
&&& 1 & 7 & 0 & 7 \\ \cline{6-9} &&
&&& 25 & 9 & 0 & 9 \\  \hline
200 &22&54 & 3601 & 0 & &&& \\ \hline
205 &194&351 & 3601 & 0 & &&& \\ \hline
243 &23&19 & 3606 & 55296 & 1 & 161 & 0 & 161 \\ \cline{6-9} &&
&&& 1 & 300 & 0 & 300 \\ \cline{6-9} &&
&&& 1 & 438 & 0 & 438 \\ \cline{6-9} &&
&&& 1 & 556 & 0 & 556 \\ \cline{6-9} &&
&&& 1 & 671 & 0 & 671 \\ \cline{6-9} &&
&&& 1 & 795 & 0 & 795 \\ \cline{6-9} &&
&&& 1 & 913 & 0 & 913 \\ \cline{6-9} &&
&&& 1 & 988 & 0 & 988 \\ \cline{6-9} &&
&&& 1 & 1040 & 0 & 1040 \\ \cline{6-9} &&
&&& 1 & 1068 & 0 & 1068 \\ \cline{6-9} &&
&&& 1 & 1089 & 0 & 1089 \\ \cline{6-9} &&
&&& 1 & 1104 & 0 & 1104 \\ \cline{6-9} &&
&&& 1 & 1114 & 0 & 1114 \\ \cline{6-9} &&
&&& 1 & 1120 & 0 & 1120 \\ \cline{6-9} &&
&&& 5 & 1123 & 0 & 1123 \\ \cline{6-9} &&
&&& 2 & 1124 & 0 & 1124 \\ \cline{6-9} &&
&&& 5 & 1125 & 0 & 1125 \\  \hline
252 & 4&10 & 3609 & 175 & 9 & 1 & 0 & 1 \\ \cline{6-9} &&
&&& 9 & 3 & 0 & 3 \\ \cline{6-9} &&
&&& 3 & 4 & 0 & 4 \\ \cline{6-9} &&
&&& 6 & 5 & 0 & 5 \\  \hline
257 & 11&23 & 3601 & 0 & &&& \\ \hline
270 & 33&49 & 3601 & 0 & &&& \\ \hline
282 &  6&8 &1 & 18 & &&& \\ \hline
283 & 4&5 &1 & 2 & &&& \\ \hline
292 & 6&8 &1 & 12 & &&& \\ \hline
314 & 12&15 &3608 & 1036 & 1 & 10 & 0 & 10 \\ \cline{6-9} &&
&&& 1 & 17 & 0 & 17 \\ \cline{6-9} &&
&&& 1 & 23 & 0 & 23 \\ \cline{6-9} &&
&&& 1 & 28 & 0 & 28 \\ \cline{6-9} &&
&&& 1 & 33 & 0 & 33 \\ \cline{6-9} &&
&&& 1 & 37 & 0 & 37 \\ \cline{6-9} &&
&&& 17 & 39 & 0 & 39 \\ \cline{6-9} &&
&&& 5 & 40 & 0 & 40 \\ \hline
315 &20&52 & 3601 & 0 & &&& \\ \hline
335 &34&54 & 3604 & 0 & &&& \\ \hline
357 & 9&15 & 3607 & 1953 & 1 & 1 & 0 & 1 \\ \cline{6-9} &&
&&& 1 & 2 & 0 & 2 \\ \cline{6-9} &&
&&& 1 & 5 & 0 & 5 \\ \cline{6-9} &&
&&& 1 & 11 & 2 & 9 \\ \cline{6-9} &&
&&& 1 & 17 & 4 & 13 \\ \cline{6-9} &&
&&& 19 & 19 & 4 & 15 \\ \cline{6-9} &&
&&& 3 & 20 & 4 & 16 \\  \hline
359 &9&19 & 3605 & 9741 & 1 & 1 & 0 & 1 \\ \cline{6-9} &&
&&& 1 & 2 & 0 & 2 \\ \cline{6-9} &&
&&& 1 & 6 & 0 & 6 \\ \cline{6-9} &&
&&& 1 & 10 & 3 & 7 \\ \cline{6-9} &&
&&& 1 & 14 & 4 & 10 \\ \cline{6-9} &&
&&& 1 & 18 & 4 & 14 \\ \cline{6-9} &&
&&& 1 & 21 & 5 & 16 \\ \cline{6-9} &&
&&& 20 & 24 & 6 & 18 \\  \hline
360 &9&19 & 3606 & 7294 & 1 & 4 & 1 & 3 \\ \cline{6-9} &&
&&& 1 & 8 & 1 & 7 \\ \cline{6-9} &&
&&& 1 & 14 & 2 & 12 \\ \cline{6-9} &&
&&& 1 & 19 & 4 & 15 \\ \cline{6-9} &&
&&& 1 & 24 & 5 & 19 \\ \cline{6-9} &&
&&& 1 & 27 & 6 & 21 \\ \cline{6-9} &&
&&& 2 & 30 & 7 & 23 \\ \cline{6-9} &&
&&& 19 & 31 & 7 & 24 \\  \hline
361 &8&13 & 3606 & 275 & 1 & 7 & 1 & 6 \\ \cline{6-9} &&
&&& 1 & 11 & 1 & 10 \\ \cline{6-9} &&
&&& 2 & 15 & 1 & 14 \\ \cline{6-9} &&
&&& 21 & 16 & 1 & 15 \\ \cline{6-9} &&
&&& 3 & 17 & 1 & 16 \\  \hline
362 &34&56 & 3601 & 0 & &&& \\ \hline
363 &4&6 & 1 & 15 & 15 & 2 & 0 & 2 \\ \hline
364 &14&24 & 3601 & 0 & &&& \\ \hline
365 &30&15 & 3601 & 0 & &&& \\ \hline
405 &8&7 & 2 & 23 & 1 & 1 & 0 & 1 \\ \cline{6-9} &&
&&& 5 & 2 & 0 & 2 \\ \cline{6-9} &&
&&& 7 & 5 & 0 & 5 \\ \cline{6-9} &&
&&& 7 & 8 & 0 & 8 \\ \cline{6-9} &&
&&& 1 & 10 & 0 & 10 \\ \hline
430 &27&52 & 3601 & 0 & &&& \\ \hline
431 & 27&52 & 3601 & 0 & &&& \\ \hline
447 & 13&24 & 3602 & 0 & &&& \\ \hline
475 & 23&34 &3603 & 0 & &&& \\ \hline
483 & 8&13 &3200 & 100 & 21 & 2 & 0 & 2 \\ \cline{6-9} &&
&&& 5 & 3 & 0 & 3 \\ \cline{6-9} &&
&&& 1 & 3 & 0 & 3 \\  \hline
530 & 10&18 & 3605 & 33180 & 2 & 1 & 0 & 1 \\ \cline{6-9} &&
&&& 1 & 2 & 0 & 2 \\ \cline{6-9} &&
&&& 6 & 9 & 0 & 9 \\ \cline{6-9} &&
&&& 1 & 12 & 2 & 10 \\ \cline{6-9} &&
&&& 9 & 14 & 2 & 12 \\ \cline{6-9} &&
&&& 7 & 15 & 2 & 13 \\  \hline
539 &6&13 &  3605 & 689 & 1 & 3 & 0 & 3 \\ \cline{6-9} &&
&&& 24 & 4 & 0 & 4 \\ \cline{6-9} &&
&&& 1 & 5 & 0 & 5 \\ \hline
552 &3&5 & 1 & 10 & &&& \\ \hline
553 & 3&5 &1 & 10 & &&& \\ \hline
609 &5&12 & 3606 & 277 & 24 & 1 & 0 & 1 \\ \hline
614 &1&4 &  1 & 5 & &&& \\ \hline
629 &5&8 & 1 & 10 & &&& \\ \hline
647 &11&17 & 3606 & 446 & 27 & 1 & 0 & 1 \\  \hline
651 &29&48 & 3601 & 0 & &&& \\ \hline
667 &103&91 & 3601 & 0 & &&& \\ \hline
676 &14&25 & 3601 & 0 & &&& \\ \hline
679 & 4&12 & 2894 & 47 & 1 & 1 & 0 & 1 \\ \cline{6-9} &&
&&& 28 & 2 & 0 & 2 \\  \hline
680 &4&12 & 3610 & 47 & 1 & 1 & 0 & 1 \\ \cline{6-9} &&
&&& 28 & 2 & 0 & 2 \\ \hline
687 & 15&17 & 3604 & 11776 & 26 & 44 & 0 & 44 \\  \hline
688 &16&21 & 3601 & 0 & &&& \\ \hline
707 &5&11 & 3627 & 341 & 6 & 1 & 0 & 1 \\ \cline{6-9} &&
&&& 9 & 2 & 0 & 2 \\ \cline{6-9} &&
&&& 8 & 3 & 0 & 3 \\ \cline{6-9} &&
&&& 1 & 4 & 0 & 4 \\ \hline
710 &7&18 & 3605 & 5891 & 2 & 1 & 0 & 1 \\ \cline{6-9} &&
&&& 1 & 3 & 0 & 3 \\ \cline{6-9} &&
&&& 23 & 5 & 0 & 5 \\ \cline{6-9} &&
&&& 1 & 6 & 0 & 6 \\ \hline
742 &2&7 &  1 & 35 & &&& \\ \hline
745 &5&14 & 3601 & 0 & &&& \\ \hline
747 &33&54 & 3601 & 0 & &&& \\ \hline
748 &4&9 & 3609 & 257 & 18 & 1 & 0 & 1 \\ \cline{6-9} &&
&&& 10 & 2 & 0 & 2 \\ \hline
755 &9&11 & 3609 & 72 & 1 & 2 & 0 & 2 \\ \cline{6-9} &&
&&& 1 & 3 & 0 & 3 \\ \cline{6-9} &&
&&& 18 & 5 & 0 & 5 \\ \cline{6-9} &&
&&& 8 & 6 & 0 & 6 \\ \cline{6-9} &&
&&& 1 & 7 & 0 & 7 \\ \hline
758 &2&5 & 1 & 9 & &&& \\ \hline
780 &4&16 & 3606 & 0 & &&& \\ \hline
781 &3&11 & 2 & 132 & &&& \\ \hline
782 &2&7 & 1 & 16 & &&& \\ \hline
783 & 3&9 &3614 & 133 & 20 & 1 & 0 & 1 \\ \cline{6-9} &&
&&& 1 & 3 & 0 & 3 \\ \cline{6-9} &&
&&& 2 & 5 & 0 & 5 \\ \cline{6-9} &&
&&& 1 & 6 & 0 & 6 \\ \cline{6-9} &&
&&& 1 & 7 & 0 & 7 \\ \hline
793 &2&5 & 1 & 12 & &&& \\ \hline
795 &2&7 & 1 & 45 & &&& \\ \hline
815 &2&7 & 1 & 1 & &&& \\ \hline
827 &10&18 & 3607 & 5504 & 1 & 28 & 3 & 25 \\ \cline{6-9} &&
&&& 1 & 36 & 3 & 33 \\ \cline{6-9} &&
&&& 1 & 56 & 4 & 52 \\ \cline{6-9} &&
&&& 1 & 66 & 5 & 61 \\ \cline{6-9} &&
&&& 1 & 73 & 5 & 68 \\ \cline{6-9} &&
&&& 1 & 75 & 5 & 70 \\ \cline{6-9} &&
&&& 1 & 76 & 6 & 70 \\ \cline{6-9} &&
&&& 1 & 77 & 7 & 70 \\ \cline{6-9} &&
&&& 1 & 78 & 7 & 71 \\ \cline{6-9} &&
&&& 2 & 79 & 7 & 72 \\ \cline{6-9} &&
&&& 12 & 81 & 7 & 74 \\ \cline{6-9} &&
&&& 4 & 84 & 8 & 76 \\ \cline{6-9} &&
&&& 1 & 85 & 8 & 77 \\ \hline
854 & 4&11 &1 & 13 & &&& \\ \hline
868 &5&10 & 1 & 50 & &&& \\ \hline
870 &7&17 & 3609 & 2648 & 2 & 1 & 1 & 0 \\ \cline{6-9} &&
&&& 20 & 3 & 3 & 0 \\ \cline{6-9} &&
&&& 4 & 4 & 3 & 1 \\  \hline
875 &4&9 & 1 & 47 & 4 & 1 & 0 & 1 \\ \cline{6-9} &&
&&& 13 & 4 & 0 & 4 \\ \cline{6-9} &&
&&& 2 & 7 & 0 & 7 \\ \hline
880 &4&15 & 3603 & 0 & &&& \\ \hline
882 &3&6 & 1 & 14 & &&& \\ \hline
886 &5&19 &  3603 & 0 & &&& \\ \hline
887 &4&11 & 3630 & 211 & 25 & 1 & 0 & 1 \\ \hline
894 &3&9 & 1 & 15 & 6 & 2 & 0 & 2 \\ \cline{6-9} &&
&&& 3 & 4 & 0 & 4 \\  \hline
905 &5&16 & 3606 & 2795 & 2 & 1 & 0 & 1 \\ \cline{6-9} &&
&&& 1 & 2 & 0 & 2 \\ \cline{6-9} &&
&&& 23 & 3 & 0 & 3 \\  \hline
906 &3&10 & 3609 & 135 & 28 & 1 & 0 & 1 \\ \hline
916 &5&7 &  17 & 24 & 1 & 1 & 0 & 1 \\ \cline{6-9} &&
&&& 13 & 2 & 0 & 2 \\ \cline{6-9} &&
&&& 3 & 4 & 0 & 4 \\ \cline{6-9} &&
&&& 1 & 6 & 0 & 6 \\ \cline{6-9} &&
&&& 8 & 8 & 0 & 8 \\  \hline
922 &3&10 & 2 & 46 & &&& \\ \hline
932 &4&8 & 7 & 36 & 20 & 2 & 0 & 2 \\ \hline
934 &16&24 & 3601 & 0 & &&& \\ \hline
940 &20&43 & 3603 & 0 & &&& \\ \hline
951 &34&72 & 3602 & 0 & &&& \\ \hline
957 &4&5 & 1 & 4 & &&& \\ \hline
968 &10&15 & 3618 & 356 & 1 & 2 & 0 & 2 \\ \cline{6-9} &&
&&& 26 & 6 & 1 & 5 \\  \hline
987 &9&12 & 3605 & 521 & 1 & 14 & 0 & 14 \\ \cline{6-9} &&
&&& 1 & 14 & 0 & 14 \\ \cline{6-9} &&
&&& 1 & 31 & 0 & 31 \\ \cline{6-9} &&
&&& 1 & 43 & 0 & 43 \\ \cline{6-9} &&
&&& 1 & 51 & 0 & 51 \\ \cline{6-9} &&
&&& 1 & 56 & 0 & 56 \\ \cline{6-9} &&
&&& 1 & 59 & 0 & 59 \\ \cline{6-9} &&
&&& 8 & 61 & 0 & 61 \\ \cline{6-9} &&
&&& 4 & 63 & 0 & 63 \\ \cline{6-9} &&
&&& 3 & 65 & 0 & 65 \\ \cline{6-9} &&
&&& 1 & 68 & 0 & 68 \\ \cline{6-9} &&
&&& 4 & 70 & 0 & 70 \\ \hline
1004 &10&21 & 3604 & 14905 & 1 & 7 & 0 & 7 \\ \cline{6-9} &&
&&& 1 & 15 & 0 & 15 \\ \cline{6-9} &&
&&& 1 & 18 & 0 & 18 \\ \cline{6-9} &&
&&& 1 & 19 & 0 & 19 \\ \cline{6-9} &&
&&& 4 & 22 & 0 & 22 \\ \cline{6-9} &&
&&& 19 & 24 & 0 & 24 \\  \hline
1021 &7&11 & 3605 & 201 & 1 & 4 & 0 & 4 \\ \cline{6-9} &&
&&& 1 & 5 & 0 & 5 \\ \cline{6-9} &&
&&& 12 & 6 & 0 & 6 \\ \cline{6-9} &&
&&& 2 & 8 & 0 & 8 \\ \cline{6-9} &&
&&& 5 & 10 & 0 & 10 \\ \cline{6-9} &&
&&& 6 & 12 & 0 & 12 \\  \hline
1024 &2&5 & 2 & 13 & &&& \\ \hline
1031 &3&5 & 1 & 16 & &&& \\ \hline
1035 &4&12 &  3601 & 0 & &&& \\ \hline
1037 &2&7 & 1 & 30 & 12 & 1 & 0 & 1 \\ \hline
1038 & 3&11 &3605 & 201 & 25 & 1 & 0 & 1 \\ \hline
1045 &3&3 & 1 & 3 & &&& \\ \hline
1053 &6&7 & 3106 & 28 & 1 & 2 & 0 & 2 \\ \cline{6-9} &&
&&& 1 & 8 & 0 & 8 \\ \cline{6-9} &&
&&& 8 & 11 & 0 & 11 \\ \cline{6-9} &&
&&& 3 & 13 & 0 & 13 \\ \cline{6-9} &&
&&& 1 & 16 & 0 & 16 \\ \cline{6-9} &&
&&& 1 & 19 & 0 & 19 \\ \cline{6-9} &&
&&& 1 & 22 & 0 & 22 \\ \cline{6-9} &&
&&& 1 & 30 & 0 & 30 \\ \cline{6-9} &&
&&& 1 & 33 & 0 & 33 \\ \cline{6-9} &&
&&& 1 & 36 & 0 & 36 \\ \cline{6-9} &&
&&& 1 & 44 & 0 & 44 \\ \cline{6-9} &&
&&& 3 & 46 & 0 & 46 \\ \cline{6-9} &&
&&& 1 & 50 & 0 & 50 \\ \cline{6-9} &&
&&& 2 & 55 & 0 & 55 \\ \cline{6-9} &&
&&& 1 & 59 & 0 & 59 \\ \cline{6-9} &&
&&& 1 & 64 & 0 & 64 \\ \hline
1054 &7&12 & 3610 & 1379 & 1 & 1 & 0 & 1 \\ \cline{6-9} &&
&&& 1 & 5 & 0 & 5 \\ \cline{6-9} &&
&&& 1 & 9 & 1 & 8 \\ \cline{6-9} &&
&&& 1 & 14 & 2 & 12 \\ \cline{6-9} &&
&&& 1 & 15 & 3 & 12 \\ \cline{6-9} &&
&&& 1 & 16 & 3 & 13 \\ \cline{6-9} &&
&&& 14 & 17 & 3 & 14 \\ \cline{6-9} &&
&&& 4 & 18 & 3 & 15 \\ \cline{6-9} &&
&&& 5 & 19 & 3 & 16 \\  \hline  
\end{longtable}
}

\clearpage
{\footnotesize
\begin{longtable}{|r|r|r|r|r|r|r|r|r|}
\caption{\textbf{$Q$-vector space} of conservation laws up to degree 5 of CRNs of the database Biomodels resulting from the \textbf{$Q$-vector space} of syzygies up to degree 4 in the generic branch of comprehensive Gröbner system. The columns represent the biomodel number, running time in seconds, the number of the generators of the syzygy \textbf{module} of the corresponding steady state ideal, dimension of the vector space of syzygies up to degree 4, 
the number of $Q$-linearly independent conservation laws obtained from those syzygies, (equivalently, the dimension of the $Q$-vector space of conservative syzygies up to degree 4), the number of $Q$-linearly independent linear conservation laws, and finally the number of $Q$-linealry independent polynomial conservation laws for
the biomodel, respectively. Blank rows show that the computations did
not finish in four hours. The major part of the computation time
in the algorithm is spent on computing a basis for the syzygy module and then computing a basis for the $Q$-vector space of syzygies up to degree 4, and for those CRNs
that the computations did not finish in four hours, in fact the syzygy $Q$-vector space basis computation did not finish in that time period (on a single core of an Intel i5-9400, 2.9 GHz processor).
\cor{$n$ and $r$ are the numbers of variables and parameters of the models, respectively. }
\label{tab:new-table1} }\\
    \hline
    Bio- & n & r& running  & \ \# syzygy  & dim. syzygy  & \# cons.\ laws & \# lin. & \# pol. \\
    Model & &  & time & module \ basis & vector space &  &  cons.\ laws & cons.\ laws  \\    
    \hline
2 & 13 & 37 & 14401 &  &  &  &  &  \\ \hline  
 28 & 16 & 31 & 14401 &  &  &  &  &  \\ \hline  
 30 & 18 & 36 & 14401 &  &  &  &  &  \\ \hline  
 69 & 10 & 25 & 14401 &  &  &  &  &  \\ \hline  
 72 & 7 & 11 & 14401 &  &  &  &  &  \\ \hline  
 80 & 10 & 17 & 14401 &  &  &  &  &  \\ \hline  
 82 & 10 & 17 & 14401 &  &  &  &  &  \\ \hline  
 85 & 17 & 42 & 14401 &  &  &  &  &  \\ \hline  
 86 & 17 & 56 & 14401 &  &  &  &  &  \\ \hline  
 92 & 4 & 7 & 32 & 3 & 155 & 20 & 2 & 18 \\ \hline  
 99 & 7 & 16 & 14401 &  &  &  &  &  \\ \hline  
 102 & 13 & 42 & 14408 &  &  &  &  &  \\ \hline  
 103 & 17 & 50 & 14401 &  &  &  &  &  \\ \hline  
 150 & 4 & 5 & 32 & 3 & 155 & 20 & 2 & 18 \\ \hline  
 156 & 3 & 8 & 1 & 3 & 36 & 0 & 0 & 0 \\ \hline  
 159 & 3 & 8 & 1 & 3 & 36 & 0 & 0 & 0 \\ \hline  
 205 & 194 & 351 & 14407 &  &  &  &  &  \\ \hline  
 243 & 23 & 19 & 14401 &  &  &  &  &  \\ \hline  
 252 & 4 & 10 & 202 & 7 & 97 & 0 & 0 & 0 \\ \hline  
 283 & 4 & 5 & 31 & 3 & 155 & 20 & 2 & 18 \\ \hline  
 335 & 34 & 54 & 14401 &  &  &  &  &  \\ \hline  
 357 & 9 & 15 & 14401 &  &  &  &  &  \\ \hline  
 359 & 9 & 19 & 14401 &  &  &  &  &  \\ \hline  
 360 & 9 & 19 & 14401 &  &  &  &  &  \\ \hline  
 361 & 8 & 13 & 14401 &  &  &  &  &  \\ \hline  
 362 & 34 & 56 & 14401 &  &  &  &  &  \\ \hline  
 363 & 4 & 6 & 12676 & 4 & 160 & 5 & 1 & 4 \\ \hline  
 365 & 30 & 15 & 14401 &  &  &  &  &  \\ \hline  
 430 & 27 & 52 & 14401 &  &  &  &  &  \\ \hline  
 431 & 27 & 52 & 14401 &  &  &  &  &  \\ \hline  
 447 & 13 & 24 & 14401 &  &  &  &  &  \\ \hline  
 475 & 23 & 34 & 14401 &  &  &  &  &  \\ \hline  
 483 & 8 & 13 & 14401 &  &  &  &  &  \\ \hline  
 539 & 6 & 13 & 14401 &  &  &  &  &  \\ \hline  
 552 & 3 & 5 & 1 & 1 & 6 & 0 & 0 & 0 \\ \hline  
 553 & 3 & 5 & 1 & 1 & 6 & 0 & 0 & 0 \\ \hline  
 609 & 5 & 12 & 1756 & 7 & 162 & 0 & 0 & 0 \\ \hline  
 614 & 1 & 4 & 1 & 0 & 0 & 0 & 0 & 0 \\ \hline  
 629 & 5 & 8 & 517 & 4 & 399 & 55 & 3 & 52 \\ \hline  
 647 & 11 & 17 & 14401 &  &  &  &  &  \\ \hline  
 651 & 29 & 48 & 14401 &  &  &  &  &  \\ \hline  
 687 & 15 & 17 & 14401 &  &  &  &  &  \\ \hline  
 688 & 16 & 21 & 14401 &  &  &  &  &  \\ \hline  
 707 & 5 & 11 & 14401 &  &  &  &  &  \\ \hline  
 710 & 7 & 18 & 14401 &  &  &  &  &  \\ \hline  
 742 & 2 & 7 & 1 & 1 & 6 & 0 & 0 & 0 \\ \hline  
 745 & 5 & 14 & 14401 &  &  &  &  &  \\ \hline  
 747 & 33 & 54 & 14401 &  &  &  &  &  \\ \hline  
 748 & 4 & 9 & 722 & 9 & 87 & 0 & 0 & 0 \\ \hline  
 755 & 9 & 11 & 14401 &  &  &  &  &  \\ \hline  
 758 & 2 & 5 & 1 & 1 & 6 & 0 & 0 & 0 \\ \hline  
 780 & 4 & 16 & 14401 &  &  &  &  &  \\ \hline  
 781 & 3 & 11 & 4 & 4 & 32 & 0 & 0 & 0 \\ \hline  
 782 & 2 & 7 & 1 & 1 & 6 & 0 & 0 & 0 \\ \hline  
 783 & 3 & 9 & 1 & 4 & 26 & 0 & 0 & 0 \\ \hline  
 793 & 2 & 5 & 1 & 1 & 6 & 0 & 0 & 0 \\ \hline  
 795 & 2 & 7 & 1 & 1 & 6 & 0 & 0 & 0 \\ \hline  
 815 & 2 & 7 & 1 & 1 & 6 & 0 & 0 & 0 \\ \hline  
 827 & 10 & 18 & 14401 &  &  &  &  &  \\ \hline  
 854 & 4 & 11 & 14401 &  &  &  &  &  \\ \hline  
 875 & 4 & 9 & 281 & 6 & 121 & 0 & 0 & 0 \\ \hline  
 880 & 4 & 15 & 14401 &  &  &  &  &  \\ \hline  
 882 & 3 & 6 & 7 & 2 & 26 & 0 & 0 & 0 \\ \hline  
 886 & 5 & 19 & 14401 &  &  &  &  &  \\ \hline  
 887 & 4 & 11 & 1428 & 8 & 107 & 0 & 0 & 0 \\ \hline  
 894 & 3 & 9 & 1 & 3 & 29 & 0 & 0 & 0 \\ \hline  
 905 & 5 & 16 & 14401 &  &  &  &  &  \\ \hline  
 906 & 3 & 10 & 22 & 3 & 22 & 0 & 0 & 0 \\ \hline  
 916 & 5 & 7 & 14401 &  &  &  &  &  \\ \hline  
 922 & 3 & 10 & 1 & 3 & 36 & 0 & 0 & 0 \\ \hline  
 932 & 4 & 8 & 7 & 6 & 91 & 0 & 0 & 0 \\ \hline  
 934 & 16 & 24 & 14401 &  &  &  &  &  \\ \hline  
 940 & 20 & 43 & 14401 &  &  &  &  &  \\ \hline  
 951 & 34 & 72 & 14401 &  &  &  &  &  \\ \hline  
 957 & 4 & 5 & 44 & 3 & 175 & 20 & 2 & 18 \\ \hline  
 987 & 9 & 12 & 14401 &  &  &  &  &  \\ \hline  
 1004 & 10 & 21 & 14401 &  &  &  &  &  \\ \hline  
 1021 & 7 & 11 & 14401 &  &  &  &  &  \\ \hline  
 1024 & 2 & 5 & 1 & 1 & 6 & 0 & 0 & 0 \\ \hline  
 1031 & 3 & 5 & 1 & 3 & 36 & 0 & 0 & 0 \\ \hline  
 1035 & 4 & 12 & 14401 &  &  &  &  &  \\ \hline  
 1037 & 2 & 7 & 1 & 1 & 6 & 0 & 0 & 0 \\ \hline  
 1038 & 3 & 11 & 3 & 4 & 29 & 0 & 0 & 0 \\ \hline  
 1045 & 3 & 3 & 2 & 2 & 55 & 5 & 1 & 4 \\ \hline  
 1053 & 6 & 7 & 13567 & 15 & 540 & 0 & 0 & 0 \\ \hline  
 1054 & 7 & 12 & 14405 &  &  &  &  &  \\ \hline    
\end{longtable}
}

\clearpage
{\footnotesize
\begin{longtable}{|r|r|r|r|r|r|r|r|r|}
\caption{$Q$-vector space of parametric conservation laws of CRNs of the database Biomodels resulting from computations on \textbf{$Q$-vector space} of syzygies up to degree 4 in each branch of a comprehensive Gröbner basis, using our Singular library \texttt{polconslaw.lib}.
  The columns represent Biomodel number, the running time in seconds,
  number of branches of the corresponding comprehensive Gr\"obner
  system computed within the running time (four hours
  on a single core of an Intel i5-9400, 2.9 GHz processor), number of the
  branch, the total number of $Q$-linearly independent conservation laws computed  from the \textbf{$Q$-vector space} of syzygies up to degree 4 in the
  branch, number of $Q$-linearly independent linear conservation laws computed, and finally,
  the number of $Q$-linearly independent polynomial conservation laws computed in the
  branch.
  Rows whose running time was above four hours and did not finish
  computing conservation laws within that branch, have blank entries
  under the columns of number of branch, number of sons. laws, and (linear/polynomial) conservation   laws.
For all the other branches, 
the branches leading to conservation laws are listed in the order they are found, with 
the corresponding number of conservation laws. Branches are missing because they were not
computed in four hours or because they do not lead to conservation laws (for instance 
branches 1-2 of Biomodel 99 lead to no conservation law). 
\cor{To simplify the table we regrouped the branches with the same numbers of conservation
laws and displayed only the numbers per subgroup. 
The gray rows
correspond to branches containing polynomial 
conservation laws algebraically independent
from  the linear ones.
$n$ and $r$ are the numbers of variables and parameters of the models, respectively. }
\label{tab:new-table2}}\\
    \hline
   Bio- & n & r & running  & \# branches & \# branches & \# cons.& \# lin.& \# pol. \\
    Model &  &  & time &  & per subgroup  &  laws &  cons. laws & cons. laws \\
    \hline
2  & 13 & 37& 14401  & None & &&& \\ \hline 
 28  & 16 & 31& 14401  & None & &&& \\ \hline 
 30  & 18 & 36& 14401  & None & &&& \\ \hline 
 69  & 10 & 25& 14401  & 12354 & &&& \\ \hline 
 72  & 7 & 11& 14401  & 74 & 1 & 14 & 2 & 12 \\ \cline{6-9} && 
 &&& 2 & 34 & 3 & 31 \\ \cline{6-9} && 
 &&& 2 & 69 & 4 & 65 \\ \cline{6-9} && 
 &&& 1 & 125 & 5 & 120 \\ \cline{6-9} && 
 &&& 2 & 209 & 6 & 203 \\ \hline 
 80  & 10 & 17& 14401  & 264 & &&& \\ \hline 
 82  & 10 & 17& 14401  & 344 & &&& \\ \hline 
 85  & 17 & 42& 14401  & None & &&& \\ \hline 
 86  & 17 & 56& 14401  & None & &&& \\ \hline 
 92  & 4 & 7& 87  & 10 & 6 & 14 & 2 & 12 \\ \cline{6-9} && 
 &&& 3 & 34 & 3 & 31 \\ \cline{6-9} && 
 &&& 1 & 69 & 4 & 65 \\ \hline 
 99  & 7 & 16& 14401  & 3783 & 2 & 4 & 1 & 3 \\ \cline{6-9} && 
 &&& 2 & 14 & 2 & 12 \\ \cline{6-9} && 
 &&& 2 & 34 & 3 & 31 \\ \cline{6-9} && 
 &&& 2 & 69 & 4 & 65 \\ \cline{6-9} && 
 &&& 2 & 125 & 5 & 120 \\ \cline{6-9} && 
 &&& 1 & 209 & 6 & 203 \\ \hline 
 102  & 13 & 42& 14401  & None & &&& \\ \hline 
 103  & 17 & 50& 14401  & None & &&& \\ \hline 
 150  & 4 & 5& 101  & 10 & 4 & 14 & 2 & 12 \\ \cline{6-9} && 
 &&& 5 & 34 & 3 & 31 \\ \cline{6-9} && 
 &&& 1 & 69 & 4 & 65 \\ \hline 
 156  & 3 & 8& 12  & 25 & \cellcolor{Gray} 1 &\cellcolor{Gray} 8 &\cellcolor{Gray} 1 & \cellcolor{Gray}7 \\ \cline{6-9} && 
 &&& 9 & 4 & 1 & 3 \\ \cline{6-9} && 
 &&& 5 & 14 & 2 & 12 \\ \cline{6-9} && 
 &&& 2 & 34 & 3 & 31 \\ \hline 
 159  & 3 & 8& 14  & 36 & 13 & 4 & 1 & 3 \\ \cline{6-9} && 
 &&& 7 & 14 & 2 & 12 \\ \cline{6-9} && 
 &&& 1 & 34 & 3 & 31 \\ \hline 
 205  & 194 & 351& 14401  & None & &&& \\ \hline 
 243  & 23 & 19& 14401  & 55296 & &&& \\ \hline 
 252  & 4 & 10& 3372  & 175 & 11 & 4 & 1 & 3 \\ \cline{6-9} && 
 &&& 13 & 14 & 2 & 12 \\ \cline{6-9} && 
 &&& 4 & 34 & 3 & 31 \\ \cline{6-9} && 
 &&& 1 & 69 & 4 & 65 \\ \hline 
 283  & 4 & 5& 16  & 2 & 1 & 14 & 2 & 12 \\ \cline{6-9} && 
 &&& 1 & 34 & 3 & 31 \\ \hline 
 335  & 34 & 54& 14401  & None & &&& \\ \hline 
 357  & 9 & 15& 14401  & 1953 & &&& \\ \hline 
 359  & 9 & 19& 14401  & 9741 & &&& \\ \hline 
 360  & 9 & 19& 14401  & 7294 & &&& \\ \hline 
 361  & 8 & 13& 14401  & 275 & 1 & 34 & 3 & 31 \\ \cline{6-9} && 
 &&& 2 & 69 & 4 & 65 \\ \hline 
 362  & 34 & 56& 14401  & None & &&& \\ \hline 
 363  & 4 & 6& 134  & 15 & 3 & 4 & 1 & 3 \\ \cline{6-9} && 
 &&& 6 & 14 & 2 & 12 \\ \cline{6-9} && 
 &&& 5 & 34 & 3 & 31 \\ \cline{6-9} && 
 &&& 1 & 69 & 4 & 65 \\ \hline 
 365  & 30 & 15& 14401  & None & &&& \\ \hline 
 430  & 27 & 52& 14401  & None & &&& \\ \hline 
 431  & 27 & 52& 14401  & None & &&& \\ \hline 
 447  & 13 & 24& 14401  & None & &&& \\ \hline 
 475  & 23 & 34& 14401  & None & &&& \\ \hline 
 483  & 8 & 13& 14401  & 100 & 2 & 69 & 4 & 65 \\ \cline{6-9} && 
 &&& 1 & 125 & 5 & 120 \\ \hline 
 539  & 6 & 13& 14401  & 689 & &&& \\ \hline 
 552  & 3 & 5& 1  & 10 & \cellcolor{Gray} 1 & \cellcolor{Gray} 2 & \cellcolor{Gray} 0 & \cellcolor{Gray} 2\\ \cline{6-9} && 
 &&&  3 & 4 & 1 & 3  \\ \cline{6-9} && 
 &&& 1 & 14 & 2 & 12 \\ \hline 
 553  & 3 & 5& 1  & 10 & \cellcolor{Gray} 1 & \cellcolor{Gray} 2 & \cellcolor{Gray} 0 & \cellcolor{Gray} 2 \\ \cline{6-9} && 
 &&&3 & 4 & 1 & 3 \\ \cline{6-9} && 
 &&& 1 & 14 & 2 & 12 \\ \hline 
 609  & 5 & 12& 5611  & 277 & 6 & 4 & 1 & 3 \\ \cline{6-9} && 
 &&& 8 & 14 & 2 & 12 \\ \cline{6-9} && 
 &&& 9 & 34 & 3 & 31 \\ \cline{6-9} && 
 &&& 4 & 69 & 4 & 65 \\ \cline{6-9} && 
 &&& 1 & 125 & 5 & 120 \\ \hline 
 614  & 1 & 4& 1  & 5 & 2 & 4 & 1 & 3 \\ \hline 
 629  & 5 & 8& 872  & 10 & 5 & 34 & 3 & 31 \\ \cline{6-9} && 
 &&& 4 & 69 & 4 & 65 \\ \cline{6-9} && 
 &&& 1 & 125 & 5 & 120 \\ \hline 
 647  & 11 & 17& 14401  & 446 & &&& \\ \hline 
 651  & 29 & 48& 14401  & None & &&& \\ \hline 
 687  & 15 & 17& 14406  & 11775 & &&& \\ \hline 
 688  & 16 & 21& 14401  & None & &&& \\ \hline 
 707  & 5 & 11& 14401  & 341 & &&& \\ \hline 
 710  & 7 & 18& 14401  & 5891 & 2 & 4 & 1 & 3 \\ \cline{6-9} && 
 &&& 2 & 14 & 2 & 12 \\ \cline{6-9} && 
 &&& 3 & 34 & 3 & 31 \\ \cline{6-9} && 
 &&& 1 & 69 & 4 & 65 \\ \cline{6-9} && 
 &&& 2 & 125 & 5 & 120 \\ \hline 
 742  & 2 & 7& 1  & 35 & 7 & 4 & 1 & 3 \\ \cline{6-9} && 
 &&& 1 & 14 & 2 & 12 \\ \hline 
 745  & 5 & 14& 14401  & None & &&& \\ \hline 
 747  & 33 & 54& 14417  & None & &&& \\ \hline 
 748  & 4 & 9& 5471& 257 & 14 & 4 & 1 & 3 \\ \cline{6-9} && 
 &&& \cellcolor{Gray} 1 &\cellcolor{Gray} 8 &\cellcolor{Gray} 1 &\cellcolor{Gray} 7 \\ \cline{6-9} && 
 &&& 9 & 14 & 2 & 12 \\ \cline{6-9} && 
 &&& \cellcolor{Gray}1 & \cellcolor{Gray}21 & \cellcolor{Gray}2 & \cellcolor{Gray}19 \\ \cline{6-9} && 
 &&& 3 & 34 & 3 & 31 \\ \hline 
 755  & 9 & 11& 14401  & 72 & 1 & 34 & 3 & 31 \\ \cline{6-9} && 
 &&& 1 & 69 & 4 & 65 \\ \hline 
 758  & 2 & 5& 1  & 9 & \cellcolor{Gray} 1 & \cellcolor{Gray}2 &\cellcolor{Gray} 0 & \cellcolor{Gray}2 \\ \cline{6-9} && 
 &&& 3 & 4 & 1 & 3 \\ \cline{6-9} && 
 &&& 1 & 14 & 2 & 12 \\ \hline 
 780  & 4 & 16& 14401  & None & &&& \\ \hline 
 781  & 3 & 11& 4031  & 132 & 16 & 4 & 1 & 3 \\ \cline{6-9} && 
 &&& 9 & 14 & 2 & 12 \\ \cline{6-9} && 
 &&& 2 & 34 & 3 & 31 \\ \hline 
 782  & 2 & 7& 1  & 16 & 1 & 4 & 1 & 3 \\ \cline{6-9} && 
 &&& 1 & 14 & 2 & 12 \\ \hline 
 783  & 3 & 9& 9903  & 133 & 25 & 4 & 1 & 3 \\ \cline{6-9} && 
 &&& 3 & 14 & 2 & 12 \\ \hline 
 793  & 2 & 5& 1  & 12 & 3 & 4 & 1 & 3 \\ \cline{6-9} && 
 &&& 1 & 14 & 2 & 12 \\ \hline 
 795  & 2 & 7& 1  & 45 &  5 & 4 & 1 & 3  \\ \cline{6-9} && 
 &&& 1 & 14 & 2 & 12  \\ \hline 
 815  & 2 & 7& 1  & 1 & &&& \\ \hline 
 827  & 10 & 18& 14401  & 5504 & 1 & 34 & 3 & 31\\ \hline 
 854  & 4 & 11& 2444  & 13 & 8 & 4 & 1 & 3 \\ \cline{6-9} && 
 &&& 4 & 14 & 2 & 12 \\ \cline{6-9} && 
 &&& 1 & 69 & 4 & 65 \\ \hline 
 875  & 4 & 9& 217  & 47 & 7 & 4 & 1 & 3 \\ \cline{6-9} && 
 &&& 9 & 14 & 2 & 12 \\ \cline{6-9} && 
 &&& 2 & 21 & 2 & 19 \\ \cline{6-9} && 
 &&& 5 & 34 & 3 & 31 \\ \cline{6-9} && 
 &&& 1 & 69 & 4 & 65 \\ \hline 
 880  & 4 & 15& 14401  & None & &&& \\ \hline 
 882  & 3 & 6& 8  & 14 & 2 & 4 & 1 & 3 \\ \cline{6-9} && 
 &&& 4 & 14 & 2 & 12 \\ \cline{6-9} && 
 &&& 1 & 34 & 3 & 31 \\ \hline 
 886  & 5 & 19& 14401  & None & &&& \\ \hline 
 887  & 4 & 11& 3209 & 211 & 15 & 4 & 1 & 3 \\ \cline{6-9} && 
 &&&\cellcolor{Gray} 1 &\cellcolor{Gray} 7 &\cellcolor{Gray} 1 & \cellcolor{Gray}6 \\ \cline{6-9} && 
 &&&\cellcolor{Gray} 1 &\cellcolor{Gray} 8 &\cellcolor{Gray} 1 & \cellcolor{Gray}7 \\ \cline{6-9} && 
 &&& 8 & 14 & 2 & 12 \\ \cline{6-9} && 
 &&&\cellcolor{Gray} 1 & \cellcolor{Gray}21 &\cellcolor{Gray} 2 & \cellcolor{Gray}19 \\ \cline{6-9} && 
 &&& 1 & 34 & 3 & 31 \\ \cline{6-9} && 
 &&& 1 & 69 & 4 & 65 \\ \hline 
 894  & 3 & 9& 2  & 15 & &&& \\ \hline 
 905  & 5 & 16& 14437  & 2795 & 2 & 4 & 1 & 3 \\ \cline{6-9} && 
 &&& 3 & 14 & 2 & 12 \\ \cline{6-9} && 
 &&& 16 & 34 & 3 & 31 \\ \cline{6-9} && 
 &&& 6 & 69 & 4 & 65 \\ \cline{6-9} && 
 &&& 1 & 125 & 5 & 120 \\ \hline 
 906  & 3 & 10& 6456  & 135 & 20 & 4 & 1 & 3 \\ \cline{6-9} && 
 &&& 7 & 14 & 2 & 12 \\ \cline{6-9} && 
 &&& 1 & 34 & 3 & 31 \\ \hline 
 916  & 5 & 7& 14401  & 24 & &&& \\ \hline 
 922  & 3 & 10& 14401  & 46 & 2 & 4 & 1 & 3 \\ \hline 
 932  & 4 & 8& 9805  & 36 &\cellcolor{Gray} 1 & \cellcolor{Gray}8 & \cellcolor{Gray}1 & \cellcolor{Gray}7 \\ \cline{6-9} && 
 &&& 10 & 4 & 1 & 3 \\ \cline{6-9} && 
 &&& 8 & 14 & 2 & 12 \\ \cline{6-9} && 
 &&& 6 & 34 & 3 & 31 \\ \cline{6-9} && 
 &&& 2 & 69 & 4 & 65 \\ \hline 
 934  & 16 & 24& 14401  & None & &&& \\ \hline 
 940  & 20 & 43& 14401  & None & &&& \\ \hline 
 951  & 34 & 72& 14414  & None & &&& \\ \hline 
 957  & 4 & 5& 43  & 4 & 2 & 14 & 2 & 12 \\ \cline{6-9} && 
 &&& 1 & 34 & 3 & 31 \\ \cline{6-9} && 
 &&& 1 & 69 & 4 & 65 \\ \hline 
 987  & 9 & 12& 14401  & 521 & &&& \\ \hline 
 1004  & 10 & 21& 14401  & 14905 & &&& \\ \hline 
 1021  & 7 & 11& 14408  & 201 & 11 & 4 & 1 & 3 \\ \cline{6-9} && 
 &&& 2 & 14 & 2 & 12 \\ \cline{6-9} && 
 &&& 5 & 34 & 3 & 31 \\ \cline{6-9} && 
 &&& 2 & 45 & 3 & 42 \\ \cline{6-9} && 
 &&& 8 & 69 & 4 & 65 \\ \cline{6-9} && 
 &&& 4 & 125 & 5 & 120 \\ \cline{6-9} && 
 &&& 1 & 209 & 6 & 203 \\ \hline 
 1024  & 2 & 5& 1  & 13 & 1 & 4 & 1 & 3 \\ \cline{6-9} && 
 &&& 1 & 14 & 2 & 12 \\ \hline 
 1031  & 3 & 5& 4  & 16 & 6 & 4 & 1 & 3 \\ \cline{6-9} && 
 &&& 2 & 14 & 2 & 12 \\ \hline 
 1035  & 4 & 12& 14406  & None & &&& \\ \hline 
 1037  & 2 & 7& 1  & 30 & 7 & 4 & 1 & 3 \\ \cline{6-9} && 
 &&& 1 & 14 & 2 & 12 \\ \hline 
 1038  & 3 & 11& 6537  & 201 & 19 & 4 & 1 & 3 \\ \cline{6-9} && 
 &&& 8 & 14 & 2 & 12 \\ \cline{6-9} && 
 &&& 1 & 34 & 3 & 31 \\ \hline 
 1045  & 3 & 3& 4  & 3 & 3 & 4 & 1 & 3 \\ \hline 
 1053  & 6 & 7& 14401  & 28 & 5 & 4 & 1 & 3 \\ \cline{6-9} && 
 &&&\cellcolor{Gray} 1 & \cellcolor{Gray}21 & \cellcolor{Gray}2 & \cellcolor{Gray}19 \\ \cline{6-9} && 
 &&& \cellcolor{Gray} 1 & \cellcolor{Gray}17 & \cellcolor{Gray}2 & \cellcolor{Gray}15 \\ \cline{6-9} && 
 &&& 1 & 34 & 3 & 31 \\ \cline{6-9} && 
 &&& \cellcolor{Gray} 1 & \cellcolor{Gray}38 &\cellcolor{Gray} 3 & \cellcolor{Gray}35 \\ \cline{6-9} && 
 &&& 3 & 69 & 4 & 65 \\ \cline{6-9} && 
 &&& 1 & 125 & 5 & 120 \\ \cline{6-9} && 
 &&& 2 & 209 & 6 & 203 \\ \hline 
 1054  & 7 & 12& 14401  & 1379 &  &  & &  \\ \hline 
\end{longtable}
}

\end{document}